\documentclass[10pt,reqno]{amsart}
\usepackage[margin=1in]{geometry}
\makeatletter
\def\section{\@startsection{section}{1}%
	\z@{.7\linespacing\@plus\linespacing}{.5\linespacing}%
	{\bfseries
		\centering
}}
\def\@secnumfont{\bfseries}
\makeatother
\usepackage{amsmath,amssymb,amsthm,graphicx,amsxtra, setspace}
\usepackage[utf8]{inputenc}
\usepackage{charter}
\usepackage{xcolor}
\usepackage{mathrsfs}
\usepackage{alltt}
\usepackage{relsize}
\usepackage{hyperref}
\usepackage{aliascnt}
\usepackage{tikz}
\usepackage{mathtools}
\usepackage{multicol}
\usepackage{upgreek}
\allowdisplaybreaks
\usepackage{scalerel,stackengine}
\stackMath
\newcommand\reallywidehat[1]{%
	\savestack{\tmpbox}{\stretchto{%
			\scaleto{%
				\scalerel*[\widthof{\ensuremath{#1}}]{\kern-.6pt\bigwedge\kern-.6pt}%
				{\rule[-\textheight/2]{1ex}{\textheight}}
			}{\textheight}%
		}{0.5ex}}%
	\stackon[1pt]{#1}{\tmpbox}%
}
\definecolor{deepblue}{rgb}{0.0,0.0,0.7}
\numberwithin{equation}{section}

\newtheorem{theorem}{Theorem}[section]

\newaliascnt{lemma}{theorem}
\newtheorem{lemma}[lemma]{Lemma}
\aliascntresetthe{lemma} 

\newaliascnt{proposition}{theorem}
\newtheorem{proposition}[proposition]{Proposition}
\aliascntresetthe{proposition}

\newaliascnt{assumption}{theorem}
\newtheorem{assumption}[assumption]{Assumption}
\aliascntresetthe{assumption}

\newaliascnt{corollary}{theorem}

\aliascntresetthe{corollary}

\newaliascnt{definition}{theorem}
\newtheorem{definition}[definition]{Definition}
\aliascntresetthe{definition}

\newaliascnt{example}{theorem}
\newtheorem{example}[example]{Example}
\aliascntresetthe{example}

\newaliascnt{remark}{theorem}

\aliascntresetthe{remark}

\newaliascnt{hypothesis}{theorem}

\aliascntresetthe{hypothesis}

\newaliascnt{property}{theorem}

\aliascntresetthe{property}

\let\originalleft\left
\let\originalright\right
\renewcommand{\left}{\mathopen{}\mathclose\bgroup\originalleft}
\renewcommand{\right}{\aftergroup\egroup\originalright}

\makeatletter
\newcommand{\doublewidetilde}[1]{{%
		\mathpalette\double@widetilde{#1}%
}}
\newcommand{\double@widetilde}[2]{%
	\sbox\z@{$\m@th#1\widetilde{#2}$}%
	\ht\z@=.9\ht\z@
	\widetilde{\box\z@}%
}
\makeatother

\def\w{\textbf{W}^{\varepsilon}_{{\theta}^{\varepsilon}}}

\def\L{\mathrm{L}}

\def\F{\mathrm{F}}
\def\C{\mathrm{C}}

\def\B{\mathrm{B}}

\def\z{\mathbf{z}}
\def\v{\mathbf{v}}
\def\V{\mathbb{V}}
\def\w{\mathbf{w}}

\def\G{\mathbb{G}}

\def\no{\nonumber}
\def\V{\mathbb{V}}

\def\U{\mathrm{U}}

\def\u{\mathbf{u}}
\def\H{\mathbb{H}}
\def\n{\mathbf{n}}

\def\q{\mathbf{q}}

\newcommand{\R}{\mathbb{R}}

\newcommand{\Addresses}{{
		\footnote{
			
			\noindent \textsuperscript{1}School of Mathematics, Indian Institute of Science Education and Research, Trivandrum (IISER-TVM),
			Maruthamala PO, Vithura, Thiruvananthapuram, Kerala, 695 551, INDIA  \par\nopagebreak \noindent
			\textit{e-mail:} \texttt{sheetal@iisertvm.ac.in}, *Corresponding Author

			\noindent \textsuperscript{2}Institute of Mathematics, Czech Academy of Sciences, Zitna 25, 11567 Praha 1, Czech Republic.  
			\par\nopagebreak \noindent
			\textit{e-mail:} \texttt{perisetti@math.cas.cz, plnmn9@gmail.com}

					{\bf Funding:}  There is no specific funding received from any grant for this work. Both authors thank IISER Thiruvananthapuram for providing facilities and a stimulating atmosphere for conducting the research. \\

}}}

\begin{document}
	\title{Viscosity solutions of Hamilton-Jacobi-Bellman equations for optimal control problem of local Cahn-Hilliard-Navier-Stokes System
		\Addresses	}
	
	
	\author[S. Dharmatti and M. Perisetti]
	{Sheetal Dharmatti\textsuperscript{1*} and Mahendranath Perisetti\textsuperscript{2}}

	\begin{abstract}
		In this work, we consider the local Cahn-Hilliard-Navier-Stokes equation with regular potential in a two-dimensional bounded domain.  We formulate a distributed optimal control problem as the minimization of a suitable cost functional subject to the controlled local Cahn-Hilliard-Navier-Stokes system and define the associated value function. We prove the Dynamic Programming Principle satisfied by the value function. Due to the lack of smoothness properties for the value function, we use the method of viscosity solutions to obtain the corresponding solution of the infinite-dimensional Hamilton-Jacobi-Bellman equation. We show that the value function is the unique viscosity solution of the Hamilton-Jacobi-Bellman equation.  The uniqueness of the viscosity solution is established via the comparison principle.
	\end{abstract}
	\keywords{Cahn-Hilliard equation, Navier-Stokes equation, Optimal control, Dynamic programming principle, Hamilton-Jacobi equation, Viscosity solution, Comparison principle}
	\subjclass{35F21, 35K55, 49J20, 49L20, 49L25, 76D05, 76T99, 93C20}
	\maketitle
	\section{Introduction}
	The famous Navier-Stokes equations govern the complex motions of a single-phase fluid and are studied in the literature extensively by physicists, engineers, and mathematicians. For the comprehensive mathematical study of these equations one can refer to \cite{MR776345, MR1855030, MR0304889, MR0609732} and references therein. The mathematical study of binary or multi-phase mixture flows has garnered interest in the last few decades. 
	J.W. Cahn and J.E. Hilliard were the first to formulate the mathematical equations of this problem and studied the spinodal decomposition of binary alloys (see \cite{CAHN1961795, cahnhilliard1958}). Similar phenomena occur in the phase separation of binary fluids, that is, fluids composed of either two phases of the same chemical species or phases of different compositions. In this case, however, the phenomenology is much more complicated because of the interplay between the phase separation stage and the fluid dynamics. The mathematical analysis of these phenomena is far from being well understood. Different phase field models can be developed by coupling Cahn-Hilliard equations with equations describing the dynamics of the flow. Thus the equations are not just non-linear but are also coupled and hence the mathematical study is challenging as well as difficult. 
	
	For the coupled Cahn-Hilliard-Navier-Stokes system, (CHNS system), the chemical interactions between two phases at the interface are governed by the Cahn-Hilliard system, and the Navier-Stokes equations with surface tension terms acting at the interface give the hydrodynamic properties of the mixture. When the two fluids have the same constant density, the temperature differences are negligible and the diffusive interface between the two phases has a small but non-zero thickness, a well-known model is the so-called “model H” (see \cite{hohenberg1977theory}).
	The coupled Cahn-Hilliard-Navier-Stokes system (model H) is described as follows: 
	\begin{align} \label{DPPsys1}
		\varphi_t + \mathbf{u} \cdot \nabla \varphi &= m \Delta \mu,  \quad \mathrm{in} \quad \Omega \times (0,T), \\
		\mu &= - \Delta \varphi + f(\varphi), \quad \mathrm{in}\quad \Omega \times (0,T),\\
		\mathbf{u}_t -\nu \Delta \mathbf{u} + (\mathbf{u} \cdot \nabla ) \mathbf{u} + \nabla \pi & = \mathcal{K}\mu \nabla \varphi + \U \quad \mathrm{in}\quad \Omega \times (0,T),\\
		\mathrm{div}(\mathbf{u}) & = 0, \quad \mathrm{in}\quad \Omega \times (0,T),\\
		\u(0) =\u_0, &\ \varphi(0) = \varphi_0,\quad \mathrm{in}\quad \Omega,\label{DPPsys2}
	\end{align}
	where $ \Omega \subset \mathbb{R}^2,$ is a bounded subset with smooth boundary $ \partial \Omega $, $ \U $ is an external volume forcing and we have assumed the density to be equal to one. Here, $ \u $ represents the mean velocity field and $ \varphi $ is the order parameter which represents the relative concentration of one of the fluids. The quantities $ \nu,m $ and $ \mathcal{K} $  are viscosity, mobility, and capillary coefficient respectively, which are positive constants.
	For $ \u $, we assume the Dirichlet (no-slip) boundary condition 
	\begin{align}
		\u = 0, \quad \mathrm{on} \ \partial \Omega \times (0,T),
	\end{align}
	and we assume that the boundary conditions for $ \varphi $ are the natural no-flux condition 
	\begin{align}
		\frac{\partial \varphi}{\partial \n} = \frac{\partial \Delta \varphi}{\partial \n}=0, \quad \mathrm{on} \ \partial \Omega \times (0,T), \label{DP21}
	\end{align}
	where $ \n $ is the outward normal to $ \partial \Omega $. 
	Note that \eqref{DP21} implies that
	\begin{align}
		\frac{\partial \mu}{\partial \n} =0, \quad \mathrm{on} \ \partial \Omega \times (0,T), \label{DP22}
	\end{align}
	where $ \mu $ is the chemical potential of the binary mixture. It is given by the first variation of the following Landau-Ginzburg energy functional
	\begin{align} \label{DP82}
		E(\varphi) = \int_\Omega \left( \frac{1}{2}|\nabla \varphi (x)|^2  + F(\varphi (x)) \right) dx,
	\end{align}
	where, $ F(s) = \int_0^s f(\tau) d\tau $ is a suitable double-well potential.   A typical example of potential $ F $ is a logarithmic potential. However, this potential is very often replaced by a polynomial approximation of the regular potential (eg: $ F(s) = s^2(s^2-1) $).
	From \eqref{DPPsys1} and \eqref{DP22}, we deduce the conservation of the average of $\varphi$ denoted by
	\begin{align*}
		\overline{\varphi (t)}  = \frac{1}{|\Omega|} \int_\Omega \varphi(x,t) dx,
	\end{align*}  
	where $ |\Omega| $ is the Lebesgue measure of $ \Omega $. More precisely, we have 
	\begin{align*}
		\overline{\varphi (t)}  = \overline{\varphi (0)} , \ \forall t \geq 0.
	\end{align*}
	A different form of free energy has been proposed in \cite{MR1453735, MR1638739} and rigorously justified as a macroscopic limit of microscopic phase segregation models with particle conserving dynamics. In this case, the gradient term in \eqref{DP82} is replaced by a non-local spatial operator 
	\begin{align*}
		\mathcal{E}(\varphi) = \frac{1}{4} \int_{\Omega} \int_{\Omega} J(x-y)(\varphi(x)-\varphi(y))^2dxdy + \int_{\Omega} F(\varphi(x))dx,
	\end{align*}
	where $ J: \mathbb{R}^n \rightarrow \mathbb{R} $ is a smooth function such that $ J(x)=J(-x) $. The system with the chemical potential which is the first variation of above $ \mathcal{E} $; is called a nonlocal Cahn-Hilliard-Navier-Stokes system and for such a system $\mu$ is given by:
	\begin{align*}
		\mu = a \varphi -J*\varphi + f(\varphi).
	\end{align*} 
	The well-posedness of nonlocal CHNS system has been well studied in the literature by several authors (see \cite{MR3090070, MR3518604, MR2834896, MR3903266, MR3019479, MR4374591}). For optimal distributed control problems for the same see \cite{MR3456388, MR4104524, MR4108622}.
	For the local CHNS system \eqref{DPPsys1}-\eqref{DP22}, the existence and uniqueness of a weak solution and the existence of the strong solution have been obtained in \cite{MR1700669} in the case of regular potential. In the same work author also studies the case of singular potential, for the existence of a weak solution by approximating the singular potential with a sequence of regular potentials and passing to the limit of the corresponding solutions. Certain stability results have also been established. In \cite{MR2580516}, authors analyse the asymptotic behaviour of the solution of the local Cahn-Hilliard-Navier-Stokes system. In fact, they have proved the existence of global and exponential attractors.
	
	
	Optimal control theory of fluid dynamics models has been an important research area of applied mathematics with many applications across the fields. In \cite{MR4190792}, authors have considered the distributed optimal control problem as the minimization of the total energy and dissipation of energy of the flow with local Cahn-Hilliard-Navier-Stokes system where the controls appear in the form of volume force densities.  The existence of an optimal control as well as the first order necessary optimality conditions are established,
	and the optimal control is characterised in terms of adjoint variables. The optimal control problems with state constraint and robust control for the same system are investigated in \cite{MR3436705, MR3565933}, respectively.  
	Optimal control problems of semi-discrete Cahn-Hilliard-Navier-Stokes system for various cases like distributed and boundary control, with non-smooth Landau-Ginzburg energies and with non-matched fluid densities are studied in  \cite{MR3729372, MR3170501, MR3843444}. These works considered the local Cahn-Hilliard-Navier-Stokes equations for their numerical studies. All these works consider the optimal control problem using Pontryagin's maximum principle. The dynamic programming principle approach is completely open to such problems. In this work, our main aim is to study the dynamic programming principle for a control problem governed by a local CHNS system and derive the corresponding Hamilton-Jacobi-Bellman equation satisfied by the value function. Further, we want to show that the value function is the unique viscosity solution of the corresponding Hamilton-Jacobi-Bellman equation.
	

	\par 
	The viscosity solution method, a notion of generalised solution, well suited for the first order fully non-linear partial differential equations typically of the Hamilton-Jacobi type was first introduced by Crandall and Lions in \cite{MR690039}. They have further studied Hamilton-Jacobi equations in infinite dimensions using viscosity solution in a series of eight papers during 1984-1992 (see \cite{MR794776, MR826434, MR852660, MR1052335, MR1111190, MR1254890, MR1297016, MR0690039}). The value function of an optimal control problem subjected to a differential equation satisfies the Hamilton-Jacobi-Bellman (HJB) equation whenever the value function is smooth. However, seldom these value functions are $C^1$. The viscosity solution theory helps to tackle this issue by showing that the value function is the unique viscosity solution of the corresponding HJB equations. For a comprehensive treatment of these ideas, one can look at \cite{MR1484411} and references therein. Various works in these directions can be found in \cite{MR1160081, MR2690118}. Most of these works have assumptions on the operators involved in the constrained PDE to be linear or have some special properties like being m-accretive. Various works in these directions can be found in \cite{MR1160081, MR2690118}.   In \cite{MR1145836, MR1280618}, authors have generalised the ideas of viscosity solution theory to infinite dimensional problems with unbounded nonlinear terms. This allows us to treat value functions corresponding to the control problems constrained by nonlinear partial differential equations using viscosity solution theory. 
	
	The optimal control problems governed by non-linear PDEs and coupled non-linear equations are well studied using Pontryagin's maximum principle approach however not much work is done using viscosity solution theory. First-order Hamilton-Jacobi-Bellman (HJB) equations associated with the control of deterministic Navier-Stokes equations were considered in \cite{MR1918929, MR1102218} and the existence and uniqueness of viscosity solution were proven. The approach presented in \cite{MR1918929} can be extended to second-order HJB equations of the same type, for example, see \cite{MR2141895}, where authors study the infinite-dimensional second-order HJB equations that arise in the problem of optimal control of stochastic Navier-Stokes equations. The equations considered in \cite{MR2141895} are more general than those investigated in \cite{MR1765670}. Authors have also introduced a different definition of viscosity solution compared to the one used in \cite{MR1918929} which allows one to handle cost functionals that are very singular. Moreover, it also applies to the first-order HJB equations investigated in \cite{MR1918929} and thus can lead to an alternative theory of these equations that covers some cases not treated by the techniques of \cite{MR1918929}. In \cite{MR1753181}, authors have studied a class of HJB equations associated with stochastic optimal control of the Duncan-Mortensan-Zakai equation. For the related results on second-order HJB equations we refer \cite{MR1092739, MR3674558, MR1741146, MR1326907, MR1919513, MR2307066, MR1214874, MR971797, MR934730, MR2690118, MR1301180}. 
	
	In the current work, we propose the use of the viscosity solution technique to study the optimal control problems governed by the local CHNS system. As per our knowledge, this is the first attempt to study the optimal control problem of the coupled nonlinear system using the viscosity solution method; which proves the existence and uniqueness of the solution of the corresponding HJB equations satisfied by the value function.
	%
	%
	%
	
	The structure of the paper is as follows. In the next section, we describe the mathematical setting to study the local Cahn-Hilliard-Navier-Stokes equation. We recall some existence and uniqueness results available in the literature. We also present some convergence results of the solution with respect to the initial data. In section \ref{sec DPP}, we define the value function and derive its continuity properties. We also state the Dynamic Programming Principle (DPP) (see Theorem \ref{DPP}) and prove that the value function satisfies the DPP. In Section 4 we state the Hamilton-Jacobi- Bellman (HJB) equations satisfied by the value function and prove that the value function satisfies the HJB equation in the viscosity sense (see Theorem \ref{HJBThm}). In the last section, we show that the value function is the unique viscosity solution of the corresponding HJB equation via comparison principle (see Theorem \ref{comp_ppl}). 
	\section{Mathematical setting}
	In this section, we introduce the necessary function spaces needed throughout the paper.  We define some operators to write \eqref{DPPsys1}-\eqref{DPPsys2} in the abstract form. We also state the existence, uniqueness, and strong solution results of the system \eqref{DPPsys1}-\eqref{DPPsys2}.
	
	Hereafter, we assume that the domain $ \Omega $ is a bounded subset of $\mathbb{R}^2$ with a smooth boundary $ \partial \Omega $. We denote by $\mathbb{W}^{m,p}(\Omega; \mathbb{R}^2) = [W^{m,p}(\Omega; \mathbb{R})]^2$ the Sobolev space of order $m \in [0, \infty)$ and $p \geq 1$ of functions with values in $\mathbb{R}^2$. The norm of $\u \in W^{m,p}(\Omega; \mathbb{R}^2)$ will be denoted by $\|u\|_{m,p}$. We denote $L^p= W^{0,p}$. Moreover, we will write $\H^m = \mathbb{W}^{m,2}(\Omega; \mathbb{R}^2)$  and $\H =W^{0,2}(\Omega; \mathbb{R}^2) $ for vector valued functions, and $H^m = W^{m,2}(\Omega; \mathbb{R}), H=W^{0,2}(\Omega; \mathbb{R})$ for scalar valued functions.  
	
	Let us set 
	\begin{align*}
		\mathcal{D} = \left\{ \u \in C_c^\infty (\Omega) \ | \ \mathrm{div}(\u)=0 \right\}.
	\end{align*}
	Then we define
	\begin{align*}
		\G_{\mathrm{div}} = \mathrm{closure \ of }\  \mathcal{D} \mathrm{ \ in \ } \L^2(\Omega),\\
		\V_{\mathrm{div}} = \mathrm{closure \ of }\  \mathcal{D} \mathrm{ \ in \ } \H^1_0(\Omega).
	\end{align*}
	We denote by $ \| \cdot \| $ and $ (\cdot, \cdot) $, the norm and the scalar product, respectively, on both $ H $ and $ \G_{{\mathrm{div}}} $.
	We define the operator $A$ by 
	\begin{align*}
		A \mathbf{u}= - \mathcal{P}\Delta \mathbf{u} , \forall \u \in D(A) = \H^2 \cap \V_{\mathrm{div}},
	\end{align*}
	where $\mathcal{P}$ is the Leray-Helmholtz projector or the Stokes operator in $H$ onto $\G_{\mathrm{div}}$. We know that $ A $ is self-adjoint and positive definite, $ A^{-1} $ is compact and $ A $ generates an analytic semigroup. For $ \alpha \geq 0 $ we denote by $\V_{\alpha}$ the domain of $A^{\frac{\alpha}{2}}$, $D(A^{\frac{\alpha}{2}})$, equipped with the norm 
	\begin{align}
		\| \u \|_{\alpha} = \| A^{\frac{\alpha}{2}}\u\|.
	\end{align}
We introduce the non-negative linear unbounded operator on $L^2(\Omega)$ 
\begin{align}
A_N \varphi = -\Delta \varphi , \quad \forall \varphi \in D(A_N) = \{ \varphi \in H^2(\Omega) : \ \frac{\partial \varphi}{\partial \n}=0 \ \mathrm{on} \ \partial \Omega  \}
\end{align}
and we endow $D(A_N)$ with the norm $\|A_N (\varphi) \| +  |\overline{\varphi}|$, which is equivalent to the $H^2$-norm. We also define linear positive unbounded operator on the Hilbert space $L^2_0(\Omega)$ of the $L^2$-functions with zero mean value
\begin{align*}
B_N \varphi = -\Delta \varphi , \quad \forall \varphi \in D(B_N) = D(A_N) \cap L^2_0(\Omega)
\end{align*} 
Note that $B_N^{-1}$ is a compact linear operator on $L^2_0(\Omega)$. More generally, we can define $B_N^s$ for any $s \in \R$ noting that $\|B_N^{s/2} \cdot
 \|, s>0$, is an equivalent norm to the canonical $H^s$-norm on $D(B_N^{s/2}) \subset H^s(\Omega) \cap L^2_0(\Omega)$. Also, note that $A_N = B_N$ on $D(B_N)$. If $\varphi$ is such that $\varphi- \overline{\varphi}  \in D(B_N^{s/2})$, we have that $\|B_N^{s/2}(\varphi -  \bar{\varphi}  ) \| +  | \bar{\varphi} | $ is equivalent to the $H^s$-norm. Moreover, we set $H^{-s}(\Omega)= (H^s(\Omega))^*$, whenever $s>0$.  We also denote $ _{X'} \langle \cdot, \cdot \rangle_X  $ as a duality pairing between the Hilbert space $ X $ and its dual space $ X' $.

We state some useful and known estimates as lemmas below.
\begin{lemma}[Gagliardo-Nirenberg interpolation inequality, Theorem 1, \cite{MR109940}]\label{GNI} Let $\Omega\subset\R^n$, $\u\in\mathbb{W}^{m,p}(\Omega;\R^n), p\geq 1$ and fix $1 \leq p,q \leq \infty$ and a natural number $m$. Suppose also that a real number $\theta$ and a natural number $j$ are such that
\begin{align}\label{theta}
\theta=\left(\frac{j}{n}+\frac{1}{q}-\frac{1}{r}\right)\left(\frac{m}{n}-\frac{1}{p}+\frac{1}{q}\right)^{-1}
\end{align}
and	$\frac{j}{m} \leq \theta \leq 1.$ Then for any $\u\in\mathbb{W}^{m,p}(\Omega;\R^n),$ we have 
\begin{align} \label{gn1}
\|\nabla^j\u\|_{L^r}\leq C\left(\|\nabla^m\u\|_{L^p}^{\theta}\|\u\|_{L^q}^{1-\theta}+\|\u\|_{L^s}\right),
\end{align}
where $s > 0$ is arbitrary and the constant $C$ depends upon the domain $\Omega, m, n$. 
\end{lemma}
The following lemma is due to Ladyzenskaya (Lemma 1 and 2, Chapter 1, \cite{MR0254401} and a particular case of the Lemma \ref{GNI}.
\begin{lemma}[Ladyzenskaya's inequality]\label{lady}
For $\u\in\C_0^{\infty}(\Omega;\R^n), n = 2, 3$, there exists a constant $C$ such that
\begin{align}\label{lady1}
\|\u\|_{L^4}\leq C^{1/4}\|\u\|^{1-\frac{n}{4}}\|\nabla\u\|^{\frac{n}{4}},\text{ for } n=2,3,
\end{align}
where $C=2,4,$ for $n=2,3$ respectively. 
\end{lemma}
	
	\begin{lemma}[Agmon's inequality] (Lemma 13.2, \cite{MR2589244})
		Let $u \in H^2 \cap H^1_0(\Omega)$. Then, there exists a constant $C>0$ such that
		\begin{align} \label{agmon}
			\|u\|_{L^\infty} \leq C\|u\|^{1/2} \|u\|^{1/2}_{H^2}.
		\end{align}
	\end{lemma}
	\begin{lemma}[Ponicare-Wirtinger inequality] (page 312, \cite{MR2759829})
		Let $ \Omega $ be a connected open set of class $ C^1 $ and let $ 1 \leq p \leq \infty $. Then there exists a constant $ C $ such that 
		\begin{align}
			\|u - \bar u \| \leq C \|\nabla u \|_{L^p}, \quad \forall u \in W^{1,p}(\Omega).
		\end{align}
	\end{lemma}
	\begin{lemma}[\cite{MR1918929}]
		If $m \geq 0, mp  \leq 2$ and $p \leq q \leq \frac{2p}{2-mp}$ then $W^{m,p}(\Omega) \hookrightarrow L^q(\Omega)$, i.e.,
		\begin{align*}
			\|\u\|_{0,q} \leq C \|\u\|_{m,p} \quad \mathrm{for} \ \u \in W^{m,p}(\Omega) ,
		\end{align*}
		(note that when $mp=2$ the embedding holds for all $q < \infty$). Combining the above with the equivalence of norms of $\V_\alpha$ and $\H^\alpha$, we find that for $\alpha \in (0,1]$ and $q \in [2, \frac{2}{1-\alpha}]\, (q \in [2, \infty) \ \mathrm{if} \ \alpha =1)$ $\V_\alpha \hookrightarrow \mathbb{L}^q(\Omega)$, i.e.,
		\begin{align}
			\| \u \|_{0,q} \leq C \|\u\|_\alpha \quad \mathrm{for} \ \u \in \V_\alpha. \label{DP48}
		\end{align}
		In particular we have for $\u=A^{-1}\u$ and $\alpha=1$,
		\begin{align*}
			\|A^{-1}\u\|_{L^q} \leq \|A^{-1/2}\u\|, \ \mathrm{for} \ 2 \leq q < \infty.
		\end{align*}
		
	\end{lemma}
	Through out this paper we  assume  that $f$ appearing in \eqref{DPPsys1} - \eqref{DPPsys2}  satisfies following properties:
	\begin{itemize}
		
		\item[(\textbf{A1})] We assume that $f \in C^2(\mathbb{R})$ satisfies 
		\begin{align} \label{f condition}
			\begin{split}
				& \lim_{|s|\rightarrow \infty } f '(s) >0, \\
				& |f''(s)| \leq C_f(1+ |s|^{r-1}), \ \forall s \in \mathbb{R},
			\end{split}
		\end{align}
		where $ C_f $ is some positive constant and $r\in [1,\infty)$ is fixed. 
	\end{itemize}
	From \eqref{f condition} it follows that
	\begin{align} \label{f condition2}
		&|f'(s)| \leq C_f(1+|s|^r), \quad \forall s \in \mathbb{R} \\
	    &|f(s)| \leq C_f(1+|s|^{r+1}), \quad \forall s \in \mathbb{R}.\label{f condition3}
	\end{align}
	Let us define the following operators,
	\begin{align*}
		&b(\u,\v,\w) = \langle B(\mathbf{u}, \mathbf{v}), \mathbf{w} \rangle= \int_\Omega (\mathbf{u} \cdot \nabla ) \mathbf{v} \cdot \mathbf{w}, \quad \forall \u,\v,\w \in D(A),\\
		&b_1(\u,\varphi, \psi) =\langle B_1(\mathbf{u}, \varphi), \psi \rangle= \int_\Omega (\mathbf{u} \cdot \nabla \varphi ) \psi  , \quad \forall \u \in D(A), \ \varphi , \psi \in D(A_N),\\
		&b_2(\mu,\varphi, \w)=\langle B_2 (\mu, \varphi), \mathbf{w} \rangle = \int_\Omega \mu (\nabla \varphi \cdot \mathbf{w}) \quad \forall \mu \in L^2(\Omega), \ \varphi \in D(A_N) \cap H^3(\Omega), \ \w \in D(A).
	\end{align*}
	Also recall that  using the properties of these operators and standard inequalities mentioned above we can deduce the following estimates for these operators (page 811, \cite{MR4254175}):
	\begin{align}
		\|B(\u,\v)\|_{\V'_{\mathrm{div}}} &\leq C \|\u\|^{1/2}\|\nabla \u\|^{1/2} \|\nabla \v\| ,\\
		\|B(\u,\v)\| &\leq C \|\u\|^{1/2} \|\nabla \u\|^{1/2} \|A\v\|^{1/2}, \\
		\|B_1(\u,\varphi)\|_{D(B_N^{1/2})'} & \leq C \|\u\|^{1/2} \|\nabla \u\|^{1/2} \|\nabla \varphi\|, \\
		\|B_1(\u,\varphi)\| &\leq C \|\u\|^{1/2} \|\nabla \u\|^{1/2} \|\nabla \varphi\|^{1/2} \|A_N\varphi\|^{1/2}, \label{DP95}\\
		\|B_2(A_N\varphi,\psi)\|_{\V'_{\mathrm{div}}} &\leq C \|A_N\varphi\|^{1/2} \|\varphi\|^{1/2}_{H^3} \|\nabla\psi\| ,\\
		\|B_2(A_N\varphi,\psi)\| &\leq C \|A_N\varphi\|^{1/2} \|\varphi\|^{1/2}_{H^3} \|\nabla \psi\|^{1/2} \|A_N \psi\|^{1/2}.
	\end{align}
	Using the above-defined operators, and using $F'(\varphi) \nabla \varphi = \nabla F(\varphi)$ and by defining new pressure term $\tilde{\pi}= \pi - F(\varphi)$, and assuming that the external forcing term acts as a control we write the controlled system \eqref{DPPsys1}-\eqref{DPPsys2} in the abstract form as follows
	\begin{align}
		&\frac{d \varphi}{d t} + B_1(\u, \varphi) + A_N \mu = 0, \label{DP5} \\
		&\mu = A_N \varphi + f (\varphi), \\
		&\frac{d \u}{d t} + A \u + B(\u, \u) - B_2(A_N \varphi , \varphi) = \U, \label{DP6} \\
		& (\varphi, \u)(\tau) = (\rho, \v). \label{DP8} 
	\end{align}
	\begin{definition} \cite{MR3436705}
		Let $\rho \in D(B_N^{1/2})$ and $\v \in \G_\mathrm{div}$. Let $\U \in L^2(\tau, T; \G_\mathrm{div})$. Then, a pair $(\varphi, \u)$ is called a weak solution of \eqref{DP5}-\eqref{DP8} on $[\tau, T]$ if it satisfies \eqref{DP5}-\eqref{DP8} in a weak sense and 
		\begin{align*}
			&\varphi \in C([\tau,T];D(B_N^{1/2})) \cap L^2(\tau,T; D(B_N)) , \varphi_t \in L^2(\tau,T; D(B_N^{1/2})'), \\
			&\u \in C([\tau,T]; \G_{\mathrm{div}}) \cap L^2(\tau,T; \V_{\mathrm{div}}) , \u_t \in L^2(\tau, T; \V'_{\mathrm{div}}).
		\end{align*}
	\end{definition}
	
	\begin{definition}\cite{MR3436705}
		If $\rho \in D(B_N)$ and $\v \in \V_{\mathrm{div}}$, then a weak solution $(\varphi, \u)$ is called a strong solution of the system \eqref{DP5}-\eqref{DP8} and it satisfies
		\begin{align*}
			& \varphi \in C([\tau,T]; D(B_N)) \cap L^2(\tau,T; D(B_N) \cap H^3(\Omega))\\
			&\u \in C([\tau,T]; \V_{\mathrm{div}}) \cap L^2(\tau,T; D(A))
		\end{align*}
	\end{definition}
Using the similar techniques as in Proposition 2.1 of \cite{MR3436705} we can prove the following theorem.	
	\begin{theorem}\label{existence}
		For $\rho \in D(B_N^{1/2}), \v \in \G_{\mathrm{div}}$, $\F \in C^2(\mathbb{R})$ and $\U \in L^2(\tau, T; \G_{\mathrm{div}})$, the system \eqref{DP5}-\eqref{DP8} has a unique weak solution and the following estimates holds for all $t \in [\tau, T]$
		\begin{align}
			&\|\varphi (t)\|_{D(B_N^{1/2})}^2 + \|\u(t)\|^2 + \int_\tau^t (\|B_N^{3/2}\varphi (s)\|^2 + \|\nabla\u(s)\|^2 )ds \leq (\|\rho\|_{D(B_N^{1/2})}^2 + \|\v\|^2) + C \int_\tau^t \|\U(s)\|_{\V'_{\mathrm{div}}}^2 ds, \label{DP138} \\
			&\int_\tau^t \| A_N \varphi (s)\|^2 + \|\varphi (s) \|^2_{H^3} ds \leq Q_0(\|\rho\|_{D(B_N^{1/2})}^2 + \|\v\|^2)+ C \int_\tau^t \|\U(s)\|_{\V'_{\mathrm{div}}}^2 ds ,\\
			& \int_\tau^t \left( \|\varphi_t (s) \|_{D(B_N)'}^2 + \|\u_t(s)\|_{\V'_{\mathrm{div}}}^2 \right) ds\leq Q_0(\|\rho\|_{D(B_N^{1/2})}^2 + \|\v\|^2)+ C \int_\tau^t \|\U(s)\|_{\V'_{\mathrm{div}}}^2 ds .
		\end{align}
		where $Q_0$ denotes a monotone non-increasing function independent of time and initial data.
	\end{theorem}
	
	\begin{theorem} [Proposition 2.2, \cite{MR3436705}] \label{strongsol}
		If $\rho \in D(B_N), \v \in \V_{\mathrm{div}}$, then  
		the system \eqref{DP5}-\eqref{DP8} admits a unique strong solution and for all $t \in [\tau, T]$ the solution satisfies
		\begin{align}
			&\|\varphi (t)\|_{D(B_N)}^2 + \|\u(t)\|^2_{\V_{\mathrm{div}}} + \int_\tau^t \left( \|A\u(s)\|^2 + \|B_N^2 \varphi(s)\|^2 + \|B_N \bar{\mu}(s)\|^2 \right)ds \no  \\
			& \hspace{3cm}\leq C \left( \|\rho\|^2_{D(B_N)} + \|\u\|^2_{\V_{\mathrm{div}}} \right) + C \int_\tau^t \|\U(s)\|^2 ds, \label{DP100} \\
			&\int_\tau^t \left( \| \varphi_t(s)\|^2 + \|\u_t(s)\|^2 \right)ds \leq C \left( \|\rho\|^2_{D(B_N)} + \|\u\|^2_{\V_{\mathrm{div}}} \right) + C \int_\tau^t \|\U(s)\|^2 ds, \\
			&\| B_N^{3/2} \varphi \|^2  \leq C \left( \|\rho\|^2_{D(B_N)} + \|\u\|^2_{\V_{\mathrm{div}}} \right) + C \int_\tau^t \|\U(s)\|^2 ds.
		\end{align}
	\end{theorem}
	Using the techniques in Lemma 3.3 \cite{MR2580516} we can prove the following result.
	\begin{theorem} \label{cont depen}
		Let $(\varphi_i ,\u_i)$ be a weak solution of \eqref{DP5}-\eqref{DP8} corresponding to the initial data \\ $(\varphi_i(\tau),\u_i(\tau)) = (\rho_i, \v_i) \in D(B_N^{1/2}) \times \G_\mathrm{div}, i=1,2$ with same control $\U$.  Then the following estimate holds for $t \in [\tau, T]$,
		\begin{align*}
			&\| \nabla (\varphi_1-\varphi_2)(t)\|^2 + \| (\u_1-\u_2)(t)\|^2 + \int_\tau^t \| \nabla (\u_1-\u_2)(s)\|^2 ds+ \int_\tau^t \|(\varphi_1-\varphi_2)(s)\|_{H^2}ds \\
			& \leq C e^{Lt}(\|\nabla (\rho_1-\rho_2)\|^2 + \| \v_1-\v_2\|^2),
		\end{align*}
		where $C$ and $ L$ are positive constants depending only on the norms of the initial data, on $\Omega$, and the parameters of the problem, but are both independent of time.
	\end{theorem}
	Now we establish continuous dependence results which will be used in the later sections.

	\begin{theorem} \label{contdep0}
		Let $(\varphi_1, \u_1)$ and $(\varphi_2, \u_2)$ be weak solutions of the system \eqref{DP5}-\eqref{DP8} corresponding to initial data $(\rho_1, \v_1)$ and $(\rho_2,\v_2)$, respectively, where $(\rho_i, \v_i) \in D(B_N^{1/2}) \times \G_{\mathrm{div}}, i=1,2$. Then, for $t \in [\tau, T]$, 
		\begin{align}
			\|(\varphi_1 - \varphi_2)(t) &\|^2+\|(\u_1-\u_2)(t) \|^2_{\V'_{\mathrm{div}}} +\int_\tau^t \|B_N (\varphi_1-\varphi_2)(s) \|^2ds+ \int_\tau^t \|\u_1(s)-\u_2(s) \|^2 ds \no \\
			&\leq C (\|\rho_1-\rho_2\|^2 + \|\v_1-\v_2 \|^2_{\V'_{\mathrm{div}}}). \label{DP62}
		\end{align}
	\end{theorem}
	\begin{proof}
		Let us denote by $\varphi= \varphi_1 - \varphi_2$ and $\u = \u_1-\u_2$. Then $(\varphi, \u)$ satisfies the following system.
		\begin{align}
			\varphi_t + B_1(\u ,\varphi_1) + B_1(\u_2 ,\varphi) &= -A_N\mu, \label{DP31} \\
			\mu &= B_N\varphi + ( f(\varphi_1)-f(\varphi_2)) ,\\
			\u_t + A \u + B(\u_1,\u_1)-B(\u_2,\u_2)& = B_2(B_N \varphi , \varphi_1)+B_2(B_N \varphi_2, \varphi). \label{DP32}
		\end{align}
		Taking inner product of \eqref{DP31} with $\varphi$ and \eqref{DP32} with $A^{-1}\u$, using the properties of operators $b, b_1$ and $b_2$, and adding we get, for $ t \in [\tau, T] $,
		\begin{align}
			\frac{1}{2} \frac{d}{dt} (\|\varphi \|^2 &+ \|\u \|_{\V'_{\mathrm{div}}}^2) + \|\u\|^2 = -b_1(\u ,\varphi_1, \varphi) - (A_N\mu, \varphi)-b(\u, \u_1, A^{-1}\u)\no  \\
			&-b(\u_2,\u, A^{-1}\u)+b_2(B_N \varphi, \varphi_1, A^{-1}\u)+b_2(B_N\varphi_2, \varphi, A^{-1}\u).  \label{DP60}
		\end{align}
		We estimate the terms in \eqref{DP60} as follows. Observe that 
		\begin{align}
			-(A_N \mu, \varphi) = -(A_N(B_N\varphi + f(\varphi_1)-f(\varphi_2)),\varphi) = -\| B_N\varphi\|^2 - (A_N (f(\varphi_1)-f(\varphi_2)), \varphi). \label{DP39}
		\end{align}
		and using \eqref{f condition2}, we get
		\begin{align} \no
			|(A_N (f(\varphi_1)-f(\varphi_2)), \varphi)| &\leq \|f(\varphi_1)-f(\varphi_2)  \| \| B_N \varphi\| \\
			& \leq C\| B_N \varphi\| \|(1 + |\varphi_1|^r + |\varphi_2|^r) \varphi \| \no \\
			&\leq \frac{1}{2}\| B_N \varphi\|^2 + C (1+ \|\varphi_1\|^{2r}_{H^2} + \|\varphi_2\|^{2r}_{H^2}) \|\varphi\|^2
		\end{align}
		Using H\"older's inequality and Agmon's inequality \eqref{agmon}, we get 
		\begin{align}
			|b_1(\u, \varphi_1, \varphi)| &\leq \|\u\| \|\nabla \varphi_1\|_{L^\infty} \|\varphi\| \no \\
			& \leq C\|\u\| \|\nabla \varphi_1\|^{1/2} \|\nabla \varphi_1 \|_{H^2}^{1/2} \|\varphi\| \no \\
			& \leq \frac{1}{6}\|\u\|^2+C\|\varphi\|^2 (\|\nabla \varphi_1\|^2 + \|\varphi_1 \|_{H^3}^2 ),
		\end{align}
		and
		\begin{align}
			|b(\u,\u_1, A^{-1}\u)| = |b(\u,A^{-1}\u, \u_1)|&\leq \|\u\| \|A^{-1/2}\u\|_{L^4}\|\u_1\|_{L^4} \no  \\
			& \leq \|\u\| \| A^{-1/2}\u\|^{1/2} \|\u\|^{1/2} \|\nabla \u_1\|^{1/2}\|\u_1\|^{1/2} \no \\
			& \leq \frac{1}{6} \|\u\|^2 + C \| \u\|_{\V'_{\mathrm{div}}}^2 \|\nabla \u_1\|^2\|\u_1\|^2.
		\end{align}
		Similarly, 
		\begin{align}
			|b(\u_2, A^{-1}\u,\u)| & \leq \frac{1}{6}\|\u\|^2 + C \|A^{-1/2}\u\|^2 \|\nabla \u_2\|^2 \|\u_2\|^2,
		\end{align}
		\begin{align}
			|b_2(B_N \varphi, \varphi_1, A^{-1}\u)| &\leq \| B_N \varphi\| \|\nabla \varphi_1\|_{L^4} \|A^{-1}\u\|_{L^4} \no \\
			& \leq \frac{1}{4} \|B_N \varphi\|^2 + C \|\varphi_1\|^2_{H^3} \|\u\|_{\V'_{\mathrm{div}}}^2,
		\end{align}
		By Poincare inequality since, $\bar{\varphi}=0$, and \eqref{DP48}, we can estimate the following
		\begin{align}
			|b_2(B_N \varphi_2, \varphi, A^{-1}\u)| &\leq \|B_N\varphi_2\|_{L^4}\|\nabla \varphi\|\|A^{-1}\u\|_{L^4} \no \\
			&\leq \|B_N\varphi_2\|^{1/2} \|B_N \varphi_2\|_{H^1}^{1/2} \|B_N\varphi\| \|\u\|_{\V'_{\mathrm{div}}} \no  \\
			&\leq \frac{1}{4} \|B_N\varphi\|^2 + C\|\u\|_{\V'_{\mathrm{div}}}^2 \|\varphi_2\|_{H^3}^2 \label{DP61}.
		\end{align}
		Substituting \eqref{DP39}-\eqref{DP61} in \eqref{DP60} , we get
		\begin{align}
			\frac{1}{2} \frac{d}{dt} (\|\varphi (t) \|^2 &+ \|\u (t)\|_{\V'_{\mathrm{div}}}^2) + \frac{1}{2}\|\u(t)\|^2 + \frac{1}{2} \|B_N \varphi(t)\|^2 \leq \alpha_1(t) \|\varphi(t)\|^2 +  \beta_1(t)\|\u(t)\|_{\V'_{\mathrm{div}}}^2 
		\end{align}
		where $\alpha_1(t) = C( 1 + \|\nabla \varphi_1(t)\|^2 + \|\varphi_1(t)\|_{H^3}^2 + \|\varphi_1\|^{2r}_{H^2} +\|\varphi_2\|^{2r}_{H^2}  )$ and $\beta_1(t) = C(\|\nabla \u_1(t)\|^2 \|\u_1(t)\|^2 + \|\nabla \u_2(t)\|^2 \|\u_2(t)\|^2 + \|\varphi_1(t)\|^2_{H^2}+ \|\varphi_2(t)\|_{H^3}^2)$. Observe that, since $(\varphi_1, \u_1)$ and $(\varphi_2, \u_2)$ are weak solutions, from Theorem \ref{existence} we have $\alpha_1(\cdot), \beta_1(\cdot) \in L^1(0,T)$. By applying Gronwall's lemma, we get the required result, namely \eqref{DP62}.
	\end{proof}

\section{Dynamic Programming Principle} \label{sec DPP}
In this section, we formulate the optimal control problem as a minimization of a suitable cost functional subject to the controlled Cahn-Hilliard-Navier-Stokes system and describe the dynamic programming principle. We consider the finite-horizon problem. Similar results can be obtained in the case of the infinite-horizon problem. For the finite-horizon case for a fixed $ T>0 $ and $ 0 \leq \tau <T $, we want to minimize the cost functional

\begin{align}
	\mathcal{J}(\tau, \rho, \v; \U) = \int_\tau^T l(s, \varphi (s), \u(s), \U(s)) ds + g(\varphi (T), \u(T)).
\end{align}
where $ (\varphi, \u) $ is the solution of \eqref{DP5}-\eqref{DP8} with control $\U \in \mathcal{U}_{ad} $ and $ \mathcal{U}_{ad}$ is a closed and convex subset of $L^2(\tau, T; \G_{\mathrm{div}})$ containing $\mathbf{0}$. Observe that $\mathcal{U}_{ad}$ is non-empty.

We formulate the optimal control problem as the minimization of the cost $ \mathcal{J} $, i.e.,
\begin{align}
\inf_{\U \in \mathcal{U}_{ad}} \mathcal{J}(\tau,  \rho,\mathbf{v}, \mathrm{U}) \label{DP44}
\end{align}
A solution to the problem \eqref{DP44} is called an optimal control and $(\varphi^*, \u^*, \U^*)$ is called optimal triplet where $(\varphi^*, \u^*)$ is the strong solution of the system \eqref{DP5}-\eqref{DP8} with optimal control $\U^*$ and initial data $(\varphi(\tau), \u(\tau))=(\rho,\v)$. 
We consider the optimal control problem \eqref{DP44} with $\mathcal{U}_{ad} = L^2 (\tau , T; \mathcal{U}_R)$, where 
$$\mathcal{U}_R := \{ \U \in \G_{\mathrm{div}} \ : \ \|\U\| \leq R \}.$$
Let us define a value function $\mathcal{V} : [0,T] \times D(B_N^{1/2}) \times \G_{\mathrm{div}} \rightarrow \mathbb{R}$ as
\begin{align} \label{valuefn}
\mathcal{V} (\tau, \rho, \mathbf{v}) = \inf_{\mathrm{U} \in \mathcal{U}_{ad}} \mathcal{J}(\tau, \rho, \mathbf{v}, \mathrm{U}).
\end{align}
We assume the following conditions on $l$ and $g$.
\begin{assumption} \label{l_assumption}
There exist families of moduli of $\sigma_K$, a constant $L_K$ for $K>0$ and, $k\geq 0$ such that 
\begin{itemize}
\item[(i)] $l : (0,T) \times D(B^{1/2}_N) \times \G_{\mathrm{div}} \times \mathcal{U}_{ad} \rightarrow \R$ is continuous such that
\begin{align} \label{DP146}
|l(t, \rho_1, \v_1 ,\U) - l(s, \rho_2, \v_2, \U)| \leq  \sigma_K(|t-s|)+ L_K(\|\nabla(\rho_1 - \rho_2)\| + \|\v_1-\v_2\|),
\end{align}
\begin{align} \label{DP132}
|l(t, \rho, \v, \U)| \leq C (1 + \|\rho\|_{D(B_N)}^k + \|\v\|_{\V_{\mathrm{div}}}^k),
\end{align}
for $\|\rho_1\|_{D(B^{1/2}_N)},\| \rho_2\|_{D(B^{1/2}_N)},\|\rho\|_{D(B_N)}, \|\v_1\|, \|\v_2\|, \|\v\| \leq K$.
\item[(ii)] $g : D(B^{1/2}_N) \times \G_{\mathrm{div}} \rightarrow \R$ is continuous such that
\begin{align*}
|g(\rho_1,\v_1) -g(\rho_2, \v_2)| \leq L_K(\|\rho_1 - \rho_2 \| + \|\v_1 - \v_2 \|_{\V_{\mathrm{div}}'}),
\end{align*}
for $\|\rho_1\|_{D(B^{1/2}_N)},\|\rho_2\|_{D(B^{1/2}_N)}, \|\v_1\|, \|\v_2\| \leq K$, and
\begin{align} \label{DP143}
|g(\rho, \v)| \leq C (1 + \| \nabla \rho\|^k + \|\v\|^k).
\end{align}
\end{itemize}
\end{assumption}

\subsection{Continuity properties of a value function}
In this section, we present continuity properties of the value function \eqref{valuefn}. Before that,  we prove the following propositions which give the continuous dependence on time which will be used in proving the smoothness properties of the value function and also while proving it to be a viscosity solution.
\begin{proposition} \label{DP130}
Let$(\rho, \v) \in D(B_N^{1/2})\times \G_{\mathrm{div}} $ and $\U \in  \mathcal{U}_{ad}$. Let $(\varphi, \u)$ be the solution corresponding to the initial data $(\varphi(\tau), \u(\tau))=(\rho, \v)$. Then the following estimates hold for $t \in [\tau, T]$:
\begin{align} \label{DP125}
\|\varphi(t)-\rho\|_{[D(B_N^{1/2})]'}^2 + \| \u(t)-\v\|^2_{\V_{{\mathrm{div}}}'}  \leq C(t-\tau).
\end{align}
\begin{align} \label{DP129}
\int_\tau^t (\|\varphi(s)-\rho\| _{D(B_N^{1/2})}^2 + \| \u(s)-\v \|^2 ) ds \leq C (t-\tau)
\end{align}
\end{proposition}
\begin{proof}
Let $(\varphi, \u)$ be the  solution of the system \eqref{DP5}-\eqref{DP8} with control $\U$ and initial data \\ $(\varphi(\tau), \u(\tau))= (\rho, \v) $. Let us denote $\xi = \varphi(t)-\rho$ and $\z=\u(t)-\v$. Then $(\xi, \z)$ satisfies
\begin{align}  \label{DP119}
\frac{d \xi}{dt} +B_1(\z, \varphi)+B_1(\v ,\varphi)+ A_N\mu&=0,\\
\mu &= B_N \xi + f(\varphi) +B_N \rho,  \\
\frac{d\z}{dt} +  A\z + A \v +B(\z,\u)+B(\v,\u) & =B_2(B_N \varphi, \varphi) +\U. \label{DP120}
\end{align}
Now take $ L^2 $ inner product of \eqref{DP119} with $B_N^{-1} \xi$, and \eqref{DP120} with $A^{-1}\z$. By adding we get 
\begin{align} \label{DP123}
& \frac{1}{2} \frac{d}{dt} (\|B_N^{-1/2}\xi \|^2 +\|\z\|_{\V_{{\mathrm{div}}}'}^2)  +  \|\z\|^2   = -(B_1(\z, \varphi), B_N^{-1} \xi ) -(B_1(\v,\varphi), B_N^{-1} \xi) -(A_N\mu, B_N^{-1}\xi) \no  \\
&-(\v,\z)-b(\z,\u,A^{-1}\z)-b(\v,\u,A^{-1}\z)+(B_2(B_N \varphi, \varphi),A^{-1}\z) +(\U,A^{-1}\z).
\end{align}
Now we estimate the right-hand side terms as follows. Observe that
\begin{align} \label{DP121}
-(A_N\mu, B_N^{-1}\xi) = -(\mu, \xi) &= -\|\nabla \xi\|^2 -(f(\varphi), \xi) -(\nabla \rho, \nabla \xi) \no \\
& \leq - \frac{1}{2} \|\nabla \xi \|^2 + \|f(\varphi)\| \|\nabla \xi \|  + \frac{1}{2} \|\nabla \rho\|^2 \no  \\
& \leq - \frac{1}{4} \|\nabla \xi \|^2 + C(1+ \|\varphi\|_{H^1}^{2(r+1)}) +  \frac{1}{2} \|\nabla \rho\|^2 
\end{align}	
where we used \eqref{f condition3} and Sobolev inequality. Similarly, we estimate,
\begin{align}
|(B_1(\z, \varphi), B_N^{-1} \xi )| &=|(\z \cdot \nabla \varphi,B_N^{-1} \xi )| = |(\z \cdot \nabla B_N^{-1} \xi , \varphi)| \no \\
&\leq \|\z\| \|B_N^{-1/2} \xi\| \|\varphi\|_{L^\infty} \no \\
&\leq \frac{1}{8} \|\z\|^2 + C \|\varphi\|_{H^2}^2 \|B_N^{-1/2} \xi\|^2,
\end{align}
\begin{align}
|(B_1(\v,\varphi), B_N^{-1} \xi)| &= |(\v \cdot \nabla B_N^{-1/2} \xi, \varphi)| \no \\
& \leq \frac{1}{2} \|\v\|^2 + C \|\varphi\|_{H^2}^2 \|B_N^{-1/2} \xi\|^2,
\end{align}

\begin{align}
|(\v,\z) | \leq \|\v\| \|\z\| \leq \frac{1}{8} \|\z\|^2 + C \|\v\|^2,
\end{align}
\begin{align}
|b(\z,\u,A^{-1}\z)| &= |b(\z, A^{-1}\z, \u)| \no \\ 
& \leq \|\z\|\|A^{-1/2} \z\|_{\mathbb{L}^4} \|\u\|_{\mathbb{L}^4} \no \\
& \leq \|\z\| \|A^{-1/2} \z\|^{1/2} \|\z\|^{1/2} \|\u\|^{1/2} \|\nabla \u\|^{1/2} \no \\
& \leq \frac{1}{8} \|\z\|^2 + C \|\u\|^2 \|\nabla \u\|^2  \| \z\|^2_{\V_{{\mathrm{div}}}'},
\end{align}
\begin{align}
|b(\v,\u,A^{-1}\z) | &\leq C \|\v\| \|A^{-1/2} \z\|^{1/2} \|\z\|^{1/2} \|\nabla \u\| \no \\
& \leq \frac{1}{6} \|\z\|^2 + C  \|\v\|^{4/3} \|A^{-1/2} \z\|^{2/3}  \|\nabla \u\|^{4/3} \no \\
& \leq \frac{1}{8} \|\z\|^2 + \frac{1}{2} \| \z\|^2_{\V_{{\mathrm{div}}}'} + C \|\v\|^2 \|\nabla \u\|^2,
\end{align}
\begin{align}
|(B_2(B_N \varphi, \varphi),A^{-1}\z)| &\leq \|B_N \varphi\|_{L^4} \|\nabla \varphi\|\| A^{-1}\z\|_{L^4} \no \\ 
& \leq C\|\varphi\|_{H^3} \|\varphi\|_{H^1} \|\z\|_{\V_{{\mathrm{div}}}'} \no \\
&  \leq \frac{1}{2} \| \varphi\|_{H^1}^2 + C \|\varphi\|_{H^3}^2 \|\z\|_{\V_{{\mathrm{div}}}'}^2,
\end{align}
\begin{align}  \label{DP122}
|(\U,A^{-1}\z)| &\leq \|\U\| \|A^{-1}\z\| \no \\
& \leq \frac{1}{2} \|\U\|^2 + C \|\z\|^2_{\V_{{\mathrm{div}}}'}.
\end{align}
Substituting \eqref{DP121}-\eqref{DP122} in \eqref{DP123}, we get 
\begin{align}
\frac{1}{2} \frac{d}{dt} (\|B_N^{-1/2}\xi \|^2 +\|\z\|_{\V_{{\mathrm{div}}}'}^2)  + \frac{1}{2} \|\z\|^2 + \frac{1}{2} \|\nabla \xi\|^2 \leq C(1+ \|\v\|^2 + \|\nabla \rho\|^2 + \|\U\|^2 + \|\varphi\|_{H^1}^{2(r+1)}) \no \\
+ C(1+\|\varphi\|^2_{H^3} + \|\u\|^2\|\nabla\u\|^2 + \|\v\|^2\|\nabla\u\|^2) (\| B_N^{-1/2}\xi\|^2 + \|\z\|^2_{\V_{{\mathrm{div}}}'}).\label{DP124}
\end{align} 
By integrating \eqref{DP124} and applying Gronwall's lemma we get 
\begin{align}
\|\varphi(t)-\rho\|_{[D(B_N^{1/2})]'}^2 + \| \u(t)-\v\|^2_{\V_{{\mathrm{div}}}'}  \leq C(t-\tau). \label{DP128}
\end{align}
By substituting \eqref{DP128} in \eqref{DP124} upon integrating from $ \tau $ to $ t $, we get \eqref{DP129}.
\end{proof}
\begin{proposition} \label{DP126}
Let$(\rho, \v) \in D(B_N^{1/2}) \times \G_{\mathrm{div}} $ and $\U \in  \mathcal{U}_{ad}$. Let $(\varphi, \u)$ be the strong solution corresponding to the initial data $(\varphi(\tau), \u(\tau))=(\rho, \v)$. Then there exists a modulus of continuity $ \omega $ independent of control $\U$ such that for $t \in [\tau, T]$,
\begin{align} \label{DP127}
\|\nabla (\varphi(t)-\rho)\| + \| \u(t)-\v\| \leq \omega (t-\tau).
\end{align}
\end{proposition}
\begin{proof}
We want to show that
\begin{align} \label{DP113}
\sup_{t \in [\tau, \tau + \epsilon], \U \in \mathcal{U}_{ad}} (\|\nabla (\varphi(t)-\rho)\|+\| \u(t) - \v \| )\rightarrow 0 \ \mathrm{as} \ \epsilon \rightarrow 0 .
\end{align}
We prove it by contradiction. Suppose that \eqref{DP113} is false. Then, at least one of the terms $ \|\nabla (\varphi(t)-\rho)\| $ or $ \| \u(t) - \v \|$ does not go to zero as $ \epsilon \rightarrow 0 $. Without loss of generality, we assume that there exists a sequence $t_n$ and $ \U_n $ such that 
\begin{align} \label{DP114}
\| \u_n(t_n) - \v\| \geq \epsilon \quad \mathrm{as} \quad t_n \rightarrow \tau.
\end{align}
where $ (\varphi_n, \u_n) $ is the strong solution of the system \eqref{DP5}-\eqref{DP8} with control $ \U_n $. However, by \eqref{DP125} we have 
\begin{align*}
&\varphi_n (t_n) \rightarrow \rho \quad \mathrm{strongly \ in} \ D(B_N^{1/2})', \\
&\u_n(t_n) \rightarrow \v \quad \mathrm{strongly\ in} \ \V'_{\mathrm{div}}.
\end{align*} 
By Theorem \ref{existance}, we have the uniform boundedness of $ \| \u_n(t_n)\|$ and $ \| \nabla  \varphi_n(t_n)\| $, and hence by uniqueness of limits as $t_n \rightarrow \tau$
\begin{equation}\label{DP117}
\begin{aligned} 
\nabla \varphi_n (t_n) \rightarrow \nabla \rho \quad \mathrm{weakly \ in} \ H, \\
\u_n(t_n) \rightarrow \v \quad \mathrm{weakly \ in} \ \G_{\mathrm{div}},
\end{aligned}
\end{equation}
and 
\begin{equation}\label{DP115}
\begin{aligned}
\|\nabla \rho\| \leq \liminf_{n \rightarrow \infty} \|\nabla \varphi_n(t_n)\|, \\
\|\v\| \leq \liminf_{n \rightarrow \infty} \|\u_n(t_n)\|.
\end{aligned}
\end{equation}
From \eqref{DP138}, we can also conclude that
\begin{align} \label{DP116}
\|\nabla  \rho\|^2+\|\v\|^2 \geq \limsup_{n \rightarrow \infty} (\|\nabla  \varphi_n(t_n)\|^2 + \|\u_n(t_n)\|^2).
\end{align}
Combining \eqref{DP115} and \eqref{DP116} we get the convergence of the norms, i.e.,
\begin{align} \label{DP118}
\|\nabla \varphi_n(t_n)\|^2 +	\|\u_n(t_n)\|^2 \rightarrow \|\nabla \rho\|^2 + \|\v\|^2.
\end{align}
Hence, using \eqref{DP117} and \eqref{DP118}, we see that 
\begin{align*}
\u_n({t_n}) \rightarrow \v \quad \mathrm{strongly \ in} \ \G_{\mathrm{div}},
\end{align*}
which gives a contradiction to \eqref{DP114}. 
\end{proof}

Assuming the higher regularity for the initial data we can improve the regularity of the terms appearing in the above estimate as proved in the next proposition.

\begin{proposition} \label{timedep1}
Let$(\rho, \v) \in D(B_N)\times \V_{\mathrm{div}} $ and $\U \in  \mathcal{U}_{ad}$. Let $(\varphi, \u)$ be the strong solution corresponding to the initial data $(\varphi(\tau), \u(\tau))=(\rho, \v)$. Then the following holds for $t \in [\tau, T]$:
\begin{align} \label{DP74}
\|\varphi(t)-\rho\|^2 + \| \u(t)-\v\|^2  \leq C(t-\tau).
\end{align}
\end{proposition}

\begin{proof}
Let $(\varphi, \u)$ be the strong solution of the system \eqref{DP5}-\eqref{DP8} with control $\U$ and initial data $(\varphi(\tau), \u(\tau))= (\rho, \v) $. Let us denote $\xi = \varphi(t)-\rho$ and $\z=\u(t)-\v$. Then $(\xi, \z)$ satisfies
	\begin{align} 
		\frac{d \xi}{dt} +B_1(\z, \varphi)+B_1(\v ,\varphi)+ A_N\mu&=0, \label{DP15}\\
		\mu &= B_N \xi + f(\varphi) +B_N \rho, \label{DP16} \\
		\frac{d\z}{dt} + A\z + A \v +B(\z,\u)+B(\v,\z)+B(\v,\v) & =B_2(B_N \xi, \varphi) + B_2(B_N\rho,\varphi)+\U. \label{DP17}
	\end{align}
	Now take $ L^2 $-inner product of \eqref{DP15} with $ \xi$, \eqref{DP16} with $ B_N \xi$ and \eqref{DP17} with $\z$. By adding we get 
	\begin{align}
		& \frac{1}{2} \frac{d}{dt} (\| \xi \|^2 +\|\z\|^2) + \|B_N \xi\|^2 + \|\nabla \z\|^2\no \\
		& = -(B_1(\v,\varphi),  \xi)  -(B_1(z, \varphi), \xi) -(f(\varphi),B_N \xi) -(B_N\rho,B_N \xi)\no \\
		& \quad -(A \v,\z) -(B(\z,\u),\z)-(B(\v,\u),\z)+(B_2(B_N\varphi, \varphi),\z)  +(\U,\z). \label{DP18}
	\end{align}
	Now we estimate the right-hand side terms one by one using H\"older's, Young's, Gagliardo-Nirenberg and Poincare's inequalities
	\begin{align}
		|(B_1(\v,\varphi),  \xi)| &\leq C \|\v\|_{\L^4} \|\nabla \varphi\|_{L^4} \| \xi\| \no \\
		& \leq \frac{1}{2}\|\nabla \v\|^2 + C \|\nabla \varphi\|_{L^4}^2 \| \xi\|^2\no \\
		&\leq \frac{1}{2}\|\nabla \v\|^2 + C \|\varphi\|_{H^2}^2 \| \xi\|^2	, \label{DP71}
	\end{align}
	\begin{align}
		|(B_1(\z,\varphi), \xi)| &\leq C \|\z\|_{\L^4} \|\nabla \varphi\|_{L^4} \| \xi\| \no \\
		& \leq \frac{1}{12}\|\nabla \z\|^2 + C \|\nabla \varphi\|_{L^4}^2 \| \xi\|^2\no \\
		&\leq \frac{1}{12}\|\nabla \z\|^2 + C \|\varphi\|_{H^2}^2 \| \xi\|^2	,
	\end{align}
	\begin{align}
		|(f(\varphi),B_N \xi)| &\leq \|f(\varphi) \| \|B_N \xi \| \leq \frac{1}{4} \|B_N \xi \|^2 + C (1 + \| \varphi \|^{2(r+1)}_{H^1})  ,
	\end{align} 
	\begin{align}
		|(B_N\rho,B_N \xi)| &\leq \| B_N \rho \| \| B_N \xi\|\leq \frac{1}{4} \|B_N \xi\|^2 + C\| B_N \rho \|^2 , 
	\end{align}
	\begin{align}
		 |(A \v,\z)| \leq  \|\nabla \v\| \| \nabla \z \| \leq \frac{1}{12} \|\nabla \z\|^2 + C \|\nabla \v\| ^2,
	\end{align}
	\begin{align}
		|(B(\z,\u),\z)| = |(B(\z,\z),\u)| \leq C \| \z\| \|\nabla \u\| \| \nabla \z\| \leq \frac{1}{12} \|\nabla \z\|^2 + C \|\z\|^2 \|\nabla\u\|^2,
	\end{align}
	\begin{align}
	|(B(\v,\u),\z)| =|(B(\v,\z),\u)| \leq C \|\nabla \v\| \|\nabla \z\| \|\nabla \u\| \leq \frac{1}{12} \| \nabla \z\|^2 +  C\|\nabla \v\|^2 \|\nabla \u\|^2,
	\end{align}
	\begin{align}
		|(B_2(B_N\varphi, \varphi),\z)| &\leq \|B_N\varphi\| \|\nabla\varphi\|_{L^4}\|\z\| _{L^4} \no \\
		&\leq C\|B_N \varphi\| \|\varphi\|_{H^2}\|\nabla \z\| \no \\
		&\leq \frac{1}{12} \|\nabla \z\| ^2+C\|B_N \varphi\|^2 \|\varphi\|^2_{H^2},
	\end{align}
	\begin{align}
		|(\U,\z)|\leq \|\U\|\|\z\|  \leq C\|\U\|\|\nabla\z\| \leq \frac{1}{12} \|\nabla \z\|^2 + C \|\U\|^2. \label{DP72}
	\end{align} 
	Substituting \eqref{DP71}-\eqref{DP72} in \eqref{DP18}, we get 
	\begin{align}
		&\frac{1}{2} \frac{d}{dt} (\|\xi \|^2 +\|\z\|^2) + \frac{1}{2}\|B_N \xi\|^2 + \frac{1}{2}\|\nabla \z\|^2 \leq C(\|\varphi\|^2_{H^2}+\|\nabla \u\|^2)( \|\nabla \xi\|^2 +\|\z\|^2 )\no \\
		& + C(1+ \|\U\|^2+ \|\varphi\|^{2(r+1)}_{H^1} + \|B_N\varphi\|^2 \|\varphi\|^2_{H^2} + \|\v\|^2\|\nabla \u\|^2 + \|\nabla \v\|^2 + \|B_N \rho\|^2) \label{DP73}.
	\end{align}
	By integrating \eqref{DP73} from $\tau $ to $t$ and
	applying Gronwall's lemma and using Theorem \ref{existance}, we get \eqref{DP74}.
\end{proof}

\begin{proposition} \label{contdep1}
	Let $  \rho \in D(B_N),\v \in \V_{\mathrm{div}} $, and $\U \in \mathcal{U}_{ad}$. Let $(\varphi, \u)$ be the strong solution corresponding to the initial data $(\varphi(\tau), \u(\tau))=(\rho, \v)$. Then, there exists  a modulus $ \omega $ independent of control $\U$ such that for $t \in [\tau, T]$, 
	\begin{align}\label{DP78}
		\|B_N(\varphi(t)-\rho)\|+ \| \u(t) - \v \|_{\V_{\mathrm{div}}} \leq \omega(t-\tau).
	\end{align}
\end{proposition}
\begin{proof}
	We want to show that
	\begin{align}
		\sup_{t \in [\tau, \tau + \epsilon], \U \in \mathcal{U}_{ad}} (\|B_N(\varphi(t)-\rho)\|+\| \u(t) - \v \|_{\V_{\mathrm{div}}} )\rightarrow 0 \ \mathrm{as} \ \epsilon \rightarrow 0 . \label{DP99} 
	\end{align}
	We prove it by contradiction. Suppose that \eqref{DP99} is false. Then, without loss of generality, we assume that there exist  sequences $t_n$ and $ \U_n $ such that 
	\begin{align} \label{DP104}
		\| \u_n(t_n) - \v\|_{\V_{\mathrm{div}}} \geq \epsilon \quad \mathrm{as} \quad t_n \rightarrow \tau.
	\end{align}
	where $ (\varphi_n, \u_n) $ is the strong solution of the system with control $ \U_n $. However, by \eqref{DP74} we have 
	\begin{align*}
		&\varphi_n (t_n) \rightarrow \rho \quad \mathrm{strongly \ in} \ D(B_N^{1/2}) ,\\
		&\u_n(t_n) \rightarrow \v \quad \mathrm{strongly\ in} \ \G_{\mathrm{div}}.
	\end{align*} 
	By \eqref{DP100}, we have the uniform boundedness of $ \| \u_n(t_n)\|_{\V_{\mathrm{div}}}$ and $ \| \varphi_n(t_n)\|_{D(B_N)} $, and hence by uniqueness of limits
	\begin{equation}\label{DP103}
		\begin{aligned} 
			B_N \varphi_n (t_n) \rightarrow B_N\rho \quad \mathrm{weakly \ in} \ H,\\
			\u_n(t_n) \rightarrow \v \quad \mathrm{weakly \ in} \ \V_{\mathrm{div}},
		\end{aligned}
	\end{equation}
	and 
	\begin{equation}\label{DP101}
		\begin{aligned}
			\|B_N\rho\| \leq \liminf_{n \rightarrow \infty} \|B_N\varphi_n(t_n)\|, \\
			\|\v\|_{\V_{\mathrm{div}}} \leq \liminf_{n \rightarrow \infty} \|\u_n(t_n)\|_{\V_{\mathrm{div}}}.
		\end{aligned}
	\end{equation}
	From \eqref{DP100}, we can also conclude that
	\begin{align} \label{DP102}
		\|B_N\rho\|^2+\|\v\|^2_{\V_{\mathrm{div}}} \geq \limsup_{n \rightarrow \infty} \|B_N \varphi_n(t_n)\|^2 + \|\u_n(t_n)\|^2_{\V_{\mathrm{div}}}.
	\end{align}
	Combining \eqref{DP101} and \eqref{DP102} we get the convergence of the norms, i.e.,
	\begin{align} \label{DP105}
		\|B_N \varphi_n(t_n)\|^2 +	\|\u_n(t_n)\|^2_{\V_{\mathrm{div}}} \rightarrow \|B_N\rho\|^2 + \|\v\|^2_{\V_{\mathrm{div}}}.
	\end{align}
	Hence, using \eqref{DP103} and \eqref{DP105}, we see that 
	\begin{align*}
		\u_n(s_n) \rightarrow \v \quad \mathrm{strongly \ in} \ \V_{\mathrm{div}},
	\end{align*}
	which gives a contradiction to \eqref{DP104}. 
\end{proof}

Now using the Theorem \ref{contdep0} and Proposition \ref{contdep1} we prove the smoothness properties of the value function.

\begin{theorem}
	For every $K >0$ there exists a modulus $\omega_K$ such that the value function defined in \eqref{valuefn} satisfies 
	\begin{align} \label{DP81}
		|\mathcal{V} (t_1,\rho_1, \v_1)-\mathcal{V} (t_2,\rho_2, \v_2)| \leq \omega_K \left( |t_1-t_2|^{1/2} + \| \rho_1 -\rho_2\|+ \| \v_1-\v_2 \|_{\V_{{\mathrm{div}}}'} \right)
	\end{align}
	for $t_1, t_2 \in [0,T]$ and $\|\rho_1\|_{D(B^{1/2}_N)}, \| \rho_2\|_{D(B^{1/2}_N)}, \| \v_1\| ,\|\v_2\| \leq K$. Moreover, 
	\begin{align} \label{DP144}
		|\mathcal{V} (t, \rho, \v)| \leq C(1 + \|\nabla \rho\|^k + \|\v\|^k).
	\end{align}
\end{theorem}

\begin{proof}
	
	First, we consider 
	\begin{align}
		|\mathcal{V}(t,\rho_1, & \v_1) - \mathcal{V}(t, \rho_2, \v_2)| = |\inf_{\U \in \mathcal{U}_{ad}} \mathcal{J}(t, \rho_1,\v_1, \U) - \inf_{\U \in \mathcal{U}_{ad}} \mathcal{J}(t, \rho_2,\v_2, \U)| \no \\
		& \leq \sup_{\U \in \mathcal{U}_{ad}} |\mathcal{J}(t, \rho_1,\v_1, \U) - \mathcal{J}(t, \rho_2,\v_2, \U)  | \no \\
		& \leq \sup_{\U \in \mathcal{U}_{ad}} \Big( \int_t^T |l(s, \varphi_1(s), \u_1(s), \U(s)) -l (s, \varphi_2(s), \u_2(s), \U(s)) | ds   \no \\
		& \hspace{3cm} + |g(\varphi_1(T), \u_1(T))-g(\varphi_2(T),\u_2(T))| \Big) \no   \\
		& \leq C \sup_{\U \in \mathcal{U}_{ad}} \Big( \int_t^T (\|\nabla(\varphi_1 - \varphi_2)\| + \|\u_1-\u_2\|)ds \no + L_R(\|\varphi_1 - \varphi_2 \| + \|\u_1 - \u_2 \|_{\V_{\mathrm{div}}'} \Big), \\
		& \leq C_K(\|\rho_1 - \rho_2 \| + \|\v_1 - \v_2 \|_{\V_{\mathrm{div}}'}), \label{DP141}
	\end{align}
for $\|\rho_1\|_{D(B^{1/2}_N)}, \|\rho_2\|_{D(B^{1/2}_N)}, \| \v_1\| ,\|\v_2\| \leq K $, where we used the Theorem \ref{contdep0} in the last step.
	
	To establish continuity of $\mathcal{V}$ in time, we consider two systems that evolve from the same initial data $(\rho, \v) $  at time $t_1$ and $t_2$. Without loss of generality, we assume with $0\leq t_1 \leq t_2 \leq T$. Let $(\varphi_2, \u_2)$ be a solution of the system \eqref{DP5}-\eqref{DP8} with initial data $(t_2, \rho, \v)$ and control $\U_2$.  Fix $\U_0 \in \mathcal{U}_R$. We define control $\U_1$ in the following way
	\[
	\U_1(s) = 
	\begin{cases}
		\U_0(s), & s \in (t_1,t_2), \\
		\U_2(s), & s \in (t_2,T).
	\end{cases}
	\]
	Let $(\varphi_1,\u_1)$ be the solution of the system \eqref{DP5}-\eqref{DP8} with initial data $(\varphi_1(t_1), \u_1(t_1))= (\rho, \v)$, and control $\U_1$. Since the solution of the system \eqref{DP5}-\eqref{DP8} with controls $\U_1$ and $\U_2$ is unique, using the semigroup property, we have 
	\begin{align}
		X_1(s;t_1, \rho, \v, \U_1) = X_2(s; t_2, X_1(t_2;t_1, \rho, \v, \U_1), \U_2), \quad \mathrm{for} \ t_1 \leq t_2 \leq s, \label{DP50}
	\end{align}
	where $X_1 = (\varphi_1, \u_1)$ and $X_2= (\varphi_2, \u_2)$.
	Now consider
	\begin{align}
		|\mathcal{V}(t_1,\rho, \v) - \mathcal{V}(t_2, \rho, \v)| &= |\inf_{\U \in \mathcal{U}_{ad}} \mathcal{J}(t_1, \rho,\v, \U) - \inf_{\U \in \mathcal{U}_{ad}} \mathcal{J}(t_2, \rho,\v, \U)| \no \\
		& \leq \sup_{\U \in \mathcal{U}_{ad}} |\mathcal{J}(t_1, \rho,\v, \U) - \mathcal{J}(t_2, \rho,\v, \U)  | \no \\
		&  \leq \sup_{\U \in \mathcal{U}_{ad}} \left( |  \int_{t_1}^T l(s, \varphi_1(s), \u_1(s), \U(s)) -  \int_{t_2}^Tl (s, \varphi_2(s), \u_2(s), \U(s))  ds |  \right. \no \\
		& \hspace{3cm} \left. +|g(\varphi_1(T), \u_1(T))-g(\varphi_2(T),\u_2(T))| \right) \no   \\
		& \leq \sup_{\U \in \mathcal{U}_{ad}}  \int_{t_1}^{t_2} | l(s, \varphi_1(s), \u_1(s), \U(s)) | ds  \no \\
		& \ \ + \sup_{\U \in \mathcal{U}_{ad}}  (\int_{t_2}^T | (l(s, \varphi_1(s), \u_1(s), \U(s)) -l (s, \varphi_2(s), \u_2(s), \U(s))|  ds)  \no \\
		& \ \ + C (\| \varphi_1(T) - \varphi_2(T) \| + \| \u_1(T) - \u_2(T)\|_{\V_{{\mathrm{div}}}'})
	\end{align}
	Observe that by Holder inequality  and \eqref{DP138}, we get
	\begin{align}
		\int_{t_1}^{t_2} | l(s, \varphi_1(s), \u_1(s), \U(s))| ds & \leq C\int_{t_1}^{t_2} (1+ \|\varphi_1(s)\|_{D(B^{1/2}_N)}^k + \| \u_1(s) \|^k)ds  \no \\
		& \leq C_K (t_2 - t_1),
	\end{align}
 Using \eqref{DP146}, \eqref{DP50}, Theorem \ref{cont depen}, Theorem \ref{contdep0} and Proposition \ref{timedep1}, we get
\begin{align}
    \sup_{\U \in \mathcal{U}_{ad}}  \int_{t_2}^T |l(s, &\varphi_1(s), \u_1(s), \U(s))  -l (s, \varphi_2(s), \u_2(s), \U(s))|  ds  \no \\
    & \leq C \int_{t_2}^T (\|\nabla (\varphi_1(s) - \varphi_2(s)) \| + \|\u_1(s) - \u_2(s) \|) ds \no \\
       & \leq C (\|\varphi_1(t_2) - \varphi_2(t_2) \| + \| \u_1(t_2) - \u_2(t_2)\| ) \no \\ 
       &=  C (\|\varphi_1(t_2) - \rho \| + \| \u_1(t_2) - \v\| )  \no \\
    & \leq C_K (t_2 - t_1)^\frac{1}{2}
\end{align}
and	using Theorem \ref{contdep0}, Proposition \ref{DP130} and Proposition \eqref{timedep1},  
\begin{align}
    \| \varphi_1(T) - \varphi_2(T) \| + \| \u_1(T) - \u_2(T)\|_{\V_{{\mathrm{div}}}'} 
    & \leq C (\|\varphi_1(t_2) - \varphi_2(t_2) \| + \|\u_1(t_2) -\u_2(t_2) \|_{\V_{{\mathrm{div}}}'}) \no \\
    &=  C (\|\varphi_1(t_2) - \rho \| + \|\u_1(t_2) -\v \|_{\V_{{\mathrm{div}}}'}) \no \\
    & \leq C (t_2 - t_1)^\frac{1}{2}
\end{align}
Hence, we get 
\begin{align}\label{DP142}
    |\mathcal{V}(t,\rho_1, \v_1) - \mathcal{V}(t, \rho_2, \v_2)| \leq w_K(t_2-t_1 ).
\end{align}
Combining \eqref{DP141} and \eqref{DP142}, we get \eqref{DP81}. Also, using \eqref{DP132}, \eqref{DP143} and Theorem \ref{existance}, we get \eqref{DP144}. 
	
\end{proof}

\subsection{Dynamic Programming Principle (DPP)}
Now we prove the main theorem of this section namely, the dynamic programming principle.
\begin{theorem}[Bellman's principle of optimality] \label{DPP}
Let $ \mathcal{V} $ be as defined in \eqref{valuefn}. Then for $ 0 \leq \tau \leq t_0 \leq T $, we have
\begin{align*}
	\mathcal{V} (\tau, \rho, \mathbf{v}) = \inf_{\mathrm{U} \in \mathcal{U}_{ad}} \left\{  \int_\tau^{t_0} l(t, \varphi(t), \u(t) , \U(t)) dt + \mathcal{V} (t_0, \varphi(t_0) , \mathbf{u}(t_0))\right\}. 
	\end{align*}
\end{theorem}
\begin{proof}
	Let $\mathrm{U}$ be an arbitrary control. 
	Let $ (\varphi,\mathbf{u}) $ be the corresponding strong solution of the system \eqref{DP5}-\eqref{DP8} with initial data $(\varphi(\tau), \u(\tau))= (\rho, \v)$. Then, for $\tau \leq t_0 \leq T$
	\begin{align*}
		\mathcal{J}(\tau,  \rho,\mathbf{v}, \mathrm{U}) &=   \int_\tau^T l(t, \varphi(t), \u(t) , \U(t)) dt  +g(\varphi(T),   \mathbf{u}(T))\\ 
		& =  \int_\tau^{t_0} l(t, \varphi(t), \u(t) , \U(t)) dt +\int_{t_0}^T l(t, \varphi(t), \u(t) , \U(t)) dt +  g(\varphi(T),   \mathbf{u}(T))\\
		& \geq  \int_\tau^{t_0} l(t, \varphi(t), \u(t) , \U(t))  dt + \inf_{\bar{\U}\in \mathcal{U}_R} \left\{ \int_{t_0}^T l (t,\bar{\varphi}(t),  \bar{\mathbf{u}}(t), \bar{\U}(t)) dt + g(\bar{\varphi}(T), \bar{\mathbf{u}}(T) )\right\} ,
	\end{align*}
	where $(\bar{\varphi}, \bar{\mathbf{u}})$ is the solution of \eqref{DP5}-\eqref{DP8} with initial conditions $(\varphi(t_0), \u (t_0))$ and control $\bar{\U}\in \mathcal{U}_R$. Then
	\begin{align} \label{DP1}
		\mathcal{J}(\tau, \rho,\mathbf{v},  \mathrm{U})\geq  \int_\tau^{t_0} l(t, \varphi(t), \u(t) , \U(t))  dt +\mathcal{V}(t_0, \varphi(t_0), \u(t_0)).
	\end{align}
	Since $\U$ is arbitrary, taking infimum on both sides of \eqref{DP1} over $\U$ we arrive at 
	\begin{align*}
		\mathcal{V} (\tau, \rho, \mathbf{v}) \geq \inf_{\U\in \mathcal{U}_R} \left\{  \int_\tau^{t_0} l(t,\varphi(t),\mathbf{u}(t) , \mathrm{U}(t)) dt + \mathcal{V} (t_0, \varphi(t_0), \u(t_0)) \right\}.
	\end{align*}
	For the case $"\leq"$, Let $\U_1$ be a control and $(\varphi_1, \u_1 )$ be corresponding solution of the system \eqref{DP5}-\eqref{DP8} with initial data $(\tau , \rho, \v)$. For $\epsilon >0$, let $\U_2 $ be such that $\mathcal{V} (t_0, \varphi_1(t_0), \u_1(t_0)) + \epsilon \geq \mathcal{J} (t_0, \varphi_2(t_0), \u_2(t_0), \U_2)$, where $(\varphi_2, \u_2)$ is the solution of the system \eqref{DP5}-\eqref{DP8} with control $\U_2$ with initial condition $(\varphi_1(t_0), \u_1(t_0))$. Now we define a new control 
	\[
	\U(t) = 
	\begin{cases}
		\U_1(t), & \tau \leq t \leq t_0, \\
		\U_2(t), & t > t_0.
	\end{cases}
	\]
	Let $(\varphi,\u)$ be solution of the system \eqref{DP5}-\eqref{DP8} with the above control $\U$ and initial data $(\tau, \rho, \v)$, then
	\begin{align*}
		\mathcal{V} (\tau, \rho, \v) &= \int_\tau^{t_0} l(t, \varphi(t), \u(t) , \U(t))  dt +\int_{t_0}^T l(t, \varphi(t), \u(t) , \U(t))  dt + g(\varphi(T), \mathbf{u}(T) )\\
		& = \int_\tau^{t_0} l(t, \varphi_1(t), \mathbf{u}_1(t) ,  \mathrm{U}_1(t)) dt  + \int_{t_0}^T l(t,\varphi_2(t) ,\mathbf{u}_2(t), \mathrm{U}_2(t) ) dt + g(\varphi(T), \mathbf{u}(T) ) \\
		&= \int_\tau^{t_0} l(t, \varphi_1(t), \mathbf{u}_1(t) ,  \mathrm{U}_1(t)) dt  + \mathcal{J}(t_0, \varphi_1(t_0), \u_1(t_0), \U_2(t) ) \\
		& \leq \int_\tau^{t_0} l(t, \varphi_1(t), \mathbf{u}_1(t) ,  \mathrm{U}_1(t))dt +\mathcal{V} (t_0, \varphi_1(t_0), \u_1(t_0)) + \epsilon
	\end{align*}
	where $(\varphi_2, \u_2)$ is the solution of \eqref{DP5}-\eqref{DP8} corresponding to control $\U_2$. Since $\epsilon$ and $\U_1$ are arbitrary we conclude that
	\begin{align*}
		\mathcal{V} (\tau, \rho, \v) \leq \inf_{\U \in \mathcal{U}_{ad}} \left\{ \int_\tau^{t_0} l(t, \varphi(t), \mathbf{u}(t) ,  \mathrm{U}(t)) dt +\mathcal{V} (t_0, \varphi(t_0), \u(t_0))  \right\} ,
	\end{align*}
	where $(\varphi, \u)$ is the solution with control $\U$. This proves the theorem.
	
\end{proof}

\section{Hamilton Jacobi Bellman equations} \label{sec HJB}
Our main goal in this work is to prove that the value function $ \mathcal{V} $ is the unique viscosity solution of the Hamilton-Jacobi-Bellman equation \eqref{HJB1}. We tackle the uniqueness issue in the next section. 


In this section, our main aim is to show that the value function $\mathcal{V}$ defined in \eqref{valuefn} is a viscosity solution of an infinite dimensional Hamilton-Jacobi-Bellman equation:  
\begin{equation} \label{HJB1}
	\begin{aligned}
		\mathcal{V}_t -\langle B_1(\v, \rho)+B_N^2 \rho +  A_N f(\rho), \partial_\rho \mathcal{V} \rangle-\langle A\v + B(\v,\v)-B_2&(B_N \rho, \rho) ,\partial_{\v} \mathcal{V} \rangle \\
		+  H(t,\rho, \v,\partial_\rho \mathcal{V}, \partial_{\v} \mathcal{V} )&=0,  \\
		\mathcal{V}(T, \rho, \v) &= g(\rho,\v) ,
	\end{aligned}
\end{equation}
where the Hamiltonian function is given by
\begin{align}
	H(t,\rho, \v, \partial_\rho \mathcal{V}, \partial_{\v} \mathcal{V}) &= \inf_ { {{\U \in\mathcal{U}_R}} } \left(\langle \U, \partial_{\v} \mathcal{V} \rangle + l(\rho,\v, \U)\right) .  \label{Hamiltonian1}
\end{align}
The following specific example is covered by the analysis of this paper (Assumption \ref{l_assumption}). 
\begin{example}
Let us define $l$ and $g$ as follows:
\begin{align*}
	l(t, \rho, \v, \U) &=  \frac{1}{2}(\|\rho\|^2 + \|\v\|^2) + \frac{1}{2} \|\U\|^2, \\
	g(\rho, \v) &= \frac{1}{2} (\| \rho \|^2 + \| \v \|_{\V'_{\mathrm{div}}}^2).
\end{align*}    
\end{example}

For the above example, the Hamiltonian function is given by 
\begin{align}
	H(t, \rho, \v, p, \q) = \frac{1}{2} (\|\rho\|^2 + \|\v\|^2) + \underbrace{\inf_{ {{\U \in\mathcal{U}_R}} } \left( \langle \U, \q \rangle + \frac{1}{2} \|\U\|^2 \right) }_{:=\Upsilon (\q)}
\end{align}
We can explicitly obtain $ \Upsilon $ as 
\[
\Upsilon (\U)= 
\begin{cases}
	-\frac{1}{2} \|\U\|^2 & \mathrm{for} \ \| \U\|\leq R,\\
	-R \|\U\| + \frac{1}{2} R^2  & \mathrm{for } \ \| \U\|\geq R.
\end{cases}
\]
Also, the optimal feedback control is given by (see page 674, \cite{MR2141895}),
\begin{align*}
	\U^*(t) = \sigma (\q(t)), 
\end{align*} 
where
\[
\sigma (\q) = 
\begin{cases}
	-\q & \mathrm{if} \ \| \q\|\leq R,\\
	- \q \frac{ R}{\|\q \|} & \mathrm{if} \ \| \q\|\geq R,
\end{cases}
\]
and $ \q = \partial_\v \mathcal{V} $.   For the value function $ \mathcal{V} $  this translates to saying $  \q(t) = \partial_\v \mathcal{V}(t, \varphi(t), \u(t)) $, whenever the value function $ \mathcal{V} $ is  differentiable. However, the value function $ \mathcal{V} $ is not differentiable in general and the above equation \eqref{HJB1} has to be understood in a much general sense. We use the viscosity solution method to get the required result. 

In general, there is no single definition of viscosity solution that applies to all equations in an infinite dimensional setup. The definition of the viscosity solution involves maxima or minima for the test function. Hence, the main idea is to consider a test function that is convex and has coercivity or lower semi-continuity properties. The choice of a test function highly depends on the regularity of the function ( in this case solution of the PDE constraint of the control problem) and hence in turn on the type of non-linearity. It also depends on the growth properties of the solution and its relevant derivatives. In our discussion, we consider a definition of a viscosity solution motivated from \cite{MR2141895} (see also \cite{MR1918929, MR1753181}). In various other problems depending upon the properties of operators involved in PDE  one needs to consider more generalised/ modified definitions (eg. see \cite{ MR3674558, MR1741146}).

\begin{definition} \label{testfunc}
	A function $\Psi$  is a test function for \eqref{HJB1} if
	$\Psi = \xi + \delta(t)[(1+\|\nabla \rho\|^2)^m+(1+\|\v\|^2)^m], m \geq 1$
	and $\xi \in C^1((0, T)  \times D(B_N^{1/2}) \times \G_{\mathrm{div}})$ such that $A^{\frac{1}{2}} \partial_\v \xi $ and $ B_N^{\frac{1}{2}} \partial_\rho \xi $ are continuous on $(0, T)  \times D(B_N^{1/2}) \times \G_{\mathrm{div}}$, and $\delta > 0$ is a  $C^1$ function on  $(0,T)$.
\end{definition}

Now we set up our definition of the viscosity solution which will be used in our work. Note that the  Fr\'echet derivatives of the test function need to be well defined and the value function $\mathcal{V} $  is defined on  $ (0, T) \times D(B^{1/2}_N) \times \G_{\mathrm{div}} $. Hence we consider the same space while defining the viscosity solution. 

%


\begin{definition} \label{def_visco}
	A weakly sequentially upper semicontinuous function  $\mathcal{V} : (0,T) \times D(B^{1/2}_N) \times \G_{\mathrm{div}} \rightarrow \R$ is called a viscosity subsolution of \eqref{HJB1} if for every test function $\Psi $, for which  $ \mathcal{V} - \Psi$ attains global maximum at $(t, \rho, \v)$ implies $ \rho \in D(B_N^{3/2}) , \v \in \V_{\mathrm{div}}$ and 
	\begin{align} \label{HJBmax1}
		&\partial_t \Psi (t, \rho, \v)- \ _{D(B_N^{1/2})'}\langle B_1(\v, \rho)+B_N^2 \rho +  A_N f(\rho), \partial_\rho \Psi(t, \rho, \v) \rangle_{D(B_N^{1/2})} \no \\
		&-  \  _{\V_{{\mathrm{div}}}'}\langle A\v + B(\v,\v)-B_2(B_N \rho, \rho) ,\partial_{\v} \Psi(t, \rho, \v) \rangle_{\V_{{\mathrm{div}}}} 
		+H (t,\rho, \v,\partial_\rho \Psi(t, \rho, \v), \partial_\v \Psi(t, \rho, \v) ) \geq 0. 
	\end{align}
Similarly, 
a weakly sequentially lower semicontinuous function  $\mathcal{V} : (0,T) \times D(B^{1/2}_N) \times \G_{\mathrm{div}} \rightarrow \R$ is called a viscosity  supersolution of \eqref{HJB1} if for every test function $\Psi $, for which  $ \mathcal{V} +\Psi$  attains global minimum at $(t, \rho, \v)$ implies $ \rho \in D(B_N^{3/2}) , \v \in \V_{\mathrm{div}}$ and
	\begin{align} \label{HJBmin1}
		&-\partial_t \Psi (t, \rho, \v)+ \ _{D(B_N^{1/2})'}\langle B_1(\v, \rho)+B_N^2 \rho +  A_N f(\rho), \partial_\rho \Psi(t, \rho, \v) \rangle_{D(B_N^{1/2})} \no \\
		&+ \ _{\V_{{\mathrm{div}}}'}\langle A\v + B(\v,\v)-B_2(B_N \rho, \rho) ,\partial_{\v} \Psi(t, \rho, \v) \rangle_{\V_{{\mathrm{div}}}}  + H (t,\rho, \v,-\partial_\rho \Psi(t, \rho, \v), -\partial_\v \Psi(t, \rho, \v) ) \leq 0  .
	\end{align}	

 A function is called a viscosity solution if it is both a viscosity subsolution and a viscosity supersolution. 
\end{definition}

Our main theorem of this section is stated below:
\begin{theorem}\label{HJBThm}
Assume that Assumption \eqref{l_assumption} is satisfied. Then, the value function $ \mathcal{V} $ is a viscosity solution of the HJB equation \eqref{HJB1}.
\end{theorem}

We prove the theorem in three main steps.
As per the definition of viscosity sub/super solution we need to consider  global maxima of $\mathcal{V} - \Psi $ and  global minima of $\mathcal{V} + \Psi $ and evaluate the  Fr\'echet derivative of test function at these points. This requires that the point  maximum/minimum lie in a better regular space which is proved in the  first step and then in steps 2 and 3 we show the viscosity supersolution and subsolution conditions respectively.
   As in the definition 4.2 we write $\Psi = \xi + h$ where  $h =\delta(t) [(1+\|\nabla \rho\|^2)^m + (1+ \|\v\|^2)^m] .$
We observe that the Fr\'echet derivatives of $h$ with respect to $\rho$ on $D(B_N^{1/2})$ and with respect to $\v$ on $\G_{\mathrm{div}}$ are given by
\begin{align*}
	\partial_\rho h &=2m \delta(t) (1 + \|\nabla \rho\|^2)^{m-1} B_N \rho,
\end{align*}
and
\begin{align*}
	\partial_{\v} h &=  2m \delta(t) (1 + \|\v\|^2)^{m-1} \v,
\end{align*}
respectively, and the Fr\'echet derivatives of $\Psi$ on $ D(B_N^{1/2}) $ and on $\G_{{\mathrm{div}}} $ with respect to $\rho$ and $\v$ are denoted by $ \partial_\rho \Psi,  \partial_{\v} \Psi$ respectively and we write 
\begin{align*}
	\partial_\rho \Psi &= \partial_\rho \xi + 2m \delta(t) (1 + \|\nabla \rho\|^2)^{m-1} B_N \rho, \\
	\partial_{\v} \Psi &= \partial_{\v} \xi + 2m \delta(t) (1 + \|\v\|^2)^{m-1} \v. 
\end{align*}

These Fr\'echet derivatives lie in the corresponding dual spaces $D(B_N^{1/2})'$ and $\G_{\mathrm{div}}$, respectively. But note that after the proof of step 1, we see that  Fr\'echet derivatives  will lie in better spaces namely   $D(B_N^{1/2}) $ and $  \G_{{\mathrm{div}}}. $
Therefore,  we can interpret the terms $ _{D(B_N^{1/2})'}\langle B_1(\v, \rho)+B_N^2 \rho +  A_N f(\rho), \partial_\rho \Psi(t, \rho, \v) \rangle_{D(B_N^{1/2})} $ and $ _{\V_{{\mathrm{div}}}'}\langle A\v + B(\v,\v)-B_2(B_N \rho, \rho) ,\partial_{\v} \Psi(t, \rho, \v) \rangle_{\V_{{\mathrm{div}}}} $ in the sense of duality pairing. We write  $ \langle \cdot, \cdot \rangle $ in place of $ _{X'}\langle \cdot, \cdot \rangle_X $, and the duality pairing is understood contextually. Currently we can assume that the Fr\'echet derivatives  are written only formally but when we will use them in steps 2 and 3 it will be meaningful. We now proceed with the proof of the theorem in 3 steps .

\begin{proof}  
We follow the techniques from \cite{MR2141895} in the following proof. \\
{\bf Step 1:} In this step our main aim is to prove that point of minima lies in a better regular space which allows us to define test function as desired in the definition  \ref{def_visco}.  In particular  let  $ \Psi(t,\rho, \v) = \xi(t,\rho, \v) + \delta(t)[(1+\|\nabla \rho\|^2)^m+(1+\|\v\|^2)^m]$.  Let $\mathcal{V} +\Psi$ attain a global minimum at $(t_0, \rho_0. \v_0)$. 
We want to prove that  $ \rho_0 \in D(B_N^{3/2}),$ and   $ \v_0 \in \V_{{\mathrm{div}}} $.

By the dynamic programming principle we have for every $\epsilon >0$ there exists a control $\U_\epsilon \in \mathcal{U}_{ad}$ such that 
\begin{align}
&\mathcal{V} (t_0, \rho_0, \v_0) + \epsilon^2 \no \\
&> \int_{t_0}^{t_0 + \epsilon}  l( \varphi_\epsilon(s) ,\u_\epsilon(s) ,\U_\epsilon (s) )ds + \mathcal{V} (t_0+\epsilon , \varphi_\epsilon(t_0 + \epsilon), \u_\epsilon(t_0 + \epsilon)), \label{DP90}
\end{align}
where $ (\varphi_\epsilon, \u_\epsilon) $ is the corresponding solution of the system \eqref{DP5}-\eqref{DP8} with control $ \U_\epsilon $ and initial data $ (\rho_0, \v_0) $. Since $(t_0, \rho_0. \v_0)$ is a global minimum point for $ \mathcal{V} + \Psi $, we have that 
\begin{align*}
\mathcal{V} (t, \rho, \v)- \mathcal{V} (t_0, \rho_0, \v_0) \geq \Psi(t_0, \rho_0, \v_0)-\Psi(t, \rho, \v), \quad \forall (t, \rho, \v) \in (0,T) \times D(B_N^{1/2}) \times \G_{\mathrm{div}}.
\end{align*}
Then, from \eqref{DP90} we get
\begin{align*}
&\epsilon^2 - \int_{t_0}^{t_0 + \epsilon} l ( \varphi_\epsilon(s), \u_\epsilon(s) , \U_\epsilon(s) )ds \\
&\geq  \mathcal{V} (t_0+\epsilon, \varphi_\epsilon(t_0+ \epsilon), \u_\epsilon(t_0+\epsilon))-\mathcal{V} (t_0, \rho_0, \v_0)\\
&\geq \Psi(t_0, \rho_0, \v_0)-\Psi(t_0+\epsilon, \varphi_\epsilon(t_0+ \epsilon), \u_\epsilon(t_0+\epsilon)) \\
& {\geq \xi(t_0,\rho_0, \v_0)-\xi(t_0+\epsilon, \varphi_\epsilon(t_0+ \epsilon), \u_\epsilon(t_0+\epsilon)) + \delta (t_0)[(1+\|\nabla \rho_0\|^2)^m + (1+ \|\v\|^2)^m]} \\ 
&{-\delta(t_0+\epsilon)[(1+\|\nabla \varphi_\epsilon(t_0+\epsilon)\|^2)^m + (1+ \|\u_\epsilon(t_0+\epsilon)\|^2)^m]} \\
\end{align*}
By the chain rule, we get
\begin{align} \label{DP91}
\epsilon^2 - &\int_{t_0}^{t_0 + \epsilon} l ( \varphi_\epsilon(s), \u_\epsilon(s) , \U_\epsilon(s) ) ds \no \\
\geq 
&-\int_{t_0}^{t_0 + \epsilon} \left( \partial_t \xi(s, \varphi_\epsilon (s), \u_\epsilon (s))+ \langle \partial_t\varphi_\epsilon (s), \partial_\rho \xi(s, \varphi_\epsilon (s), \u_\epsilon (s))\rangle + \langle \partial_t \u_\epsilon(s), \partial_\v \xi(s, \varphi_\epsilon (s), \u_\epsilon (s))\rangle  \right) ds \no \\
& -\int_{t_0}^{t_0 + \epsilon} { \delta'(s)(1+\|\nabla \varphi_\epsilon(s)\|^2)^m} ds -\int_{t_0}^{t_0+\epsilon} 2m \delta(s)  \langle B_N \varphi_\epsilon(s), \partial_t\varphi_\epsilon(s) \rangle (1+\|\nabla \varphi_\epsilon(s)\|^2)^{m-1} ds  \no \\
& - \int_{t_0}^{t_0 + \epsilon} \delta'(s)(1+\|\u_\epsilon(s)\|^2)^m
ds -  \int_{t_0}^{t_0 + \epsilon} 2m \delta(s)  \langle \u_\epsilon(s),  \partial_t\u_\epsilon(s) \rangle (1+\|\u_\epsilon(s)\|^2)^{m-1}ds.
\end{align}
Now, setting $\lambda = \inf_{t \in [t_0, t_0+\epsilon]} \delta(t)$ for some fixed $\epsilon_0>0$ and taking $\epsilon < \epsilon_0$ we divide \eqref{DP91} by $ \epsilon $ to get,
\begin{align}
\epsilon-& \frac{1}{\epsilon}  \int_{t_0}^{t_0 + \epsilon} l ( \varphi_\epsilon(s), \u_\epsilon(s) , \U_\epsilon(s) ) ds \no \\
\geq& -\frac{1}{\epsilon} \int_{t_0}^{t_0+\epsilon} \partial_t \xi(s, \varphi_\epsilon (s), \u_\epsilon (s))ds \no \\
& + \frac{1}{\epsilon} \int_{t_0}^{t_0+\epsilon}  \left\langle B_1(\u_\epsilon(s), \varphi_\epsilon(s))+B_N^2\varphi_\epsilon(s)+A_Nf(\varphi_\epsilon(s)), \partial_\rho \xi(s, \varphi_\epsilon (s), \u_\epsilon (s)) \right\rangle ds \no \\
& + \frac{1}{\epsilon} \int_{t_0}^{t_0+\epsilon} \langle A\u_\epsilon (s) + B(\u_\epsilon(s),\u_\epsilon(s)) - B_2(B_N \varphi_\epsilon(s), \varphi_\epsilon(s))-\U_\epsilon(s), \partial_\v \xi(s, \varphi_\epsilon (s), \u_\epsilon (s)) \rangle ds \no \\
&- \frac{1}{\epsilon}\int_{t_0}^{t_0 + \epsilon}  
{\delta'(s)(1 + \|\nabla \varphi_\epsilon(s)\|^2)^m}
ds-  \frac{1}{\epsilon}\int_{t_0}^{t_0 + \epsilon}  {\delta'(s)(1+\|\u_\epsilon(s)\|^2)^m
} ds \no  \\
&+{ \frac{2m}{\epsilon}\int_{t_0}^{t_0+\epsilon} \delta(s)  \langle B_N \varphi_\epsilon(s), B_1(\u_\epsilon(s), \varphi_\epsilon(s))+B_N^2\varphi_\epsilon(s)+A_Nf(\varphi_\epsilon(s))\rangle (1 + \|\nabla \varphi_\epsilon(s)\|^2)^{m-1}ds} \no \\
&+ {  \frac{2m}{\epsilon} \int_{t_0}^{t_0 + \epsilon}  \delta(s)  \langle \u_\epsilon(s), A\u_\epsilon (s) + B(\u_\epsilon(s),\u_\epsilon(s)) - B_2(B_N \varphi_\epsilon(s), \varphi_\epsilon(s))-\U_\epsilon(s) \rangle (1+\|\u_\epsilon(s)\|^2)^{m-1} ds} \label{DP145}
.
\end{align}
Rearranging above inequality and taking the definition of $ \lambda $ into account we get
\begin{align} \label{DP96}
{ \frac{2m\lambda}{\epsilon}} & \int_{t_0}^{t_0 + \epsilon} { (1 +\| \u_\epsilon(s)\|^2 )^{m-1} \|\nabla \u_\epsilon(s)\|^2 } ds+ \frac{2m\lambda}{\epsilon} \int_{t_0}^{t_0 + \epsilon} {(1+ \|\nabla \varphi_{\epsilon}(s)\|^2)^{m-1} \| B_N^{3/2} \varphi_\epsilon(s) \|^2 } ds  \no \\
\leq& \epsilon- \frac{1}{\epsilon}  \int_{t_0}^{t_0 + \epsilon} l \left( \varphi_\epsilon(s), \u_\epsilon(s) ,\U_\epsilon(s)  \right)ds + \frac{1}{\epsilon} \int_{t_0}^{t_0+\epsilon} \partial_t \xi(s, \varphi_\epsilon (s), \u_\epsilon (s))ds \no \\
& -\frac{1}{\epsilon} \int_{t_0}^{t_0 + \epsilon} \left\langle B_1(\u_\epsilon(s), \varphi_\epsilon(s))+B_N^2\varphi_\epsilon(s)+A_Nf(\varphi_\epsilon(s)), \partial_\rho \xi(s, \varphi_\epsilon (s), \u_\epsilon (s)) \right\rangle ds \no \\
& -\frac{1}{\epsilon}  \int_{t_0}^{t_0 + \epsilon} \langle A\u_\epsilon (s) + B(\u_\epsilon(s),\u_\epsilon(s)) - B_2(B_N \varphi_\epsilon(s), \varphi_\epsilon(s))-\U_\epsilon(s), \partial_\v \xi(s, \varphi_\epsilon (s), \u_\epsilon (s)) \rangle ds \no \\
&+ \frac{1}{\epsilon}\int_{t_0}^{t_0 + \epsilon} { \delta'(s)(1+\|\nabla \varphi_\epsilon(s)\|^2)^m } ds +  \frac{1}{\epsilon}\int_{t_0}^{t_0 + \epsilon}  { \delta'(s)(1+ \|\u_\epsilon(s)\|^2)^m } ds \no \\
& -\frac{2}{\epsilon} \int_{t_0}^{t_0 + \epsilon} {m  \delta(s) (1+\|\nabla \varphi_\epsilon(s)\|^2)^{m-1}} \langle B_N \varphi_\epsilon(s), B_1(\u_\epsilon(s), \varphi_\epsilon(s))+A_Nf(\varphi_\epsilon(s))\rangle ds \no  \\
& +\frac{2}{\epsilon} \int_{t_0}^{t_0+\epsilon} {m \delta(s) (1+\|\u_\epsilon (s)\|^2)^{m-1} }  \langle \u_\epsilon(s),  B_2(B_N \varphi_\epsilon(s), \varphi_\epsilon(s))+\U_\epsilon(s) \rangle ds.
\end{align}
In step 2, from \eqref{DP96} we will deduce \eqref{HJBmin1} by passing to the limit as $ \epsilon \rightarrow 0 $. But we need to estimate the regularity first.
Let us denote by $ C = C(\|\nabla \rho_0\|, \|\v_0\|, \delta,f, R) $  a generic constant.  Observe that, 
from the assumptions on $\xi$, we have that
\begin{align}
&\|B_N ^{1/2} \partial_\rho \xi (s, \varphi_\epsilon(s), \u_\epsilon(s))\|\leq C, \label{DP93}\\
&\| A^{1/2} \partial_\v \xi(s, \varphi_\epsilon (s), \u_\epsilon (s))\| \leq C \label{DP92},
\end{align}
for $\epsilon < \epsilon_0$ for some fixed $\epsilon_0.$
Therefore we get from \eqref{DP93}, \eqref{DP95} and Theorem \ref{existance}
\begin{align} \label{DP106}
|\langle B_1(\u_\epsilon(s), \varphi_\epsilon(s)) , &\partial_\rho \xi (s, \varphi_\epsilon(s), \u_\epsilon(s)) \rangle| \no \\
&\leq C \|\u_\epsilon(s)\|^{1/2}\|\nabla \u_\epsilon(s)\|^{1/2} \|\nabla \varphi_\epsilon(s) \| ^{1/2}\|B_N \varphi_\epsilon(s)\|^{1/2}\| \partial_\rho \xi (s, \varphi_\epsilon(s), \u_\epsilon(s)) \| \no \\
&\leq  C\|\nabla \u_\epsilon(s)\| \|\nabla \varphi_\epsilon (s)\|^{1/2}\|B_N^{3/2} \varphi_\epsilon(s)\|^{1/2}\| \partial_\rho \xi (s, \varphi_\epsilon(s), \u_\epsilon(s)) \|\no  \\
& \leq \frac{\lambda}{4} \|\nabla \u_\epsilon(s)\| ^2 + C(\|\nabla \rho_0\|, \|\v_0\|) \|B_N^{3/2} \varphi_\epsilon(s)\| \no \\
&\leq \frac{\lambda}{4} \|\nabla \u_\epsilon(s)\| ^2 + \frac{\lambda}{7}\|B_N^{3/2} \varphi_\epsilon(s)\|^2 + C,
\end{align}
\begin{align}
|\langle B_N^2 \varphi_\epsilon(s),  \partial_\rho \xi (s, \varphi_\epsilon(s), \u_\epsilon(s))  \rangle| &= |\langle B_N^{3/2} \varphi_\epsilon(s), B_N^{1/2} \partial_\rho \xi (s, \varphi_\epsilon(s), \u_\epsilon(s))   \rangle| \no \\
& \leq  \| B_N^{3/2} \varphi_\epsilon(s)\|\| B_N^{1/2} \partial_\rho \xi (s, \varphi_\epsilon(s), \u_\epsilon(s))\| \no \\
& \leq \frac{\lambda}{7} \| B_N^{3/2} \varphi_\epsilon(s)\|^2 + C.
\end{align}
Moreover, using  \eqref{f condition2}, we have
\begin{align} \label{DP147}
   |\langle A_N f( \varphi_\epsilon(s)),  & \partial_\rho \xi (s, \varphi_\epsilon(s), \u_\epsilon(s))  \rangle| \no \\
   &= |\langle A_N^{1/2} f( \varphi_\epsilon(s)), B_N^{1/2} \partial_\rho \xi (s, \varphi_\epsilon(s), \u_\epsilon(s))   \rangle| \no  \\
   & \leq C_f| ((1+|\varphi_\epsilon(s)|^r ) \nabla \varphi_\epsilon (s) , B_N^{1/2} \partial_\rho \xi (s, \varphi_\epsilon(s), \u_\epsilon(s))) | \no \\
   &\leq C \|\varphi_\epsilon (s) \|_{L^{4r}}^r  \|\nabla \varphi_\epsilon(s)\|_{L^4} \|  B_N^{1/2} \partial_\rho \xi (s, \varphi_\epsilon(s), \u_\epsilon(s))\| \no \\
   &\leq C \|\nabla \varphi_\epsilon(s)\|^{\frac{2r-1}{2}} \|\varphi_\epsilon(s)\|^\frac{1}{2} \| \nabla \varphi_\epsilon(s)\|^\frac{1}{2} \|B_N^{3/2} \varphi_\epsilon(s) \|^\frac{1}{2}\|  B_N^{1/2} \partial_\rho \xi (s, \varphi_\epsilon(s), \u_\epsilon(s))\| \no \\
   &\leq \frac{\lambda}{7} \|B_N^{3/2} \varphi_\epsilon(s) \|^2 + C 
\end{align}
Using \eqref{DP92}, Theorem \ref{existance} and Theorem \ref{strongsol}, we have
\begin{align}
|\langle A \u_\epsilon(s), \partial_\v \xi (s, \varphi_\epsilon(s), \u_\epsilon(s)) \rangle| &= |\langle A^{1/2} \u_\epsilon(s), A^{1/2}\partial_\v \xi (s, \varphi_\epsilon(s), \u_\epsilon(s))  \rangle| \no \\
& \leq  \|A^{1/2} \u_\epsilon(s)\| \| A^{1/2}\partial_\v \xi (s, \varphi_\epsilon(s), \u_\epsilon(s))  \| \no \\
& \leq  \frac{\lambda}{4} \|\nabla \u_\epsilon\| ^2 + C,
\end{align}
\begin{align}
|\langle B(\u_\epsilon(s), \u_\epsilon(s)), \partial_\v \xi (s, \varphi_\epsilon(s), \u_\epsilon(s)) \rangle| &= | b(\u_\epsilon(s), \u_\epsilon(s), \partial_\v \xi (s, \varphi_\epsilon(s), \u_\epsilon(s)))| \no \\
& \leq \|\u_\epsilon(s)\| \|\nabla \u_\epsilon(s)\| \|A^{1/2}\partial_\v \xi (s, \varphi_\epsilon(s), \u_\epsilon(s))\| \no \\
& \leq \frac{\lambda}{4} \|\nabla \u_\epsilon\| ^2 + C,
\end{align}
and
\begin{align}
|\langle B_2(B_N \varphi_\epsilon(s), \varphi_\epsilon(s)), \partial_\v \xi (s, \varphi_\epsilon(s), \u_\epsilon(s)) \rangle| 
&\leq \| B_N \varphi_\epsilon(s)\|_{L^4}\|\nabla  \varphi_\epsilon(s))\| \| \partial_\v \xi (s, \varphi_\epsilon(s), \u_\epsilon(s)) \|_{\mathbb{L}^4} \no \\
&\leq C\| B_N \varphi_\epsilon(s)\|^{1/2}\| B_N^{3/2} \varphi_\epsilon(s)\|^{1/2}\no \\
& \leq \frac{\lambda}{7}\| B_N^{3/2} \varphi_\epsilon(s)\|^2 + C.
\end{align}
We also have
\begin{align}
\frac{1}{\epsilon}  \int_{t_0}^{t_0 + \epsilon}  l ( \varphi_\epsilon(s),  \u_\epsilon(s), \U_\epsilon (s)) ds \leq \frac{C}{\epsilon}  \int_{t_0}^{t_0 + \epsilon}  (1+\|\nabla \varphi_\epsilon(s)\|^k + \| \u_\epsilon(s)\| ^k )ds \leq C .
\end{align}
and
\begin{align}
|\langle \U_\epsilon(s), \partial_\v \xi (s, \varphi_\epsilon(s), \u_\epsilon(s)) \rangle| \leq C.
\end{align}
Similarly, since $\delta \in  C^1([t_0, t_0+\epsilon_0])$, by Theorem \ref{existance} and Theorem \ref{strongsol}, we have for $ s \in [t_0 , t_0 + \epsilon] $
\begin{align}
|{m  \delta(s) (1+\|\nabla \varphi_\epsilon(s)\|^2)^{m-1}}\langle B_N \varphi_\epsilon(s),& B_1(\u_\epsilon(s), \varphi_\epsilon(s))\rangle| \no \\
&\leq C\| B_N \varphi_\epsilon\|_{\L^4} \|\u_\epsilon\|_{\L^4} \|\nabla \varphi_\epsilon \| \no \\
& \leq C\| B_N \varphi_\epsilon\|^{1/2} \|B_N^{3/2} \varphi_\epsilon\|^{1/2}\|\nabla \u_\epsilon\| \|\nabla \varphi_\epsilon \| \no \\
& \leq \frac{\lambda}{7} \| B_N^{3/2} \varphi_\epsilon(s)\|^2 + C,
\end{align}
Similar to the \eqref{DP147}
\begin{align}
|m  \delta(s) (1+\|\nabla \varphi_\epsilon(s)\|^2)^{m-1} &\langle B_N\varphi_\epsilon (s), A_N f(\varphi_\epsilon(s))\rangle|\no \\ &= |m \delta(s)(1+\|\nabla \varphi_\epsilon(s)\|^2)^{m-1} \langle B_N^{3/2} \varphi_\epsilon(s), A_N^{1/2}f(\varphi_\epsilon(s)) \rangle|  \no \\
&\leq m \delta (s) (1+\|\nabla \varphi_\epsilon(s)\|^2)^{m-1} \|B_N ^{3/2} \varphi_\epsilon(s)\| \|(1+|\varphi_\epsilon(s)|^r \|_{L^4} \|\nabla \varphi_\epsilon (s)\|_{L^4} \no \\
&\leq  \frac{\lambda}{7} \| B_N^{3/2} \varphi_\epsilon(s)\|^2 + C
\end{align}
\begin{align} 
|{m  \delta(s) (1+\| \u_\epsilon(s)\|^2)^{m-1}} \langle \u_\epsilon(s), B_2(B_N\varphi_\epsilon(s), \varphi_\epsilon(s)) \rangle| &\leq C \|\u_\epsilon(s)\|_{L^4} \|B_N\varphi_\epsilon(s)\|_{L^4} \|\nabla \varphi_\epsilon(s)\|\no \\
&\leq C \|\u\|^{1/2}\|\nabla \u_\epsilon(s)\|^{1/2} \|B_N^{3/2}\varphi_\epsilon(s)\| \|\nabla \varphi_\epsilon(s)\|  \no  \\
&\leq  \frac{\lambda}{7} \|B_N^{3/2} \varphi_\epsilon(s) \|^2 + \frac{\lambda}{4} \|\nabla \u_\epsilon(s)\|^2 + C  ,
\end{align}
and
\begin{align} \label{DP107}
|{m  \delta(s) (1+\| \u_\epsilon(s)\|^2)^{m-1}}\langle \u_\epsilon(s), \U_\epsilon(s) \rangle| \leq C\| \u_\epsilon\| \|\U_\epsilon\| \leq C.
\end{align}
Substituting \eqref{DP106}-\eqref{DP107} in \eqref{DP96} we get that 
\begin{align} \label{DP112}
\frac{\lambda}{\epsilon} &\int_{t_0}^{t_0 + \epsilon} \|\nabla \u_\epsilon(s)\|^2 ds + \frac{\lambda}{\epsilon} \int_{t_0}^{t_0 + \epsilon}  \|B_N^{3/2} \varphi_\epsilon(s)\|^2 ds \leq C,
\end{align}
where $C$ is independent of $\epsilon$. Therefore there exist sequences $\{\epsilon_n\}$ and $\{t_n\}$ such that $\epsilon_n \rightarrow 0$ and $t_n \in (t_0, t_0 + \epsilon_n)$ and 
\begin{align*}
\|\nabla \u_\epsilon(t_n)\|^2 + \|\B_N^{3/2}\varphi_\epsilon(t_n)\|^2 \leq C.
\end{align*} 
Thus, there exist $\bar{\v} \in \V_{\mathrm{div}}, \bar{\rho} \in D(B_N^{3/2})$ and a subsequence $\{t_n \}$ along which, we have 
\begin{align*}
&\u_{\epsilon_n}(t_n) \rightarrow \bar{\v} \quad \mathrm{weakly \ in} \ \V_{\mathrm{div}}, \\
&\varphi_{\epsilon_n}(t_n) \rightarrow \bar{\rho} \quad \mathrm{weakly \ in} \ D(B_N^{3/2}).
\end{align*}  
This also implies that 
\begin{align*}
&\u_{\epsilon_n}(t_n) \rightarrow \bar{\v} \quad \mathrm{weakly \ in} \ \G_{\mathrm{div}}, \\
&\varphi_{\epsilon_n}(t_n) \rightarrow \bar{\rho} \quad \mathrm{weakly \ in} \ D(B_N).
\end{align*}  
However, by \eqref{DP74}, \eqref{DP78}, we  have 
\begin{align*}
&\u_{\epsilon_n}(t_n) \rightarrow \v_0 \quad \mathrm{strongly \ in} \ \G_{\mathrm{div}}, \\
&\varphi_{\epsilon_n}(t_n) \rightarrow \rho_0 \quad \mathrm{strongly \ in} \ D(B_N).
\end{align*}
Therefore, by uniqueness of weak limit in $\G_{\mathrm{div}}$ and $ D(B_N)$ it follows that $\v_0 = \bar{\v} \ \in \V_{\mathrm{div}} $, $\rho_0 = \bar{\rho} \in D(B_N^{3/2})$.

{\bf Step 2:} In this step, we show the supersolution inequality. That is, we need to pass to the limit as $\epsilon \rightarrow 0$ in \eqref{DP96} along a subsequence and get 
\begin{align} \label{DP97}
-(&\partial_t  \Psi (t_0, \rho_0, \v_0)-\langle B_1(\v_0, \rho_0)+B_N^2 \rho_0+  A_N f(\rho_0), \partial_\rho \Psi(t_0, \rho_0, \v_0) \rangle \no \\
&-\langle A\v_0 + B(\v_0,\v_0)-B_2(B_N \rho_0, \rho_0) ,\partial_{\v} \Psi(t_0, \rho_0, \v_0) \rangle) \no \\
&+ H (t_0,\rho_0, \v_0,-\partial_\rho \Psi(t_0, \rho_0, \v_0), -\partial_\v \Psi(t_0, \rho_0, \v_0) ) \leq 0.
\end{align}
To prove \eqref{DP97}, we need to estimate the terms in \eqref{DP145}. Let  $g(\epsilon)$ denote a generic modulus of continuity function such that $g(\epsilon)\rightarrow 0 $ as $\epsilon \rightarrow 0$, which may vary within the calculation. Using \eqref{DP127}  and continuity property of $\xi$ we get
\begin{align}
&\Bigg| \frac{1}{\epsilon} \int_{t_0}^{t_0 + \epsilon}\left( \partial_t \xi (s, \varphi_\epsilon (s), \u_\epsilon (s)) - \partial_t \xi (t_0, \rho_0, \v_0) \right) ds \Bigg| \no \\
& \leq \frac{1}{\epsilon} \int_{t_0}^{t_0 + \epsilon}g(|s-t_0|+ \| \varphi_\epsilon(s)-\rho_0\|_{D(B_N^{1/2})} + \|\u_\epsilon (s)-\v_0 \| )ds \leq  g(\epsilon). \label{DP84}
\end{align}
Observe that
\begin{align*}
|\langle B_1(\u_\epsilon(s), \varphi_\epsilon(s)),\partial_\rho \xi &(s, \varphi_\epsilon (s), \u_\epsilon (s)) \rangle-\langle B_1(\v_0,\rho_0),\partial_\rho \xi (t_0, \rho_0, \v_0 )\rangle | \\
\leq & | \langle B_1 ( \u_\epsilon(s) - \v_0, \varphi_\epsilon(s) ) , \partial_\rho \xi (s, \varphi_\epsilon (s), \u_\epsilon (s))\rangle | \\
&+|\langle B_1(\v_0, \varphi_\epsilon(s)-\rho_0),\partial_\rho \xi (s, \varphi_\epsilon (s), \u_\epsilon (s)) \rangle | \\
&+|\langle B_1(\v_0,\rho_0),\partial_\rho \xi (s, \varphi_\epsilon (s), \u_\epsilon (s)) -\partial_\rho \xi (t_0, \rho_0, \v_0 )\rangle|.
\end{align*}
Using H\"older's and Gagliardo-Nirenberg inequalities we get 
\begin{align*}
|\langle B_1(\u_\epsilon(s)- \v_0, \varphi_\epsilon(s), & \partial_\rho \xi (s, \varphi_\epsilon (s), \u_\epsilon (s))\rangle | \leq \|\u_\epsilon(s)- \v_0 \|_{\mathbb{L}^4} \|\nabla \varphi_\epsilon(s)) \|_{\L^4} \| \partial_\rho \xi (s, \varphi_\epsilon (s), \u_\epsilon (s)) \|\\
&  \leq C\|\nabla(\u_\epsilon(s)- \v_0 )\| \|\nabla \varphi_\epsilon(s)) \|^{1/2} \| B_N \varphi_\epsilon(s)\|^{1/2} \| \partial_\rho \xi (s, \varphi_\epsilon (s), \u_\epsilon (s)) \|,
\end{align*}
\begin{align*}
|\langle B_1(\v_0, \varphi_\epsilon(s)-\rho_0),\partial_\rho \xi (s, \varphi_\epsilon (s), \u_\epsilon (s)) \rangle | \leq \|\v_0\|_{\mathbb{L}^4}\| \nabla (\varphi_\epsilon(s)-\rho_0)\|_{\L^4} \|\partial_\rho \xi (s, \varphi_\epsilon (s), \u_\epsilon (s)) \|\\
\leq C \|\nabla \v_0\| \|\nabla (\varphi_\epsilon(s)-\rho_0)\|^{1/2}\|B_N(\varphi_\epsilon(s)-\rho_0)\|^{1/2} \|\partial_\rho \xi (s, \varphi_\epsilon (s), \u_\epsilon (s)) \|,
\end{align*}
\begin{align*}
| \langle B_1(\v_0,\rho_0)&,\partial_\rho \xi (s, \varphi_\epsilon (s), \u_\epsilon (s)) -\partial_\rho \xi (t_0, \rho_0, \v_0 )\rangle| \\
&\leq \|\v_0\|_{\mathbb{L}^4}\|\nabla \rho_0\|_{\L^4} \| \partial_\rho \xi (s, \varphi_\epsilon (s), \u_\epsilon (s)) -\partial_\rho \xi (t_0, \rho_0, \v_0 )\| \\
&\leq C \|\nabla \v_0\| \|\nabla \rho_0\|^{1/2}\|B_N \rho_0\|^{1/2} g(|s-t_0|+ \| \varphi_\epsilon(s)-\rho_0\|_{D(B_N^{1/2})} + \|\u_\epsilon (s)-\v_0 \| ).
\end{align*}
Using  \eqref{DP127}, \eqref{DP78} and continuity properties of $\partial_\rho \xi$, we get
\begin{align}
|\frac{1}{\epsilon} \int_{t_0}^{t_0 + \epsilon}\langle B_1(\u_\epsilon(s), \varphi_\epsilon(s)),\partial_\rho \xi (s, \varphi_\epsilon (s), \u_\epsilon (s)) \rangle-\langle B_1(\v_0,\rho_0),\partial_\rho \xi (t_0, \rho_0, \v_0 )\rangle| ds\leq  g(\epsilon).
\end{align}
Similarly, we get
\begin{align*}
|\langle  A_N f(\varphi_\epsilon(s)), &\partial_\rho \xi(s, \varphi_\epsilon (s), \u_\epsilon (s)) \rangle - \langle A_N f(\rho_0),\partial_\rho \xi (t_0, \rho_0, \v_0 ) \rangle| \no \\
\leq &|\langle A_N f(\varphi_\epsilon(s)) - A_N f(\rho_0), \partial_\rho \xi(s, \varphi_\epsilon (s), \u_\epsilon (s)) \rangle | \\
&+ | \langle A_N f(\rho_0),\partial_\rho \xi(s, \varphi_\epsilon (s), \u_\epsilon (s))-\partial_\rho \xi (t_0, \rho_0, \v_0 )  \rangle|,
\end{align*}
Using \eqref{f condition2} and \eqref{DP93}, we get
\begin{align}
|\langle A_N f&(\varphi_\epsilon(s)) - A_N f(\rho_0), \partial_\rho \xi(s, \varphi_\epsilon (s), \u_\epsilon (s)) \rangle |  \no \\
&= |\langle A_N^{\frac{1}{2}} (f(\varphi_\epsilon(s)) -  f(\rho_0)), B_N^{\frac{1}{2}}\partial_\rho \xi(s, \varphi_\epsilon (s), \u_\epsilon (s)) \rangle | \no \\
& \leq \| A_N^{\frac{1}{2}} (f(\varphi_\epsilon(s)) -  f(\rho_0)) \| \|B_N^{\frac{1}{2}}\partial_\rho \xi(s, \varphi_\epsilon (s), \u_\epsilon (s)) \| \no \\
& \leq   C (1 + \| \varphi_\epsilon (s)\|^{r-1}_{H^2} + \| \rho_0 \|^{r-1}_{H^2}) (\| \varphi_\epsilon (s)\|_{H^2} + \| \rho_0 \|_{H^2}) \|\varphi_\epsilon(s)-\rho_0\|_{H^1} \no \\
& \quad + C (1 + \| \varphi_\epsilon (s)\|_{H^2} + \| \rho_0 \|_{H^2}) \|\varphi_\epsilon(s)-\rho_0\|_{H^1} , 
\end{align}
and 
\begin{align}
| \langle A_N f(\rho_0), \partial_\rho \xi(s, \varphi_\epsilon (s),& \u_\epsilon (s))-\partial_\rho  \Psi (t_0, \rho_0, \v_0 )  \rangle| \no \\
&= |\langle A_N^{\frac{1}{2}} f(\rho_0),B_N^{\frac{1}{2}}(\partial_\rho \xi(s, \varphi_\epsilon (s), \u_\epsilon (s))-\partial_\rho \xi (t_0, \rho_0, \v_0 )  )\rangle|\no \\
&\leq \|A_N^{\frac{1}{2}} f(\rho_0) \| \| B_N^{\frac{1}{2}}(\partial_\rho \xi(s, \varphi_\epsilon (s), \u_\epsilon (s))-\partial_\rho \xi (t_0, \rho_0, \v_0 )) \| \no \\
& \leq C (1+ \|\rho_0\|_{H^1}^r) \|\rho_0\|_{H^2} g(|s-t_0|+ \| \varphi_\epsilon(s)-\rho_0\|_{D(B_N^{1/2})} + \|\u_\epsilon (s)-\v_0 \| ).
\end{align}
Hence, we get
\begin{align}
\frac{1}{\epsilon} \int_{t_0}^{t_0+\epsilon}|\langle A_N f(\varphi_\epsilon(s)), &\partial_\rho \xi(s, \varphi_\epsilon (s), \u_\epsilon (s)) \rangle - \langle A_N f(\rho_0),\partial_\rho \xi (t_0, \rho_0, \v_0 ) \rangle|ds \leq  g(\epsilon).
\end{align}
Using the continuity of $A^\frac{1}{2} \partial_\v \xi$, we estimate the following, 
\begin{align*}
&|\langle A\u_\epsilon (s) ,\partial_\v \xi (s, \varphi_\epsilon (s), \u_\epsilon (s) \rangle  - \langle A\v_0, \partial_\v \xi (t_0, \rho_0, \v_0\rangle | \\ 
& \leq |\langle A(\u_\epsilon (s) -\v_0),\partial_\v \xi (s, \varphi_\epsilon (s), \u_\epsilon (s) \rangle| + | \langle A\v_0,\partial_\v \xi (s, \varphi_\epsilon (s), \u_\epsilon (s)- \partial_\v \xi (t_0, \rho_0, \v_0\rangle |  \\
& \leq |\langle A^\frac{1}{2}(\u_\epsilon (s) -\v_0),A^\frac{1}{2}\partial_\v \xi (s, \varphi_\epsilon (s), \u_\epsilon (s) )\rangle| + | \langle A^\frac{1}{2}\v_0, A^\frac{1}{2} (\partial_\v \xi (s, \varphi_\epsilon (s), \u_\epsilon (s)- \partial_\v \xi(t_0, \rho_0, \v_0) ) \rangle |  \\
& \leq \| \nabla  (\u_\epsilon (s) -\v_0)\| \|A^\frac{1}{2}\partial_\v \xi (s, \varphi_\epsilon (s), \u_\epsilon (s))  \| + \|\nabla \v_0\| \|A^\frac{1}{2} (\partial_\v \xi (s, \varphi_\epsilon (s), \u_\epsilon (s)- \partial_\v \xi (t_0, \rho_0, \v_0) )\| \\
&\leq C \| \nabla  (\u_\epsilon (s) -\v_0)\| + \|\nabla \v_0\| g (|s-t_0| + \|\varphi_\epsilon(s)-\rho_0\| +\|\u_\epsilon(s)-\v_0\|) ,
\end{align*}
which, using \eqref{DP127} and \eqref{DP78} implies
\begin{align}
\Bigg| \frac{1}{\epsilon} \int_{t_0}^{t_0 + \epsilon}(\langle A\u_\epsilon (s) ,\partial_\v \xi (s, \varphi_\epsilon (s), \u_\epsilon (s) \rangle  - \langle A\v_0, \partial_\v \xi (t_0, \rho_0, \v_0\rangle) ds \Bigg| \leq g(\epsilon).
\end{align}
Now consider the trilinear form
\begin{align*}
&|\langle B(\u_\epsilon(s),\u_\epsilon(s)), \partial_\v \xi (s, \varphi_\epsilon (s), \u_\epsilon (s)) \rangle -\langle B(\v_0,\v_0), \partial_\v \xi (t_0,\rho_0,\v_0) \rangle | \\
&\leq |\langle B(\u_\epsilon(s)-\v_0,\u_\epsilon(s)), \partial_\v \xi (s, \varphi_\epsilon (s), \u_\epsilon (s)) \rangle| + |\langle B(\v_0,\u_\epsilon(s)- \v_0), \partial_\v \xi (s, \varphi_\epsilon (s), \u_\epsilon (s)) \rangle | \\
& + |\langle B(\v_0,\v_0), \partial_\v \xi (s, \varphi_\epsilon (s), \u_\epsilon (s)) -\partial_\v \xi (t_0,\rho_0,\v_0)\rangle |
\end{align*}
We observe that 
\begin{align*}
|\langle B(\u_\epsilon(s)-\v_0,\u_\epsilon(s)), \partial_\v \xi (s, \varphi_\epsilon (s), \u_\epsilon (s)) \rangle|  \leq C \|\nabla(\u_\epsilon(s)-\v_0) \| \|\nabla \u_\epsilon(s)\|\| A^{\frac{1}{2}} \partial_\v \xi (s, \varphi_\epsilon (s), \u_\epsilon (s))\|,
\end{align*}
\begin{align*}
|\langle B(\v_0,\u_\epsilon(s)- \v_0), \partial_\v \xi (s, \varphi_\epsilon (s), \u_\epsilon (s)) \rangle | \leq C\|\nabla \v_0\| \|\nabla(\u_\epsilon(s)-\v_0) \| \| A^{\frac{1}{2}} \partial_\v \xi (s, \varphi_\epsilon (s), \u_\epsilon (s))\|,
\end{align*}
and
\begin{align*}
|\langle B(\v_0,\v_0), &\partial_\v \xi (s, \varphi_\epsilon (s), \u_\epsilon (s)) -\partial_\v \xi (t_0,\rho_0,\v_0)\rangle | \\
&\leq C \|\v_0\| \|\nabla \v_0\|\|A^{\frac{1}{2}} \left( \partial_\v \xi (s, \varphi_\epsilon (s), \u_\epsilon (s)) -\partial_\v \xi (t_0,\rho_0,\v_0) \right)\| \\
& \leq C \|\v_0\| \|\nabla \v_0\| g(|s-t_0| + \|\varphi_\epsilon(s) - \rho_0)\| \| \u_\epsilon(s)-\v_0\|).
\end{align*}
It follows from \eqref{DP127} and \eqref{DP78} that
\begin{align}
\left| \frac{1}{\epsilon} \int_{t_0}^{t_0 + \epsilon}\langle B(\u_\epsilon(s),\u_\epsilon(s)), \partial_\v \xi (s, \varphi_\epsilon (s), \u_\epsilon (s)) \rangle -\langle B(\v_0,\v_0), \partial_\v \xi (t_0,\rho_0,\v_0) \rangle \right| \leq g(\epsilon).
\end{align}
Now consider
\begin{align*}
|\langle &B_2 (B_N \varphi_\epsilon(s), \varphi_\epsilon(s)), \partial_\v \xi (s, \varphi_\epsilon (s), \u_\epsilon (s)) \rangle - \langle B_2(B_N \rho_0, \rho_0),\partial_\v \xi (t_0, \rho_0, \v_0) \rangle| \\
\leq& |\langle B_2(B_N (\varphi_\epsilon(s)-\rho_0), \varphi_\epsilon(s)), \partial_\v \xi(s, \varphi_\epsilon (s), \u_\epsilon (s))  \rangle|+ |\langle B_2(B_N  \rho_0,\varphi_\epsilon(s)-\rho_0), \partial_\v \xi(s, \varphi_\epsilon (s), \u_\epsilon (s))\rangle| \\
& + |\langle B_2(B_N  \rho_0,\rho_0), \partial_\v \xi(s, \varphi_\epsilon (s), \u_\epsilon (s))-\partial_\v \xi(t_0, \rho_0, \v_0)  \rangle|.
\end{align*}
From H\"older's and Gagliardo-Nirenberg inequality we obtain
\begin{align*}
|\langle B_2(B_N (\varphi_\epsilon(s)-\rho_0)&, \varphi_\epsilon(s)), \partial_\v \xi(s, \varphi_\epsilon (s),  \u_\epsilon (s))  \rangle|\\
&\leq \|B_N (\varphi_\epsilon(s)-\rho_0) \| \|\nabla \varphi_\epsilon(s) \|_{L^4} \| \partial_\v \xi(s, \varphi_\epsilon (s), \u_\epsilon (s))\|_{\mathbb{L}^4} \\ 
& \leq \|B_N (\varphi_\epsilon(s)-\rho_0) \| \|\varphi_\epsilon(s) \|^{1/2}_{H^1} \| \varphi_\epsilon(s)\|_{H^2}^{1/2} \|\partial_\v \xi (s, \varphi_\epsilon (s), \u_\epsilon (s)) \|_{\V},
\end{align*}
\begin{align*}
|\langle B_2(B_N  &\rho_0,\varphi_\epsilon(s)-\rho_0), \partial_\v \xi(s, \varphi_\epsilon (s), \u_\epsilon (s)) \rangle|\\
&\leq \|B_N  \rho_0\| \|\nabla ( \varphi_\epsilon(s)-\rho_0)\|_{L^4} \| \partial_\v \xi (s, \varphi_\epsilon (s), \u_\epsilon (s)\|_{L^4} \\
&\leq \|B_N  \rho_0\| \|\nabla ( \varphi_\epsilon(s)-\rho_0)\|^{1/2} \| \|B_N(\varphi_\epsilon(s)-\rho_0)\|^{1/2} \| \partial_\v \xi (s, \varphi_\epsilon (s), \u_\epsilon (s)\|_{\V},
\end{align*}
and
\begin{align*}
&|\langle B_2(B_N  \rho_0,\rho_0), \partial_\v \xi (s, \varphi_\epsilon (s), \u_\epsilon (s))-\partial_\v \xi (t_0, \rho_0, \v_0)  \rangle| \\
&\leq \|B_N  \rho_0 \| \|\nabla \rho_0\|_{L^4} \|\partial_\v \xi(s, \varphi_\epsilon (s), \u_\epsilon (s))-\partial_\v \xi (t_0, \rho_0, \v_0)  \|_{L^4}\\
&\leq \|B_N  \rho_0 \| \|\nabla \rho_0\|^{1/2}\|B_N \rho_0\|^{1/2} g(|s-t_0|+ \| \varphi_\epsilon(s)-\rho_0\|_{D(B_N^{1/2})} + \|\u_\epsilon (s)-\v_0 \| ).
\end{align*}
Combining above three estimates and using \eqref{DP127}, \eqref{DP78}, we get 
\begin{align} 
|\frac{1}{\epsilon} \int_{t_0}^{t_0 + \epsilon}\langle B_2& (B_N \varphi_\epsilon(s), \varphi_\epsilon(s)), \partial_\v \xi (s, \varphi_\epsilon (s), \u_\epsilon (s)) \rangle - \langle B_2(B_N \rho_0, \rho_0),\partial_\v \xi (t_0, \rho_0, \v_0) \rangle|  \leq C g(\epsilon). \label{DP85}
\end{align}
Now we estimate the terms in \eqref{DP96} involving $\delta$. We can estimate
\begin{align}
&\left| \frac{2m}{\epsilon} \int_{t_0}^{t_0 + \epsilon} \delta(s) (1 +\| \u_\epsilon(s)\|^2 )^{m-1} \langle A\u_\epsilon(s), \u_\epsilon(s)\rangle - \delta(t_0) (1 +\| \v\|^2 )^{m-1}\langle A\v_0, \v_0\rangle ) ds\right| \no \\
& \leq  \frac{2m}{\epsilon} \int_{t_0}^{t_0 + \epsilon} |\delta (s) (1 +\| \u_\epsilon(s)\|^2 )^{m-1} \langle A^{\frac{1}{2}} \u_\epsilon(s), A^{\frac{1}{2}} (\u_\epsilon(s)-\v_0)\rangle|ds \no  \\
& \quad + \frac{2m}{\epsilon} \int_{t_0}^{t_0 + \epsilon} |\delta(s)  (1 +\| \u_\epsilon(s)\|^2 )^{m-1}\langle A^{\frac{1}{2}}(\u_\epsilon(s)-\v_0), A^{\frac{1}{2}}\v_0 \rangle| ds \no \\
& \quad + \frac{2m}{\epsilon} \int_{t_0}^{t_0 + \epsilon} |\delta(s) ((1 +\| \u_\epsilon(s)\|^2 )^{m-1} - (1 +\| \v_0\|^2 )^{m-1}) \langle A\v_0, \v_0 \rangle| ds \no \\
& \quad +  \frac{2m}{\epsilon}\int_{t_0}^{t_0 + \epsilon} |(\delta(s)-\delta(t_0))  (1 +\| \v_0\|^2 )^{m-1} \langle A\v_0, \v_0 \rangle | ds \no \\
& \leq \frac{C}{\epsilon} \int_{t_0}^{t_0+\epsilon} \|\nabla \u_\epsilon(s)\| \| \nabla (\u_\epsilon(s)-\v_0)\|ds +\frac{C}{\epsilon} \int_{t_0}^{t_0 + \epsilon}\| \nabla(\u_\epsilon(s) - \v_0)\|\|\nabla\v_0\|ds \no \\
& \quad + \frac{C}{\epsilon} \int_{t_0}^{t_0 + \epsilon} \| \u_\epsilon(s) - \v_0\| \|\nabla \v_0\|^2 +g(\epsilon) \no \\
& \leq g(\epsilon),
\end{align}
where we used $a^m-b^n = (a-b)(a^{m-1} + a^{m-2}b + .. + b^{m-1} )$ to estimate the third term
Using the fact that $f \in C^3(\R)$ and $ \delta \in C^1([t_0, t_0+\epsilon]) $, we get
\begin{align}
& \Big| \frac{2m}{\epsilon} \int_{t_0}^{t_0 + \epsilon}(\delta(s) (1+ \|\nabla \varphi_{\epsilon}(s)\|^2)^{m-1} \langle B_N \varphi_\epsilon(s), A_N f(\varphi_\epsilon(s)) \rangle \no \\
& \hspace{2cm}- \delta(t_0) (1+ \|\nabla \rho_0 \|^2)^{m-1}\langle B_N \rho_0, A_N f(\rho_0) \rangle ) ds) \Big| 
\leq g(\epsilon).
\end{align} 
%
Similarly, we obtain
\begin{align}
\frac{2m}{\epsilon} \Big| \int_{t_0}^{t_0 + \epsilon}   (\delta(s) (1+\|\nabla \varphi_\epsilon(s)\|^2)^{m-1} \langle B_N \varphi_\epsilon(s), B_1(\u_\epsilon(s), \varphi_\epsilon(s))\rangle \no \\
- \delta(t_0) (1 + \|\nabla \rho_0\|^2)^{m-1} \langle B_N\rho_0, B_N(\v_0, \rho_0) \rangle )ds \Big| \leq g(\epsilon),
\end{align}
\begin{align}
\frac{2m}{\epsilon} \Big| \int_{t_0}^{t_0+\epsilon} ( \delta(s) (1+\|\u_\epsilon (s)\|^2)^{m-1} \langle \u_\epsilon(s),  B_2(B_N \varphi_\epsilon(s), \varphi_\epsilon(s)) \rangle \no  \\
- \delta(t_0) (1 + \|\v_0\|^2)^{m-1} \langle \v_0,  B_2(B_N \rho_0, \rho_0) \rangle ) ds \Big| \leq g(\epsilon), 
\end{align}
\begin{align}
\frac{1}{\epsilon} \Big| &\int_{t_0}^{t_0+\epsilon} (\delta'(s) (1 +\| \u_\epsilon(s)\|^2 )^m \|\u_\epsilon(s)\|^2 - \delta'(t_0) (1 +\| \v\|^2 )^m \|\v_0\|^2 )ds \Big|
\leq g(\epsilon),
\end{align}
\begin{align}
\frac{1}{\epsilon} \Big| &\int_{t_0}^{t_0+\epsilon} (\delta'(s)(1+ \|\nabla \varphi_{\epsilon}(s)\|^2)^m \|\nabla \varphi_\epsilon(s)\|^2 - \delta'(t_0) (1+ \|\nabla \rho_0\|^2)^m\|\nabla \rho_0\|^2)ds \Big|  
\leq g(\epsilon).
\end{align}
and 
\begin{align}
&\left| 	\frac{1}{\epsilon} \int_{t_0}^{t_0 + \epsilon} [ l(s, \varphi_{\epsilon}(s), \u_{\epsilon}, \U_\epsilon(s)) - l (t_0, \rho_0, \v_0, \U_\epsilon(s)) ]ds \right| \no \\
&\leq \frac{1}{\epsilon}  \int_{t_0}^{t_0 + \epsilon} (\sigma_R(|s-t_0|) + L_R (\|\varphi_{\epsilon}(s)-\rho_0\|_{D(B_N^{1/2})} + \| \u_{\epsilon}(s)-\v_0 \|) ) ds \leq g(\epsilon)
\end{align}
Now we deal with the higher order term. Observe that by \eqref{DP96}
\begin{align*}
\frac{1}{\epsilon} \Bigg\|\int_{t_0}^{t_0+\epsilon} \sqrt{\delta(s)}
&(1+ \|\nabla \varphi_{\epsilon}(s)\|^2)^{\frac{m-1}{2}} B_N^{3/2} \varphi_{\epsilon}(s) ds\Bigg\|^2 \\
&\leq \frac{1}{\epsilon} \int_{t_0}^{t_0+\epsilon} \delta(s) (1+ \|\nabla \varphi_{\epsilon}(s)\|^2)^{m-1}\|B_N^{3/2} \varphi_{\epsilon}(s)\|^2   ds \leq C  .
\end{align*}
Therefore, there exists  a sequence $ \epsilon_n \rightarrow 0 $ and $\rho \in H $ such that 
\begin{align*}
\Pi_n = \frac{1}{\epsilon_n} \int_{t_0}^{t_0+\epsilon_n} \sqrt{\delta(s)}(1+ \|\nabla \varphi_{\epsilon_n}(s)\|^2)^{m-1} B_N^{3/2} \varphi_{\epsilon_n}(s) ds \rightharpoonup \rho \quad \mathrm{in} \ H . 
\end{align*}
However, as $ n \rightarrow \infty $, we have
\begin{align*}
B_N^{-3/2} \Pi_n = \frac{1}{\epsilon_n} \int_{t_0}^{t_0+\epsilon_n} \sqrt{\delta(s)} \, (1+ \|\nabla \varphi_{\epsilon_n}(s)\|^2)^{m-1}\varphi_{\epsilon_n}(s) ds \rightarrow \sqrt{\delta(t_0)}\,(1+ \|\nabla \rho_0\|^2)^{m-1}  \rho_0 , 
\end{align*}
strongly in $ H $. Therefore, it follows that
\begin{align*}
\rho = \sqrt{\delta(t_0)}(1+ \|\nabla \rho_0\|^2)^{m-1} B_N^{3/2} \rho_0.
\end{align*}
Then, we get 
\begin{align}
\liminf_{n \rightarrow \infty} \frac{1}{\epsilon_n} \int_{t_0}^{t_0+\epsilon_n} \delta(s)(1+ \|\nabla \varphi_{\epsilon_n}(s)\|^2)^{m-1} \|B_N^{3/2} \varphi_{\epsilon_n}(s)\|^2 ds \geq  \delta(t_0) (1+ \|\nabla \rho_0\|^2)^{m-1}\| B_N^{3/2} \rho_0\|^2.
\end{align}
Using the same argument we can get 
\begin{align*}
\frac{1}{\epsilon_n} \int_{t_0}^{t_0+\epsilon_n} B_N^{3/2} \varphi_{\epsilon_n}(s) ds \rightharpoonup  B_N^{3/2} \rho_0 \quad \mathrm{ \ in \ } H, 
\end{align*}
as $ n \rightarrow \infty $. Using the above convergence and  \eqref{DP112} we also get
\begin{align}
& \frac{1}{\epsilon_n} \Big| \int_{t_0}^{t_0+\epsilon_n} ( \langle B_N^2\varphi_{\epsilon_n}(s), \partial_\rho \xi(s, \varphi_{\epsilon_n} (s), \u_\epsilon (s))\rangle - \langle B_N^2\rho_0 , \partial_\rho \xi (t_0, \rho_0, \v_0 )\rangle )ds \Big|\no \\
&\leq \frac{1}{\epsilon_n} \int_{t_0}^{t_0+\epsilon_n} |\langle B_N^{3/2}\varphi_{\epsilon_n}(s), B_N^{1/2}(\partial_\rho \xi(s, \varphi_{\epsilon_n} (s), \u_\epsilon (s)) - \partial_\rho \xi (t_0, \rho_0, \v_0 ) )\rangle| ds \no \\
& \quad + \frac{1}{\epsilon_n}  \int_{t_0}^{t_0+\epsilon_n} |\langle B_N^{3/2}\varphi_{\epsilon_n}(s)-B_N^{3/2}\rho_0 , \partial_\rho \xi (t_0, \rho_0, \v_0 ) \rangle |ds  \no \\
& \leq  \frac{1}{\epsilon_n} \int_{t_0}^{t_0+\epsilon_n} \|B_N^{3/2}\varphi_{\epsilon_n}(s)\| g(|s-t_0| + \|\varphi_{\epsilon_n} (s)
-\rho_0\|_{D(B_N^{1/2})} + \|\u_{\epsilon_n}(s) - \v_0 \| ) ds \no \\
& \quad + \frac{1}{\epsilon_n} | \int_{t_0}^{t_0+\epsilon_n} \langle B_N^{3/2}\varphi_{\epsilon_n}(s)-B_N^{3/2}\rho_0 , \partial_\rho \xi (t_0, \rho_0, \v_0 ) \rangle ds |  \rightarrow 0 \ \mathrm{ \ as \ } n \rightarrow \infty.\label{DP108}
\end{align}
Using \eqref{DP84}-\eqref{DP108} in \eqref{DP96} we get
\begin{align*}
&-\Psi_t (t_0, \rho_0, \v_0) + \langle B_1(\v_0, \rho_0)+B_N^2 \rho_0 +  A_N f(\rho_0), \partial_\rho \Psi(t_0, \rho_0, \v_0) \rangle \\
& +\langle A\v_0 + B(\v_0,\v_0)-B_2(B_N \rho_0, \rho_0) ,\partial_{\v} \Psi(t_0, \rho_0, \v_0) \rangle \\
&+\frac{1}{2}(\|\rho_0\|^2+ \|\v_0\|^2)+ \frac{1}{\epsilon} \int_{t_0}^{t_0+\epsilon} [\langle \U_\epsilon(s), -\partial_\v \Psi(t_0, \rho_0, \v_0)  \rangle +\|\U_\epsilon(s) \|^2] ds \leq g(\epsilon).
\end{align*}
Now taking infimum inside the integral over $ \U \in \mathcal{U}_{R} $ and letting $\epsilon \rightarrow 0$ we get \eqref{DP97}.

{\bf Step 3:} In this step we  want to show the inequality  \eqref{HJBmax1},  which will prove that  $ \mathcal{V} $ is a viscosity subsolution of \eqref{HJB1}. Let $\mathcal{V}-\Psi$ attains global maximum at $(t_0, \rho_0, \v_0)$. As in the step 1, we can show that the point of maxima lies in a better regular space. Then
\begin{align}\label{DP2}
\mathcal{V}(t, \rho, \v) - &\mathcal{V}(t_0, \rho_0, \v_0) \leq \xi (t, \rho, \v) - \xi (t_0, \rho_0, \v_0) \no \\
&+ \delta(t)[(1+\|\nabla \rho \|^2)^m +(1+ \|\v\|^2)^m] - \delta(t_0) [(1+\|\nabla \rho_0 \|^2)^m + (1+\|\v_0\|^2)^m].
\end{align}
From the dynamic programming principle Theorem \ref{DPP} we have
\begin{align} \label{DP3}
\mathcal{V} (t_0, \rho_0, \v_0) \leq \frac{1}{2} \int_{t_0}^{t_0+\epsilon} \left(\|\varphi(s)\|^2+ \|\u(s)\|^2 + \|\U(s)\|^2 \right)ds +\mathcal{V} (t_0 + \epsilon, \varphi(t_0 + \epsilon), \u(t_0 + \epsilon)) .
\end{align}
Substituting \eqref{DP3} in \eqref{DP2} we get
\begin{align*} 
&\xi(t_0, \rho_0, \v_0)-\xi(t_0+\epsilon, \varphi(t_0+\epsilon), \u (t_0 + \epsilon)) +\delta(t_0)[ (1+\|\nabla \rho_0 \|^2)^m + (1+ \|\v_0\|^2)^m] \\
&- \delta(t_0+\epsilon)[(1+\|\nabla\varphi(t_0+\epsilon) \|^2)^m +(1+ \|\u(t_0+\epsilon)\|^2)^m ]
\leq \frac{1}{2} \int_{t_0}^{t_0+\epsilon} \left(\|\varphi(s)\|^2+ \|\u(s)\|^2 + \|\U(s)\|^2 \right)ds.
\end{align*}
Using chain rule, we get for a constant control $ \U \in \mathcal{U}_{R} $,
\begin{align*}
\frac{1}{2} &\int_{t_0}^{t_0+\epsilon} l \left( \varphi(s), \u(s), \U \right)dt \\
\geq&-\int_{t_0}^{t_0 + \epsilon} \left( \partial_t \xi(s, \varphi (s), \u (s))+ \langle \partial_t\varphi(s) , \partial_\rho \xi(s, \varphi (s), \u (s))\rangle + \langle \partial_t \u(s), \partial_\v \xi(s, \varphi (s), \u (s))\rangle  \right) ds \no \\
& -\int_{t_0}^{t_0 + \epsilon} \delta'(s)(1+\|\nabla \varphi(s)\|^2)^m ds -\int_{t_0}^{t_0+\epsilon} m \delta(s) (1+\|\nabla \varphi(s)\|^2)^{m-1} \langle B_N \varphi(s), \partial_t\varphi(s) \rangle ds  \no \\
& - \int_{t_0}^{t_0 + \epsilon} \delta'(s) (1+\|\u(s)\|^2)^m ds -  \int_{t_0}^{t_0 + \epsilon} m \delta(s) (1+\|\u(s)\|^2)^{m-1} \langle \u(s), \partial_t\u(s) \rangle ds. 
\end{align*}
That is,
\begin{align*}
\frac{1}{2} &\int_{t_0}^{t_0+\epsilon} l \left( \varphi(s), \u(s), \U \right)ds \\
\geq  &-\int_{t_0}^{t_0 + \epsilon}  \partial_t \xi(s, \varphi (s), \u (s))ds\\
& +\frac{1}{\epsilon} \int_{t_0}^{t_0+\epsilon}  \left\langle B_1(\u(s), \varphi(s))+B_N^2\varphi(s)+A_Nf(\varphi(s)), \,  \partial_\rho \xi(s, \varphi (s), \u (s)) \right\rangle ds\\
& + \frac{1}{\epsilon} \int_{t_0}^{t_0+\epsilon} \langle A\u(s) + B(\u(s),\u(s)) - B_2(B_N \varphi(s), \varphi(s))-\U(s), \, \partial_\v \xi(s, \varphi (s), \u (s)) \rangle ds \\
&- \frac{1}{\epsilon}\int_{t_0}^{t_0 + \epsilon} \delta'(s)(1+\|\nabla \varphi(s)\|^2)^m ds-  \frac{1}{\epsilon}\int_{t_0}^{t_0 + \epsilon} \delta'(s) (1+\|\u(s)\|^2)^m ds \\
&+ \frac{2m}{\epsilon}\int_{t_0}^{t_0+\epsilon} \delta(s) (1+\|\nabla \varphi(s)\|^2)^{m-1} \langle B_N \varphi(s), \,  B_1(\u(s), \varphi(s))+B_N^2\varphi(s)+A_Nf(\varphi(s))\rangle ds\\
&+  \frac{2m}{\epsilon} \int_{t_0}^{t_0 + \epsilon}  \delta(s) (1+\|\u(s)\|^2)^{m-1} \langle \u(s), \, A\u (s) + B(\u(s),\u(s)) - B_2(B_N \varphi(s), \varphi(s))-\U(s) \rangle ds.
\end{align*}
Now using similar techniques used in supersolution part (Step 2) and using continuous dependence estimates of trajectories \eqref{DP127} and \eqref{DP78} we deduce
\begin{align*}
&\partial_t \Psi (t_0, \rho_0, \v_0)-\langle B_1(\v_0, \rho_0)+B_N^2 \rho_0 +  A_N f(\rho_0), \partial_\rho \Psi(t_0, \rho_0, \v_0) \rangle \no \\
&-\langle A\v_0 + B(\v_0,\v_0)-B_2(B_N \rho_0, \rho_0) ,\partial_{\v} \Psi(t_0, \rho_0, \v_0) \rangle \\
&+ H (t_0,\rho_0, \v_0,\partial_\rho \Psi(t_0, \rho_0, \v_0), \partial_\v \Psi(t_0, \rho_0, \v_0) ) \geq 0, \
\end{align*}
which gives that $  \mathcal{V} $ is a viscosity subsolution. Hence, $ \mathcal{V} $ is a viscosity solution of \eqref{HJB1} in the sense of Definition \ref{def_visco}.
\end{proof}

\section{Uniqueness}

In this section, we prove the comparison principle, which in turn implies the uniqueness of the viscosity solution. Viscosity solution theory intrinsically provides the well-posedness or uniqueness of solutions for non-linear Hamilton Jacobi type of equations satisfied in the viscosity sense and we exploit this property.

%
%
We will assume the following hypothesis on the Hamiltonian $ H $.
\begin{assumption} \label{DP133}
$ H : [0,T] \times D(B^{1/2}_N) \times \G_{{\mathrm{div}}} \times D(B_N^{1/2}) \times \G_{{\mathrm{div}}} \rightarrow \mathbb{R} $, and there exists a modulus of continuity $ \omega $ and $ \omega_K $ such that 
\begin{align} \label{DP139}
|H(t, \rho_1, \v_1, p, \q) - H(t, \rho_2, \v_2, p, \q) |\leq  & \,\omega_K ( \|\nabla (\rho_1-\rho_2) \| )+ \omega_K( \|\v_1 - \v_2 \|  ) \no \\
&+ \omega ( \| \nabla(\rho_1 -\rho_2) \| \|\nabla p \| )+ \omega (\|\v_1 - \v_2\|  \| \q \|) ,
\end{align}
if $\| \rho_1\|_{D(B^{1/2}_N)}  ,\|\rho_2 \|_{D(B^{1/2}_N)} , \|\v_1\|, \|\v_2\|\leq K $,
\begin{align}
|	H(t, \rho, \v, p_1, \q_1) - H(t, \rho, \v, p_2, \q_2) |\leq & \,\omega ( (1+\|\nabla \rho\|)\|\nabla (p_1-p_2) \| ) + \omega ( (1+\|\v\|)\|\q_1-\q_2\| ) ,
\end{align}
\begin{align}
| H(t, \rho, \v, p, \q) - H(s, \rho, \v, p, \q) |  \leq &\,  \omega_K (t-s), \label{DP140}
\end{align}
if $\| \rho_1\|_{D(B^{1/2}_N)}  ,\|\rho_2 \|_{D(B^{1/2}_N)} , \|\v_1\| , \|\v_2\| , \|p\|_{D(B_N^{1/2})}, \|\q\| \leq K $.
\end{assumption}
%
Note that the Assumption \ref{DP133} is compatible with the Assumption \ref{l_assumption} on $l$ and $g$. The Next Theorem proves the comparison principle under the Assumption \ref{DP133}.

\begin{theorem} \label{comp_ppl}
Let Assumption \ref{DP133} be satisfied. Let $\mathcal{V}_1, \mathcal{V}_2 : (0, T) \times D(B^{1/2}_N) \times \G_{\mathrm{div}} \rightarrow \R$ be a viscosity subsolution and a viscosity supersolution, respectively, of \eqref{HJB1}. Let
\begin{align*}
\mathcal{V}_1(t, \rho, \v) \leq C(1 + \|\rho\|_{D(B_N^{1/2})}^k + \|\v\|^k),\\
-\mathcal{V}_2(t, \rho, \v) \leq C(1 + \|\rho\|_{D(B_N^{1/2})}^k + \|\v\|^k),
\end{align*}
for some $k>0$ and
\begin{equation} \label{terminal_condition}
\begin{aligned}
\lim_{t \rightarrow T} (\mathcal{V}_1 (t, \rho, \v) -g(\rho, \v))^+ = 0,\\
\lim_{t \rightarrow T} (\mathcal{V}_2 (t, \rho, \v) -g(\rho, \v))^- = 0,
\end{aligned}
\end{equation}
uniformly on bounded subsets of $D(B^{1/2}_N) \times \G_{\mathrm{div}}$. Then $\mathcal{V}_1 \leq \mathcal{V}_2$.
\end{theorem}
The proof of this theorem is industrious hence we divide it  into 3 steps. We employ ideas from  {\cite{MR2141895}} to prove the Theorem. The main idea of the proof is to use a contradiction. We want to prove the comparison of a subsolution and a supersolution;  we assume on the contrary that at a certain point, the comparison is reversed and using the definitions of test functions and the properties of Hamiltonian arrive at a contradiction.
%

\begin{proof}
{\bf Step 1:} 
We begin by defining  $ \mathcal{V}_{1}^\beta $ and $ \mathcal{V}_{2}^\beta  $ and by doubling the variables we define the function $\Phi$ which suitably attains global minima and maxima and also helps in defining the appropriate test functions.

For a given $\beta >0$, we define 
\begin{align*}
\mathcal{V}_{1}^\beta  (t, \rho, \v) = \mathcal{V}_1(t,\rho,\v) - \frac{\beta }{t}, \quad \mathcal{V}_{2}^\beta (t, \rho, \v) = \mathcal{V}_2(t,\rho,\v) + \frac{\beta}{t}.
\end{align*}
Then, we see that $\mathcal{V}_{1}^\beta$ and $\mathcal{V}_{2}^\beta$ respectively satisfy
\begin{align}
&\partial_t \mathcal{V}_{1}^\beta - \langle   B_1(\v, \rho) + A_N^2 \rho + A_N f(\rho) ,\partial_\rho \mathcal{V}_{1}^\beta\rangle - \langle A \v + B(\v, \v) - B_2(A_N \rho, \rho ),\partial_\v \mathcal{V}_{1}^\beta \rangle \no \\
&+ H (t,\rho, \v,\partial_\rho \mathcal{V}_1^\beta(t, \rho, \v), \partial_\v \mathcal{V}_1^\beta (t, \rho, \v) )\geq \frac{\beta}{T^2},
\end{align}
and
\begin{align}
&\partial_t \mathcal{V}_{2}^\beta - \langle   B_1(\v, \rho) + A_N^2 \rho + A_N f(\rho) ,\partial_\rho \mathcal{V}_{2}^\beta \rangle - \langle  A \v + B(\v, \v) - B_2(A_N \rho, \rho ),\partial_\v \mathcal{V}_{2}^\beta \rangle \no \\
&+ H (t,\rho, \v,\partial_\rho \mathcal{V}_2^\beta(t, \rho, \v), \partial_\v \mathcal{V}_2^\beta (t, \rho, \v) ) \leq- \frac{\beta}{T^2}.
\end{align}
Let $m>1$ be such that $2m \geq k+1$. For $\epsilon, \delta, \gamma >0$ and $0 < T_\delta <T$, we consider the function 
\begin{align*}
\Phi ( t,s, \rho, \tilde{\rho}, \v, \tilde{\v}) = &\mathcal{V}_1^\beta (t, \rho,\v) - \mathcal{V}_2^\beta(s, \tilde{\rho}, \tilde{\v}) -\frac{1}{2\epsilon} (\|\rho-\tilde{\rho} \|^2 +\|\v-\tilde{\v}\|_{\V'_{\mathrm{div}}}^2) \no \\
&-\delta e^{K_\beta (T-t)}[(1+ \|\nabla  \rho\|^2)^m + (1+\|\v\|^2)^m] \no \\
&-\delta e^{K_\beta (T-s)}[(1+\|\nabla \tilde{\rho}\|^2)^m + (1+\|\tilde{\v}\|^2)^m] - \frac{(t-s)^2}{2 \gamma},
\end{align*}
on $ (0, T_\delta ]  \times D(B_N^{1/2}) \times \G_{\mathrm{div}}$. The choice of $ K_\beta $ will be made precise later on. 
We can prove that $\Phi$ is weakly sequentially upper semicontinuous on $(0,T_\delta] \times D(B_N^{1/2}) \times \G_{\mathrm{div}}$ in the same lines of proof of Theorem 5.2 in \cite{MR2141895}).
Hence, it has a global maximum. Suppose that $ \Phi $ has a global maximum at $(t_0,s_0, \rho_0, \tilde{\rho}_0, \v_0, \tilde{\v}_0)$ where $ 0 < t_0, s_0 $ and  for a fixed $ \delta, \;  \rho_0, \tilde{\rho_0} \in D(B^{1/2}_N) $, $ \v_0 , \tilde{\v}_0  \in \G_{{\mathrm{div}}}$ are bounded independent of $ \epsilon $. Moreover, as in the proof of step 1 \ref{HJBThm} we can show that $ \rho_0, \tilde{\rho_0} \in D(B_N^{3/2}) $, $ \v_0 , \tilde{\v}_0  \in \V_{{\mathrm{div}}}$.
We also observe that
\begin{align}
\limsup_{\epsilon \rightarrow 0} \limsup_{\gamma \rightarrow 0} \frac{1}{2 \epsilon}(\|\rho_0-\tilde{\rho}_0 \|^2 +\|\v_0-\tilde{\v}_0\|_{\V'_{\mathrm{div}}}^2) = 0 \ \mathrm{for \ fixed \ \delta, \lambda},\label{epsilon1} 
\end{align} 
and 
\begin{align}
\limsup_{\gamma \rightarrow 0}  \frac{(t_0-s_0)^2}{2 \gamma} = 0 \ \mathrm{for \ fixed \ \delta, \epsilon, \lambda}. \label{gamma1}
\end{align}
 If $ \mathcal{V}_1 $ is not less than or equal to $ \mathcal{V}_2 $, then   for small $ \beta $, $ \delta $ and $ T_\delta $ sufficiently close to $ T$, using the assumption \eqref{terminal_condition}, and \eqref{epsilon1}, \eqref{gamma1} 
we can conclude that $ t_0, s_0 < T_\delta $, if $ \epsilon $ and $ \gamma $ are small enough.

Now, we use the fact that $ \mathcal{V}_1^\beta $ and $ \mathcal{V}_2^\beta $ are viscosity subsolution and supersolution, respectively. 
Let us define
\begin{align*}
\Psi_1(t,\rho,\v) =& \mathcal{V}^\beta _2(s_0, \tilde{\rho}_0, \tilde{\v}_0)+\frac{1}{2\epsilon} (\|\rho-\tilde{\rho}_0 \|^2+\|\v-\tilde{\v}_0\|_{\V'_{\mathrm{div}}}^2) \\
&+\delta e^{K_\beta(T-t)} [(1+\|\nabla \rho\|^2)^m+(1+\|\v\|^2)^m] \no \\
&+ \delta e^{K_\beta (T-s_0)} [(1+\|\nabla \tilde{\rho}_0\|^2)^m   + (1+\|\tilde{\v}_0 \|^2)^m] + \frac{(t-s_0)^2}{2\gamma}.
\end{align*}

It follows from the definition of the viscosity subsolution, 
since $ \mathcal{V}^\beta_1 - \Psi_1 $ has maximum at $(t_0, \rho_0, \v_0)  $, we have 
\begin{align}
&\frac{t_0-s_0}{\gamma} -\delta K_\beta e^{K_\beta (T-t_0)}[(1+\|\nabla \rho_0 \|^2)^m + (1+\|\v_0\|^2 )^m] - \langle  B_1(\v_0, \rho_0)+B_N^2 \rho_0 +  A_N f(\rho_0), \frac{1}{\epsilon} (\rho_0-\tilde{\rho}_0) \rangle \no \\
&- \langle  B_1(\v_0,\rho_0 ) + B_N^2 \rho_0 +  A_N f(\rho_0), 2m\delta e^{K_\beta (T-t_0)} (1+\|\nabla \rho_0 \|^2  )^{m-1}B_N \rho_0 \rangle \no \\
& - \langle A\v_0 + B(\v_0,\v_0)-B_2(B_N \rho_0, \rho_0) , \frac{1}{\epsilon}A^{-1}(\v_0-\tilde{\v}_0)\rangle \no \\
&- \langle  A\v_0 + B(\v_0,\v_0)-B_2(B_N \rho_0, \rho_0) ,  2m\delta e^{K_\beta (T-t_0)}(1 + \|\v_0\|^2 )^{m-1}\v_0 \rangle \no \\
&+ H\left(t_0, \rho_0, \v_0, \frac{1}{\epsilon}( \rho_0 - \tilde{\rho}_0) + 2m\delta (1+\|\nabla \rho_0 \|^2 )^{m-1} e^{K_\beta (T-t_0)}B_N \rho_0 , \right. \no \\
&\hspace{2cm}\left. \frac{1}{\epsilon} A^{-1}(\v_0-\tilde{\v}_0) + 2m\delta (1+\|\v_0\|^2 )^{m-1} e^{K_\beta (T-t_0)}\v_0\right) \geq \frac{\beta}{T^2}. \label{DP34}
\end{align}
Similarly, define
\begin{align*}
\Psi_2(s,\tilde{\rho},\tilde{\v}) = & \mathcal{V}^\beta _1(t_0,\rho_0, \v_0)-\frac{1}{2\epsilon}( \|\rho_0 -\tilde{\rho} \|^2+\|\v_0-\tilde{\v}\|_{\V'_{\mathrm{div}}}^2) \\
&-\delta e^{K_\beta(T-t_0)}[(1+\|\nabla \rho_0\|^2)^m + (1+ \|\v_0\|^2)^m] \no \\
&- \delta e^{K_\beta(T-s)}[(1+\|\nabla \tilde{\rho}\|^2)^m+ (1+\|\tilde{\v}\|^2)^m] - \frac{(t_0-s)^2}{2 \gamma}.
\end{align*}

By the definition of viscosity supersolution, since $ \mathcal{V}^\beta_2 + \Psi_2 $ has minimum at $ (s_0, \tilde{\rho}_0, \tilde{\v}_0) $, we get
\begin{align}
&\frac{t_0-s_0}{\gamma} +\delta K_\beta e^{K_\beta(T-s_0)}[(1+\|\nabla \tilde{\rho}_0\|^2)^m+ (1+\|\tilde{\v}_0\|^2 )^m] - \langle B_1(\tilde{\v}_0, \tilde{\rho}_0)+B_N^2 \tilde{\rho}_0 +  A_N f(\tilde{\rho}_0),\frac{1}{\epsilon} (\rho_0-\tilde{\rho}_0) \rangle \no \\
&+ 2m\delta e^{K_\beta(T-s_0)} \langle B_1(\tilde{\v}_0, \tilde{\rho}_0) + B_N^2 \tilde{\rho}_0 +  A_N f(\tilde{\rho}_0), (1+\|\nabla \tilde{\rho}_0\|^2 )^{m-1}  B_N \tilde{\rho}_0 \rangle \no \\
& -\langle A\tilde{\v}_0 + B(\tilde{\v}_0,\tilde{\v}_0)-B_2(B_N \tilde{\rho}_0, \tilde{\rho}_0) , \frac{1}{\epsilon}A^{-1}(\v_0-\tilde{\v}_0) \rangle \no \\
& +2m\delta e^{K_\beta (T-s_0)} \langle A\tilde{\v}_0 + B(\tilde{\v}_0,\tilde{\v}_0)-B_2(B_N \tilde{\rho}_0, \tilde{\rho}_0),(1+\|\tilde{\v}_0\|^2 )^{m-1} \tilde{\v}_0 \rangle \no \\ 
&+ H\left(s_0, \tilde{\rho}_0, \tilde{\v}_0, \frac{1}{\epsilon}( \rho_0 - \tilde{\rho}_0) - 2m\delta e^{K_\beta (T-s_0)}(1+\|\nabla \tilde{\rho}_0\|^2)^{m-1}B_N \tilde{\rho}_0 , \right. \no \\
& \hspace{2cm} \left.\frac{1}{\epsilon}A^{-1}(\v_0-\tilde{\v}_0) - 2m\delta e^{K_\beta(T-s_0)}(1+\|\tilde{\v}_0\|^2 )^{m-1}\tilde{\v}_0 \right)  \leq -\frac{\beta}{T^2}. \label{DP35}
\end{align}

Combining \eqref{DP34} and \eqref{DP35} we get 
\begin{align}
&\delta K_\beta e^{K_\beta (T-t_0)}[(1+\|\nabla \rho_0 \|^2)^m + (1+\|\v_0\|^2 )^m ]+ \delta K_\beta e^{K_\beta(T-s_0)}[(1+\|\nabla \tilde{\rho}_0\|^2)^m+(1+\|\tilde{\v}_0\|^2 )^m] \no \\
&+  \langle B_1(\v_0, \rho_0)-B_1(\tilde{\v}_0, \tilde{\rho}_0) +  A_N f(\rho_0)- A_N f(\tilde{\rho}_0), \frac{1}{\epsilon} ( \rho_0-\tilde{\rho}_0) \rangle  + \frac{1}{\epsilon}\|B_N (\rho_0 - \tilde{\rho}_0)\|^2\no \\
& + 2m\delta e^{K_\beta (T-t_0)}(1+\|\nabla \rho_0 \|^2 )^{m-1}[  \langle B_1(\v_0,\rho_0 )  +  A_N f(\rho_0), B_N \rho_0 \rangle + \|B_N^{3/2} \rho_0\|^2]\no \\
& +2m\delta e^{K_\beta(T-s_0)} (1+\|\nabla \tilde{\rho}_0\|^2)^{m-1} [\langle  B_1(\tilde{\v}_0, \tilde{\rho}_0)  +  A_N f(\tilde{\rho}_0),  B_N \tilde{\rho}_0 \rangle + \|B_N^{3/2} \tilde{\rho}_0\|^2]\no \\
& +\frac{1}{\epsilon} \|\v_0 -\tilde{\v}_0\|^2 + \langle B(\v_0,\v_0)-B(\tilde{\v}_0,\tilde{\v}_0) + B_2(B_N \tilde{\rho}_0, \tilde{\rho}_0)-B_2(B_N \rho_0, \rho_0) , \frac{1}{\epsilon}A^{-1}(\v_0-\tilde{\v}_0) \rangle \no \\
&+ 2m\delta e^{K_\beta (T-t_0)} (1 + \|\v_0\|^2 )^{m-1}[\|\nabla\v_0\|^2 -\langle B_2(B_N \rho_0, \rho_0) , \v_0 \rangle ]\no \\
& +2m\delta e^{K_\beta (T-s_0)}(1+\|\tilde{\v}_0\|^2 )^{m-1} [\|\nabla \tilde{\v}_0 \|^2-\langle B_2(B_N \tilde{\rho}_0, \tilde{\rho}_0),\tilde{\v}_0\rangle ]\no \\ 
& + H\left(s_0, \tilde{\rho}_0, \tilde{\v}_0, \frac{1}{\epsilon}( \rho_0 - \tilde{\rho}_0) - 2m\delta e^{K_\beta (T-s_0)}(1+\|\nabla \tilde{\rho}_0\|^2)^{m-1}B_N \tilde{\rho}_0 ,  \right. \no \\
& \hspace{2cm} \left. \frac{1}{\epsilon}A^{-1}(\v_0-\tilde{\v}_0) - 2m\delta e^{K_\beta(T-s_0)}(1+\|\tilde{\v}_0\|^2 )^{m-1}\tilde{\v}_0 \right) \no \\
& -H\left(t_0, \rho_0, \v_0, \frac{1}{\epsilon}( \rho_0 - \tilde{\rho}_0) + 2m\delta e^{K_\beta (T-t_0)}(1+\|\nabla \rho_0 \|^2  )^{m-1} B_N \rho_0 , \right. \no \\
&\hspace{2cm} \left. \frac{1}{\epsilon} A^{-1}(\v_0-\tilde{\v}_0) + 2m\delta e^{K_\beta (T-t_0)}(1+ \|\v_0\|^2 )^{m-1} \v_0\right) \leq - \frac{2\beta }{T^2} \label{DP109}. 
\end{align}
 To arrive at a contradiction, we need to show that LHS in the above inequality goes to 0 for a fixed $\beta >0$.

{\bf Step 2:}  In this step we  estimate the terms in the left-hand side of \eqref{DP109}. Now onwards, let us denote $\sigma$ to be a local modulus of continuity and $C$  a generic constant. First, we write
\begin{align*}
\frac{1}{\epsilon} \langle B_1(\v_0, \rho_0)-B_1(\tilde{\v}_0, \tilde{\rho}_0),\rho_0-\tilde{\rho}_0 \rangle =\frac{1}{\epsilon}\langle B_1(\v_0-\tilde{\v}_0, \rho_0)-B_1(\tilde{\v}_0, \rho_0-\tilde{\rho}_0),\rho_0-\tilde{\rho}_0 \rangle .
\end{align*}
Observe that $ \langle B_1(\tilde{\v}_0, \rho_0-\tilde{\rho}_0),\rho_0-\tilde{\rho}_0 \rangle =0$. Using Agmon's inequality 
\begin{align}
\frac{1}{\epsilon} |\langle B_1(\v_0-\tilde{\v}_0, \rho_0),\rho_0-\tilde{\rho}_0 \rangle | &\leq \frac{C}{\epsilon} \|\v_0-\tilde{\v}_0\| \| \nabla \rho_0 \| \| \rho_0 -\tilde{\rho}_0\|_{L^\infty} \no \\
& \leq  \frac{C}{\epsilon} \|\v_0-\tilde{\v}_0\| \| \nabla \rho_0 \| \| \rho_0 -\tilde{\rho}_0\|^{1/2} \|B_N(\rho_0 -\tilde{\rho}_0) \|^{1/2} \no \\
& \leq \frac{1}{5\epsilon}\| B_N(\rho_0 -\tilde{\rho}_0) \|^2 + \frac{C}{\epsilon} \|\v_0-\tilde{\v}_0\|^{4/3} \| \rho_0 -\tilde{\rho}_0\|^{2/3}\no \\
& \leq \frac{1}{5\epsilon}\| B_N(\rho_0 -\tilde{\rho}_0) \|^2 + \frac{1}{4\epsilon} \|\v_0-\tilde{\v}_0\|^2 + \sigma(\epsilon).\label{DP83}
\end{align}
Similarly,
\begin{align*}
b_2(B_N \rho_0,\rho_0 &,A^{-1}(\v_0-\tilde{\v}_0) )-b_2( B_N \tilde{\rho}_0, \tilde{\rho}_0,A^{-1}(\v_0-\tilde{\v}_0) ) \\
&= b_2(B_N(\rho_0-\tilde{\rho}_0),\rho_0,A^{-1}(\v_0-\tilde{\v}_0)) + b_2(B_N\tilde{\rho}_0,\rho_0-\tilde{\rho}_0,A^{-1}(\v_0-\tilde{\v}_0)).
\end{align*}
Using Sobolev inequality and \eqref{DP48}, we estimate
\begin{align*}
\frac{1}{\epsilon}|b_2(B_N(\rho_0-\tilde{\rho}_0),\rho_0 &,A^{-1}(\v_0-\tilde{\v}_0))| \no \\
&\leq \frac{c}{\epsilon} \|B_N (\rho_0-\tilde{\rho}_0 )\| \| \nabla \rho_0 \|_{L^4} \| A^{-1}(\v_0-\tilde{\v}_0)\|_{L^4} \no \\
&\leq \frac{c}{\epsilon} \|B_N (\rho_0-\tilde{\rho}_0 )\| \|\nabla \rho_0\|^\frac{1}{2} \|B_N \rho_0\|^\frac{1}{2} \|\v_0 - \tilde{\v_0} \|_{\V'_{\mathrm{div}}} \no  \\
&\leq \frac{1}{5\epsilon} \|B_N (\rho_0-\tilde{\rho}_0) \|^2 + C \| \nabla \rho_0 \| \|B_N^{3/2} \rho_o\| \frac{\|\v_0-\tilde{\v}_0\|_{\V'_{\mathrm{div}}}^2}{\epsilon} \no \\
&\leq \frac{1}{5\epsilon} \|B_N (\rho_0-\tilde{\rho}_0) \|^2 + \delta \|B_N^{3/2} \rho_o\|^2 + \sigma(\epsilon; \delta) \no \\
&\leq \frac{1}{5\epsilon} \|B_N (\rho_0-\tilde{\rho}_0) \|^2 + \frac{m\delta}{4} e^{K_\beta (T-t_0)} (1 + \|\nabla \rho_0\|^2)^{m-1}\|B_N^{3/2} \rho_0\|^2 + \sigma(\epsilon; \delta) ,
\end{align*}
Now, observe that using integration by parts we get that
\begin{align*}
    \frac{1}{\epsilon} |b_2(B_N\tilde{\rho}_0,\rho_0-\tilde{\rho}_0,&A^{-1}(\v_0-\tilde{\v}_0))|  \no\\ 
    &\leq \frac{1}{\epsilon} |( (\rho_0 - \tilde{\rho_0}) B_N^{3/2} \tilde{\rho_0} , A^{-1}(\v_0-\tilde{\v}_0) )| + \frac{1}{\epsilon} | ( (\rho_0 - \tilde{\rho_0}) B_N \tilde{\rho_0} , A^{-1/2}(\v_0-\tilde{\v}_0) ) |,
\end{align*}
and we estimate right-hand side terms of the above as follows:
\begin{align*}
    \frac{1}{\epsilon} |( (\rho_0 - \tilde{\rho}_0) &B_N^{3/2} \tilde{\rho}_0 , A^{-1}(\v_0-\tilde{\v}_0) )| \no \\
    &\leq \frac{1}{\epsilon} \|\rho_0 - \tilde{\rho}_0 \|_{L^4} \|B_N^{3/2} \tilde{\rho}_0 \|  \|A^{-1}(\v_0-\tilde{\v}_0) \|_{L^4} \no \\
    &\leq \delta \|B_N^{3/2} \tilde{\rho_0} \|^2 + \frac{1}{4\delta \epsilon} \|\rho_0 - \tilde{\rho}_0 \| \| \nabla(\rho_0 - \tilde{\rho}_0) \| \frac{\|\v_0 - \tilde{\v}_0 \|^2_{\V_{div}}}{\epsilon} \no \\
    &\leq \delta \|B_N^{3/2} \tilde{\rho}_0 \|^2 + \frac{1}{8\delta} \left(\frac{\| \rho_0 - \tilde{\rho}_0\|^2}{\epsilon} + \frac{\|B_N (\rho_0 - \tilde{\rho}_0)\|^2}{\epsilon} \right) \frac{\|\v_0 - \tilde{\v}_0 \|^2_{\V_{div}}}{\epsilon} \no \\
    &\leq \frac{m\delta}{3} e^{K_\beta (T-s_0)} (1+\|\nabla \tilde{\rho}_0\|^2)^{m-1} \|B_N^{3/2} \tilde{\rho}_0 \|^2 + \sigma(\epsilon; \delta) + \sigma_2(\epsilon; \delta) \frac{\|B_N (\rho_0 - \tilde{\rho}_0)\|^2}{\epsilon},
\end{align*}
\begin{align*}
    \frac{1}{\epsilon} | (  (\rho_0 - \tilde{\rho_0}) B_N \tilde{\rho}_0 , A^{-1/2}(\v_0-\tilde{\v}_0) ) | 
    &\leq \frac{1}{\epsilon} \|\rho_0 - \tilde{\rho}_0\| \| B_N \tilde{\rho}_0\|_{L^4} \| A^{-1/2}(\v_0-\tilde{\v}_0)\|_{L^4} \no \\
    &\leq \frac{1}{\epsilon} \|\rho_0 - \tilde{\rho}_0\| \| B_N^{3/2}  \tilde{\rho}_0\| \|\v_0 - \tilde{\v}_0\|_{\V_{\mathrm{div}}}^\frac{1}{2} \| \v_0-\tilde{\v}_0\|^\frac{1}{2} \no \\
    &\leq \delta \| B_N^{3/2}  \tilde{\rho}_0\|^2 +  \frac{1}{4 \delta \epsilon^2} \|\rho_0 - \tilde{\rho}_0\|^2 \|\v_0 - \tilde{\v}_0\|_{\V_{\mathrm{div}}} \| \v_0-\tilde{\v}_0\| \no \\
    &\leq \delta \| B_N^{3/2}  \tilde{\rho}_0\|^2 + \frac{1}{8\delta} \frac{\|\rho_0 - \tilde{\rho}_0\|^2}{\epsilon} \left( \frac{\|\v_0 - \tilde{\v}_0\|_{\V_{\mathrm{div}}}^2}{\epsilon} + \frac{\| \v_0-\tilde{\v}_0\|^2}{\epsilon} \right) \no \\
    &\leq \frac{m\delta}{3} e^{K_\beta (T-s_0)} (1+\|\nabla \tilde{\rho}_0\|^2)^{m-1} \| B_N^{3/2}  \tilde{\rho}_0\|^2 + \sigma(\epsilon; \delta) + \frac{\| \v_0-\tilde{\v}_0\|^2}{\epsilon} \sigma_1(\epsilon;\delta),
\end{align*}
Hence, we have 
\begin{align}
\frac{1}{\epsilon} &|b_2(B_N\tilde{\rho}_0,\rho_0-\tilde{\rho}_0,A^{-1}(\v_0-\tilde{\v}_0))| \no \\
&\leq \frac{2m\delta}{3} e^{K_\beta (T-s_0)} (1+\|\nabla \tilde{\rho}_0\|^2)^{m-1} \| B_N^{3/2}  \tilde{\rho}_0\|^2 + \sigma(\epsilon; \delta) + \frac{\| \v_0-\tilde{\v}_0\|^2}{\epsilon} \sigma_1(\epsilon;\delta) + \sigma_2(\epsilon; \delta) \frac{\|B_N (\rho_0 - \tilde{\rho}_0)\|^2}{\epsilon},
\end{align}
Similarly, we also have
\begin{align*}
b(\v_0, \v_0, A^{-1}(\v_0-\tilde{\v}_0))-&b(\tilde{\v}_0,\tilde{\v}_0, A^{-1}(\v_0-\tilde{\v}_0)) \\
&=b(\v_0 - \tilde{\v}_0,\v_0, A^{-1}(\v_0-\tilde{\v}_0)) + b(\tilde{\v}_0,\v_0-\tilde{\v}_0,A^{-1}(\v_0-\tilde{\v}_0) ).
\end{align*}
Using H\"older's inequality and \eqref{DP48} we get
\begin{align}
\frac{1}{\epsilon} |b(\v_0 - \tilde{\v}_0,\v_0, A^{-1}(\v_0-\tilde{\v}_0))| &=\frac{1}{\epsilon} |b(\v_0 - \tilde{\v}_0, A^{-1}(\v_0-\tilde{\v}_0),\v_0)| \no \\
&\leq \frac{1}{\epsilon} \|\v_0 - \tilde{\v}_0\|  \| A^{-1/2}(\v_0-\tilde{\v}_0)\|^{1/2} \|\|\v_0-\tilde{\v}_0 \|^{1/2} \| \v_0\|^{1/2} \|\nabla \v_0\|^{1/2} \no \\
&\leq \frac{1}{4\epsilon} \|\v_0 - \tilde{\v}_0\| ^2 + \frac{C}{\epsilon} \| A^{-1/2}(\v_0-\tilde{\v}_0)\| \|\|\v_0-\tilde{\v}_0 \| \| \v_0\| \|\nabla \v_0\|\no \\
&\leq  \frac{1}{4\epsilon} \|\v_0 - \tilde{\v}_0\| ^2  + \delta \|\nabla \v_0\|^2 + C_\delta \frac{\| \v_0-\tilde{\v}_0\|_{\V'_{\mathrm{div}}}^2}{\epsilon} \frac{\|\v_0 - \tilde{\v}_0\|^2}{\epsilon} \no\\
&\leq \frac{1}{4\epsilon} \|\v_0 - \tilde{\v}_0\| ^2  + m\delta e^{K_\beta (T-t_0)} (1 + \|\v_0\|^2 )^{m-1}  \|\nabla \v_0\|^2 +  \sigma_1(\epsilon; \delta)\frac{\|\v_0 - \tilde{\v}_0\|^2}{\epsilon} \label{DP55}
\end{align}
Similarly,
\begin{align}
\frac{1}{\epsilon} |b(\tilde{\v}_0,\v_0-\tilde{\v}_0,A^{-1}(\v_0-\tilde{\v}_0))| &=\frac{1}{\epsilon} |b(\tilde{\v}_0,A^{-1}(\v_0-\tilde{\v}_0),A^{-1}(\v_0-\tilde{\v}_0))| \no \\
&\leq  \frac{1}{4\epsilon} \|\v_0 - \tilde{\v}_0\| ^2  + m\delta e^{K_\beta (T-t_0)} (1 + \|\tilde{\v}_0\|^2 )^{m-1}  \|\nabla \tilde{\v}_0\|^2 +  \sigma_1(\epsilon; \delta)\frac{\|\v_0 - \tilde{\v}_0\|^2}{\epsilon}
\end{align}
The terms involving $f $ can be estimated using \eqref{f condition2} as follows, 
For $0\leq \theta \leq 1$, using Gagliardo-Nirenberg inequality we have
\begin{align}
    \frac{1}{\epsilon}| \langle A_N f(\rho_0)- A_N f(\tilde{\rho}_0), \rho_0- \tilde{\rho}_0 \rangle | &=\frac{1}{\epsilon} | \langle  f(\rho_0)-f(\tilde{\rho}_0), B_N(\rho_0- \tilde{\rho}_0)\rangle | \no  \\
    &\leq \frac{C}{\epsilon} \|f'(\rho_0 + \theta \tilde{\rho}_0) \| \|\rho_0- \tilde{\rho}_0\|_{L^\infty} \|B_N(\rho_0- \tilde{\rho}_0) \| \no \\
    &\leq \frac{C}{\epsilon} (1+ \|\rho_0 + \theta \tilde{\rho}_0)\|_{L^{2r}}^r ) \|\rho_0- \tilde{\rho}_0\|^\frac{1}{2} \|B_N(\rho_0- \tilde{\rho}_0) \|^\frac{3}{2}  \no \\
    &\leq \frac{1}{5 \epsilon} \|B_N (\rho_0 - \tilde{\rho}_0) \|^2 + C(1 + \|\rho_0\|\|\nabla \rho_0\|^{(r-1)} + \|\tilde{\rho}_0\| \|\nabla \tilde{\rho}_0\|^{(r-1)} ) \frac{\|\rho_0 - \tilde{\rho}_0\|^2}{\epsilon} \no \\
    &\leq  \frac{1}{5 \epsilon} \|B_N (\rho_0 - \tilde{\rho}_0) \|^2 + \sigma(\epsilon, \delta),
\end{align}
where we used the triangle inequality and Lemma \ref{GNI} with $\theta=1-\frac{1}{r} $ in the second last inequality. Moreover,
\begin{align}
    2 \delta e^{K_\beta(T-s_0)} |\langle A_N f(\tilde{\rho}_0), B_N \tilde{\rho}_0\rangle | 
    &= 2 \delta e^{K_\beta(T-s_0)} |\langle  f'(\tilde{\rho_0}) \nabla \tilde{\rho}_0, B_N^{3/2} \tilde{\rho}_0 \rangle | \no \\
    & \leq  2 \delta e^{K_\beta(T-s_0)} C  |( (1+ |\tilde{\rho}_0|^r)  \nabla \tilde{\rho}_0 , B_N^{3/2} \tilde{\rho}_0 ) | \no \\
    & \leq 2 \delta e^{K_\beta(T-s_0)} C  (1+\| \tilde{\rho}_0\|_{L^{4r}}^r  ) \|\nabla \tilde{\rho}_0\|_{L^4} \|B_N^{3/2} \tilde{\rho}_0\| \no \\
    &\leq 2 \delta e^{K_\beta(T-s_0)} (1+\| \tilde{\rho}_0\|_{L^{4r}}^r  ) \|\nabla \tilde{\rho}_0\|^\frac{1}{2} \|B_N^{3/2} \tilde{\rho}_0\|^\frac{3}{2} \no \\
    &\leq \frac{m\delta}{3} e^{K_\beta (T-s_0)} (1+\|\nabla \tilde{\rho}_0\|^2)^{m-1}  \|B_N^{3/2} \tilde{\rho}_0\|^2 + \sigma(\delta)
\end{align}
and similarly,
\begin{align}
2 \delta e^{K_\beta(T-t_0)} |\langle A_N f(\rho_0), B_N \rho_0 \rangle | 
& \leq \frac{m\delta}{4} e^{K_\beta(T-t_0)} (1+\|\nabla \rho_0\|^2)^{m-1}\|B_N^{3/2} \rho_0\| ^2 + \sigma(\delta).  
\end{align}
It is also easy to observe the following estimate:
\begin{align}
    2m\delta &e^{K_\beta (T-t_0)}(1+\|\nabla \rho_0 \|^2 )^{m-1}  |\langle B_1(\v_0,\rho_0 ), B_N \rho_0 \rangle| \no \\ 
    &\leq 2Cm\delta e^{K_\beta (T-t_0)} \|\v_0\|_{L^4} \|\nabla \rho_0\| \|B_N \rho_0\|_{L^4} \no \\
    &\leq \frac{m\delta}{4} e^{K_\beta (T-s_0)} (1+\|\nabla \rho_0\|^2)^{m-1} |B_N^{3/2} \rho_0\|^2 + \frac{m\delta}{3} e^{K_\beta (T-s_0)} (1+\v_0\|^2)^{m-1} |\nabla \v_0\|^2 + \sigma(\delta),
\end{align}
and similarly,
\begin{align}
     2m\delta &e^{K_\beta (T-t_0)}(1+\|\nabla \tilde{\rho}_0 \|^2 )^{m-1}  |\langle B_1(\tilde{\v}_0,\tilde{\rho}_0 ), B_N \tilde{\rho}_0 \rangle| \no \\
     &\leq  \frac{m\delta}{3} e^{K_\beta (T-s_0)} (1+\|\nabla \tilde{\rho}_0\|^2)^{m-1} |B_N^{3/2} \tilde{\rho}_0\|^2 + \frac{m\delta}{3} e^{K_\beta (T-s_0)} (1+\tilde{\v}_0\|^2)^{m-1} |\nabla \tilde{\v}_0\|^2 + \sigma(\delta),
\end{align}
\begin{align}
     2m&\delta e^{K_\beta (T-t_0)} (1 + \|\v_0\|^2 )^{m-1} |\langle B_2(B_N \rho_0, \rho_0) , \v_0 \rangle | \no \\
  &\leq \frac{m\delta}{4} e^{K_\beta (T-s_0)} (1+\|\nabla \rho_0\|^2)^{m-1} |B_N^{3/2} \rho_0\|^2 + \frac{m\delta}{3} e^{K_\beta (T-s_0)} (1+\v_0\|^2)^{m-1} |\nabla \v_0\|^2 + \sigma(\delta),
  \end{align}
\begin{align}\label{DP56}
    2m&\delta e^{K_\beta (T-s_0)}(1+\|\tilde{\v}_0\|^2 )^{m-1} |\langle B_2(B_N \tilde{\rho}_0, \tilde{\rho}_0),\tilde{\v}_0\rangle | \no \\
    &\leq  \frac{m\delta}{3} e^{K_\beta (T-s_0)} (1+\|\nabla \tilde{\rho}_0\|^2)^{m-1} |B_N^{3/2} \tilde{\rho}_0\|^2 + \frac{m\delta}{3} e^{K_\beta (T-s_0)} (1+\tilde{\v}_0\|^2)^{m-1} |\nabla \tilde{\v}_0\|^2 + \sigma(\delta).
\end{align}
Substituting \eqref{DP83}-\eqref{DP56} in \eqref{DP109}, we get
\begin{align}
&\delta K_\beta e^{K_\beta (T-t_0)}[(1+\|\nabla \rho_0 \|^2)^m + (1+\|\v_0\|^2 )^m ]+ \delta K_\beta e^{K_\beta(T-s_0)}[(1+\|\nabla \tilde{\rho}_0\|^2)^m+(1+\|\tilde{\v}_0\|^2 )^m] \no \\
& + m\delta e^{K_\beta (T-t_0)}[(1+\|\nabla \rho_0 \|^2 )^{m-1} \|B_N^{3/2} \rho_0\|^2 + (1 + \|\v_0\|^2 )^{m-1}\|\nabla\v_0\|^2  ]  + \frac{1}{5\epsilon} \|B_N(\rho_0-\tilde{\rho}_0) \|^2\no  \\
&+ m\delta e^{K_\beta(T-s_0)} [(1+\|\nabla \tilde{\rho}_0\|^2)^{m-1} \|B_N^{3/2} \tilde{\rho}_0\|^2 
+(1+\|\tilde{\v}_0\|^2 )^{m-1} \|\nabla \tilde{\v}_0 \|^2] + \frac{1}{4\epsilon} \|\v_0 - \tilde{\v}_0\|^2 \no \\
& + H(s_0, \tilde{\rho}_0, \tilde{\v}_0, \frac{1}{\epsilon}( \rho_0 - \tilde{\rho}_0) - 2m\delta e^{K_\beta (T-s_0)}(1+\|\nabla \tilde{\rho}_0\|^2)^{m-1}B_N \tilde{\rho}_0 , \no \\
& \hspace{2cm} \frac{1}{\epsilon}A^{-1}(\v_0-\tilde{\v}_0) - 2m\delta e^{K_\beta(T-s_0)}(1+\|\tilde{\v}_0\|^2 )^{m-1}\tilde{\v}_0) \no \\
& -H(t_0, \rho_0, \v_0, \frac{1}{\epsilon}( \rho_0 - \tilde{\rho}_0) + 2m\delta e^{K_\beta (T-t_0)}(1+\|\nabla \rho_0 \|^2  )^{m-1} B_N \rho_0 , \no \\
&\hspace{2cm} \frac{1}{\epsilon} A^{-1}(\v_0-\tilde{\v}_0) + 2m\delta e^{K_\beta (T-t_0)}(1+ \|\v_0\|^2 )^{m-1} \v_0) \no \\
&\leq - \frac{2\beta }{T^2} + \sigma(\delta, \epsilon) + \sigma_1(\epsilon; \delta)\frac{\|\v_0 - \tilde{\v}_0\|^2}{\epsilon} + \sigma_2(\epsilon;\delta) \frac{\|B_N(\rho-\tilde{\rho})\|^2}{\epsilon}. \label{DP134}
\end{align}
{\bf{Step 3:}}  In this step we use assumptions on $H$ in  (\ref{DP134}), appropriately choose the modulus continuities, and show how one can arrive at a contradiction.
Let $C_\beta$ be a constant such that 
\begin{align}  \label{5.24}
\omega(s) \leq \frac{\beta}{8T^2} + C_\beta \, s.
\end{align} 
Using the assumptions \eqref{DP133}- \eqref{DP140} on Hamiltonian and \eqref{5.24}, we get
\begin{align*}
& \left| H\left(s_0, \tilde{\rho}_0, \tilde{\v}_0, \frac{1}{\epsilon}( \rho_0 - \tilde{\rho}_0) - 2m\delta e^{K_\beta (T-s_0)}(1+\|\nabla \tilde{\rho}_0\|^2)^{m-1}B_N \tilde{\rho}_0 ,  \right. \right. \no \\
& \hspace{0.2cm} \left. \left. \frac{1}{\epsilon}A^{-1}(\v_0-\tilde{\v}_0) - 2m\delta e^{K_\beta(T-s_0)}(1+\|\tilde{\v}_0\|^2 )^{m-1}\tilde{\v}_0\right) - H\left(s_0, \tilde{\rho}_0, \tilde{\v}_0,\frac{1}{\epsilon}( \rho_0 - \tilde{\rho}_0),\frac{1}{\epsilon}A^{-1}(\v_0-\tilde{\v}_0) \right) \right| \no \\
& + \left| H\left(t_0, \rho_0, \v_0, \frac{1}{\epsilon}( \rho_0 - \tilde{\rho}_0) + 2m\delta e^{K_\beta (T-t_0)}(1+\|\nabla \rho_0 \|^2  )^{m-1} B_N \rho_0 , \right. \right. \no \\
&\hspace{.2cm} \left. \left. \frac{1}{\epsilon} A^{-1}(\v_0-\tilde{\v}_0) + 2m\delta e^{K_\beta (T-t_0)}(1+ \|\v_0\|^2 )^{m-1} \v_0\right) -H\left(t_0, \rho_0, \v_0, \frac{1}{\epsilon}( \rho_0 - \tilde{\rho}_0) ,\frac{1}{\epsilon} A^{-1}(\v_0-\tilde{\v}_0) \right) \right| \\
& \leq \omega ((1+\|\nabla \tilde{\rho}_0\|) 2m\delta e^{K_\beta (T-s_0)}(1+\|\nabla \tilde{\rho}_0\|^2)^{m-1}\|B_N^{3/2} \tilde{\rho}_0 \|) \no \\
&  \hspace{1cm} +\omega ((1+ \|  \tilde{\v}_0\|) 2m\delta e^{K_\beta(T-s_0)}(1+\|\tilde{\v}_0\|^2 )^{m-1} \| \tilde{\v}_0 \|) \no \\
& \hspace{1cm} + \omega ( (1+\|\nabla \rho_0\|) 2m\delta e^{K_\beta (T-t_0)}(1+\|\nabla \rho_0 \|^2  )^{m-1} \|B_N^{3/2} \rho_0 \| ) \no \\
& \hspace{1cm} + \omega ( (1+\|\v_0\|)2m\delta e^{K_\beta (T-t_0)}(1+ \|\v_0\|^2 )^{m-1} \|\v_0\|  ) \no \\
&\leq \frac{\beta}{2T^2} + C_\beta(1+\|\nabla \tilde{\rho}_0\|) 2m\delta e^{K_\beta (T-s_0)}(1+\|\nabla \tilde{\rho}_0\|^2)^{m-1}\|B_N^{3/2} \tilde{\rho}_0 \| \no \\
& \hspace{1cm}+ C_\beta (1+ \|\tilde{\v}_0\|) 2m\delta e^{K_\beta(T-s_0)}(1+\|\tilde{\v}_0\|^2 )^{m-1} \|\tilde{\v}_0 \| \no \\
& \hspace{1cm}+ C_\beta (1+\|\nabla \rho_0\|) 2m\delta e^{K_\beta (T-t_0)}(1+\|\nabla \rho_0 \|^2  )^{m-1} \|B_N^{3/2} \rho_0 \| \no \\
& \hspace{1cm}+  C_\beta (1+\|\v_0\|)2m\delta e^{K_\beta (T-t_0)}(1+ \|\v_0\|^2 )^{m-1} \|\v_0\| \no \\
& \leq \frac{\beta}{2T^2} + 2m\delta e^{K_\beta (T-s_0)} C_\beta ^2 (1+\|\nabla \tilde{\rho}_0\|^2)^m + m\delta e^{K_\beta (T-s_0)}  (1+\|\nabla \tilde{\rho}_0\|^2)^{m-1} \|B_N^{3/2} \tilde{\rho}_0 \|^2 \no \\
& \hspace{1cm} + 2m\delta e^{K_\beta(T-s_0)} C_\beta^2 (1+ \|\tilde{\v}_0\|^2)^m + m\delta e^{K_\beta(T-s_0)} (1+ \|\tilde{\v}_0\|^2)^{m-1} \|\nabla \tilde{\v}_0\|^2 \no \\
& \hspace{1cm} + 2m\delta e^{K_\beta(T-t_0)} C_\beta^2 (1+\|\nabla \rho_0\|^2)^m  + m\delta e^{K_\beta(T-t_0)} (1+\|\nabla \rho_0\|^2)^{m-1} \|B_N^{3/2} \rho_0 \|^2 \no \\
& \hspace{1cm} + 2m\delta e^{K_\beta(T-t_0)} C_\beta^2 (1+ \|\v_0\|^2)^m + m\delta e^{K_\beta(T-t_0)} (1+ \|\v_0\|^2)^{m-1} \|\nabla \v_0\|^2.
\end{align*}
where we used Poincare's inequality for $\v_0, \tilde{\v}_0$ in the last inequality.   Choosing $K_\beta = 1+4 m C_\beta^2$, and Using above estimate in \eqref{DP134} we have
\begin{align}
&\frac{\delta}{2} K_\beta e^{K_\beta (T-t_0)} [ (1+\|\nabla \rho_0\|^2)^{m}+(1+ \|\v_0\|^2)^m] + \frac{\delta}{2} K_\beta e^{K_\beta(T-s_0)} [ (1+\|\nabla \tilde{\rho}_0\|^2)^{m}+(1+ \|\tilde{\v}_0\|^2)^m] \no \\
& +m\delta e^{K_\beta (T-t_0)}[(1+\|\nabla \rho_0 \|^2 )^{m-1} \|B_N^{3/2} \rho_0\|^2 ]  + \frac{1}{5\epsilon} \|B_N(\rho_0-\tilde{\rho}_0) \|^2\no  \\
&+ m\delta e^{K_\beta(T-s_0)} [(1+\|\nabla \tilde{\rho}_0\|^2)^{m-1} \|B_N^{3/2} \tilde{\rho}_0\|^2 
] + \frac{1}{4\epsilon} \|\v_0 - \tilde{\v}_0\|^2 \no \\
& + H\left( s_0, \tilde{\rho}_0, \tilde{\v}_0, \frac{1}{\epsilon}( \rho_0 - \tilde{\rho}_0) , \frac{1}{\epsilon}A^{-1}(\v_0-\tilde{\v}_0) \right)-H\left(t_0, \rho_0, \v_0, \frac{1}{\epsilon}( \rho_0 - \tilde{\rho}_0) , \frac{1}{\epsilon} A^{-1}(\v_0-\tilde{\v}_0) \right)\no  \\ 
&\leq - \frac{\beta }{T^2} + \sigma(\delta, \epsilon) + \sigma_1(\epsilon; \delta)\frac{\|\v_0 - \tilde{\v}_0\|^2}{\epsilon} + \sigma_2(\epsilon;\delta) \frac{\|B_N(\rho-\tilde{\rho})\|^2}{\epsilon}. \label{DP136}
\end{align}
For fixed $ \beta  $ and $ \delta $ we know that there exists $ R_\delta>0 $ such that $ \|  \rho_0\|_{D(B^{1/2}_N)}, \|  \tilde{\rho} \|_{D(B^{1/2}_N)}, \| \v_0 \|, \|\tilde{\v}\| \leq R_\delta $.
Let $ D_{\beta, \delta} $ be such that 
\begin{align}
\omega_{R_\delta} (s) \leq \frac{\beta}{4T^2} + D_{\beta, \delta} \, s .
\end{align}
Then we have
\begin{align}
&\left| H\left(s_0, \tilde{\rho}_0, \tilde{\v}_0, \frac{1}{\epsilon}( \rho_0 - \tilde{\rho}_0) , \frac{1}{\epsilon}A^{-1}(\v_0-\tilde{\v}_0) \right)-H\left(t_0, \rho_0, \v_0, \frac{1}{\epsilon}( \rho_0 - \tilde{\rho}_0) , \frac{1}{\epsilon} A^{-1}(\v_0-\tilde{\v}_0) \right) \right|\no \\
& \leq \omega_{R_\delta} (s_0 - t_0) + \omega_{R_\delta} (\|\nabla (\rho_0 - \tilde{\rho}_0) \| + \| \v_0-\tilde{\v}_0 \|) \no \\
& \hspace{1cm}+ \omega \left(\frac{1}{\epsilon}\|\nabla (\rho_0 - \tilde{\rho}_0) \| \,\|\nabla (\rho_0 - \tilde{\rho}_0) \| + \frac{1}{\epsilon}\| \v_0-\tilde{\v}_0 \|\|A^{-1}(\v_0-\tilde{\v}_0) \| \right)\no \\
& \leq  \frac{\beta}{4T^2}+\sigma (\gamma) + \frac{\beta}{4T^2} + D_{\beta, \delta} (\|\nabla (\rho_0 - \tilde{\rho}_0) \| + \| \v_0-\tilde{\v}_0 \| ) \no \\
& \hspace{1.4cm}+ \frac{\beta}{4T^2}  + C_\beta \left(\frac{1}{\epsilon}\|\nabla (\rho_0 - \tilde{\rho}_0) \|^2 + \frac{1}{\epsilon} \|\v_0 - \tilde{\v}_0\| \|A^{-1}(\v_0 - \tilde{\v}_0)\| \right)   \no \\
& \leq \sigma (\gamma) + \frac{3\beta}{4T^2} + D_{\beta, \delta} (\|\nabla (\rho_0 - \tilde{\rho}_0) \| + \| \v_0-\tilde{\v}_0 \| ) + \frac{C_\beta}{\epsilon} \|\nabla (\rho_0 - \tilde{\rho}_0) \|^2 + \frac{C_\beta^2}{\epsilon} \|\v_0-\tilde{\v}_0  \|^2 _{\V_{\mathrm{div}}'} + \frac{1}{\epsilon} \| \v_0-\tilde{\v}_0 \|^2 \no \\
& \leq \sigma (\gamma, \epsilon) + \frac{3\beta}{4T^2} + D_{\beta, \delta} (\|\nabla (\rho_0 - \tilde{\rho}_0) \| + \| \v_0-\tilde{\v}_0 \| + \frac{C_{\beta}}{\epsilon} \|\nabla (\rho_0 - \tilde{\rho}_0) \|^2)+ \frac{1}{\epsilon} \| \v_0-\tilde{\v}_0 \|^2 \no \\
& \leq \sigma (\gamma, \epsilon) + \frac{3\beta}{4T^2} + D_{\beta, \delta} (\|B_N (\rho_0 - \tilde{\rho}_0) \| + \| \v_0-\tilde{\v}_0 \|) + \frac{C_{\beta}}{\epsilon} \|B_N (\rho_0 - \tilde{\rho}_0) \|^2+ \frac{1}{\epsilon} \| \v_0-\tilde{\v}_0 \|^2 \label{DP135} 
\end{align}
where we used Poincare's inequality in the last step above.
Substituting \eqref{DP135} in \eqref{DP136}, we get
\begin{align}
\left(\frac{1}{5}-C_\beta -\sigma_2(\epsilon;\delta)\right) \frac{\| B_N (\rho_0 - \tilde{\rho}_0) \|^2}{\epsilon} - D_{\beta, \delta} \| B_N (\rho_0 - \tilde{\rho}_0) \| + \left(\frac{1}{4} - \sigma_1(\epsilon;\delta ) \right)\frac{\| \v_0-\tilde{\v}_0 \|^2}{\epsilon} -D_{\beta, \delta}\| \v_0-\tilde{\v}_0 \| \no \\
\leq - \frac{\beta}{4T^2} + \sigma(\gamma, \epsilon, \delta  ). \label{DP137}
\end{align}
Choose $C_\beta$ such that $\frac{1}{5}-C_\beta -\sigma_2(\epsilon;\delta) > 0$ to observe that
\begin{align*}
\lim_{\epsilon \rightarrow 0} \inf_{r>0} \left( \frac{(\frac{1}{5}-C_\beta - \sigma_2(\epsilon; \delta))r^2}{\epsilon} - D_{\beta, \delta} \, r \right) = 0 
\end{align*}
and choosing $\frac{1}{4} - \sigma_1(\epsilon;\delta) > \frac{1}{8}$ we have
\begin{align*}
\lim_{\epsilon \rightarrow 0} \inf_{r>0} \left( \frac{r^2}{8\epsilon} - D_{\beta, \delta}\, r \right) = 0
\end{align*}
If we send $\gamma \rightarrow 0, \epsilon \rightarrow 0$ and $\delta \rightarrow 0$ in \eqref{DP137}, LHS  goes to zero which gives a contradiction. 
\end{proof}

\subsection*{Acknowledgements} 
The authors would like to thank Dr. Manil T Mohan from IIT Roorkee, INDIA, for useful discussions. The authors would also like to thank Prof. Andrzej Swiech from Georgia Institute of Technology, USA, and the anonymous referee for their valuable comments and suggestions.

\bibliographystyle{plain}
\bibliography{ref_DPP}

\end{document}